\documentclass{article}

\usepackage{amsmath}
\usepackage{mathtools}
\usepackage{amssymb}
\usepackage{amsthm}
\usepackage{amscd}
\usepackage[alphabetic,backrefs,lite]{amsrefs}
\usepackage{mathrsfs}
\usepackage[all,cmtip]{xy}
\usepackage{enumerate}
\usepackage{comment}
\usepackage{cleveref}

\newcommand{\A}{{\mathbb A}}
\newcommand{\C}{{\mathbb C}}
\newcommand{\F}{{\mathbb F}}
\newcommand{\G}{{\mathbb G}}

\newcommand{\Q}{{\mathbb Q}}
\newcommand{\R}{{\mathbb R}}
\newcommand{\DeligneS}{{\mathbb S}}

\newcommand{\Z}{{\mathbb Z}}
\newcommand{\Qbar}{{\overline{\Q}}}

\newcommand{\kbar}{{\overline{k}}}

\newcommand{\ksep}{{k^{\operatorname{sep}}}}

\newcommand{\red}{{\operatorname{red}}}

\newcommand{\opp}{{\operatorname{opp}}}
\newcommand{\nr}{{\operatorname{nr}}}

\newcommand{\injects}{\hookrightarrow}
\newcommand{\isom}{\simeq}

\newcommand{\plim}[1][]{\mathop{\varprojlim}\limits_{#1}}

\DeclareMathOperator{\tr}{tr} 
\DeclareMathOperator{\Spec}{Spec}

\DeclareMathOperator{\ch}{ch}
\DeclareMathOperator{\lcm}{lcm}
\DeclareMathOperator{\Hom}{Hom}
\DeclareMathOperator{\Isom}{Isom}
\DeclareMathOperator{\Gal}{Gal}
\DeclareMathOperator{\End}{End}
\DeclareMathOperator{\Aut}{Aut}
\DeclareMathOperator{\GL}{GL}
\DeclareMathOperator{\SL}{SL}
\DeclareMathOperator{\Lie}{Lie}
\DeclareMathOperator{\Nrd}{Nrd}
\DeclareMathOperator{\Trd}{Trd}
\DeclareMathOperator{\Frob}{Frob}

\newtheorem{theorem}{Theorem}[section]
\newtheorem{proposition}[theorem]{Proposition}
\newtheorem{lemma}[theorem]{Lemma}
\newtheorem{corollary}[theorem]{Corollary}

\theoremstyle{definition}
\newtheorem{definition}[theorem]{Definition}
\newtheorem{example}[theorem]{Example}
\newtheorem*{acknowledgements}{Acknowledgements}

\theoremstyle{remark}

\newtheorem{remark}[theorem]{Remark}

\crefname{theorem}{Theorem}{Theorems}
\crefname{proposition}{Proposition}{Propositions}
\crefname{lemma}{Lemma}{Lemmata}
\crefname{corollary}{Corollary}{Corollaries}
\crefname{conjecture}{Conjecture}{Conjectures}
\crefname{definition}{Definition}{Definitions}
\crefname{example}{Example}{Examples}
\crefname{remark}{Remark}{Remarks}
\crefname{section}{Section}{Sections}
\crefname{equation}{the equation}{the equations}
\Crefname{equation}{The equation}{The equations}

\begin{document}

\title{Rational points on Shimura varieties classifying abelian varieties with quaternionic multiplication} 
\author{Koji Matsuda\thanks{The University of Tokyo}}
\date{}

\maketitle

\begin{abstract}
We study the characters induced by suitable level structures of abelian varieties with quaternionic multiplication following the methods of Mazur, Momose, who studied the characters induced by elliptic curves, and Arai--Momose, who studied the characters induced by abelian surfaces with quaternionic multiplication.
Also using results on such characters we show that Shimura varieties parametrizing them
have rarely rational points.
\end{abstract}

\section{Introduction}

In 1977 Mazur \cite{MazurX1} showed that for a sufficiently large positive integer $N$ the modular curve $X_1(N)$, which classifies elliptic curves with level-$\Gamma_1(N)$ structures, has no non-cuspidal $\Q$-rational points.
The corresponding results for larger number fields are studied by many authors, for example Kenku--Momose \cite{KenkuMomose} and Kamienny \cite{Kamienny} showed it for quadratic number fields, and Merel \cite{Merel} showed it for number fields of a given arbitrary degree.
These are proven by studying the Mordell--Weil groups of the modular Jacobian varieties $J_0(N)$ over the rationals.

On the other hand, in 1978 Mazur \cite{MazurX0} studied a different generalization of his result of \cite{MazurX1}.
Namely, he showed that for a sufficiently large prime $p$, the modular curve $X_0(p)$, which classifies the elliptic curves with their level-$\Gamma_0(p)$ structures, has no non-cuspidal $\Q$-rational points.
It is also generalized to larger number fields.
For example Momose \cite{Momose} showed it for quadratic number fields which are not imaginary quadratic number fields of class number one.
These results are proven by a similar method of \cite{MazurX1} and by studying the characters obtained from elliptic curves and their level-$\Gamma_0(p)$ structures.

The modular curves can be considered as the Shimura varieties associated with the algebraic group $\GL_2 \Q$, which are one of the most classical and simple Shimura varieties.
Replacing the group $\GL_2 \Q$ by the algebraic group induced by the unit group of an indefinite quaternion algebra over $\Q$, which is a twist of $\GL_2 \Q$, we obtain certain types of Shimura varieties, which are called the Shimura curves.
Mazur's theorem about the rational points of $X_0(p)$ is generalized to these Shimura curves.
More precisely, let $B$ be an indefinite division quaternion algebra of discriminant $\Delta$ over the rationals, $p$ a prime not dividing $\Delta$, and $M_0^B(p)$ be the Shimura curve of level-$\Gamma_0(p)$ associated with $B$, which is the coarse moduli scheme classifying the QM-abelian surfaces with their level-$\Gamma_0(p)$ structures.
(For precise definitions see below.)
Then in 2014 Arai--Momose \cite{AraiI} showed, for quadratic number fields $k$ which are not imaginary quadratic number fields of class number one and for a sufficiently large prime $p$, that $M^B_0(p)$ has no non-CM rational points over $k$.
Arai \cite{AraiII,AraiIII} generalized it for more general $k$.
These are proven by studying the characters induced from QM-abelian surfaces with their level-$\Gamma_0(p)$ structures, which is similar to the methods of \cite{MazurX0,Momose}.

In this paper we generalize it in a different direction.
Namely, let $F$ be a totally real number field of degree $d$, $B$ a totally indefinite division quaternion algebra over $F$, and for a prime $p_F$ of $F$ not dividing the discriminant of $B/\Q$, let $M^B_i(p_F)$ ($i = 0$ or $1$) be the Shimura variety associated with $B$ of level-$\Gamma_i(p_F)$, which is a normal projective variety of dimension $d$ over $\Q$.
This is the coarse moduli scheme classifying $*$-polarized QM-abelian varieties with their level-$\Gamma_i(p_F)$ structures.
We show, for a certain type of a number field $k$ and for many primes $p_F$, that $M^B_i(p_F)$ has no rational points over $k$ following the methods of \cite{MazurX0,Momose,AraiI}.

First by studying local behavior of the Galois characters induced by QM-abelian varieties with their level-$\Gamma_0(p_F)$ structures, we show a theorem about the non-existence of rational points of the Shimura varieties of level-$\Gamma_1(p_F)$.
This seems to be analogous to showing the non-existence of certain kinds of modular curves over local fields by studying the rational points of elliptic curves over finite fields.

\begin{theorem}[\cref{M_1^B(p)(k)_is_empty}] \label{intro:M_1^B(p)(k)_is_empty}
Let $k$ be a finite extension of $\Q_\ell$ of inertia degree $f$.
Then there exists a finite set $V(\ell^f)$ of primes of $F$ depending only on $\ell$, $f$, and on $F$, but not on $B$, such that for every $p_F \not\in V(\ell^f) \cup \operatorname{Ram}(B/F)$ we have $M_1^B(p_F)(k) = \varnothing$.
\end{theorem}

This is essentially equivalent to the following theorem:

\begin{theorem}[\cref{uniform_torsion_points_theorem}] \label{intro:uniform_torsion_points_theorem}
Let $\ell$ be a rational prime and $f$ a positive integer.
Then there exists an integer $N$ depending only on $\ell$, $f$, and on $F$, but not on $B$, such that for every QM-abelian variety $A$ over a local field $k$ whose residue characteristic is $\ell$ and whose absolute inertia degree is $f$, if $A$ has a $k$-rational point of prime order $p$, then $p$ divides $N \Delta$.
\end{theorem}

This theorem is analogous to Merel's theorem on the torsion points of elliptic curves \cite[Th\'eor\`eme]{Merel}.
While Merel's theorem requires a very deep theory of modular curves, our theorem does not require geometry nor arithmetic of Shimura varieties itself, but is derived from standard arguments on QM-abelian varieties.
This is due to the fact that our Shimura varieties have no cusps, i.e., QM-abelian varieties always have potentially good reduction, whereas we need to deal with the cusps in the case of modular curves.

By studying more about the character over a number field, we also show theorems about the non-existence of rational points of the Shimura varieties of level-$\Gamma_0(p_F)$,
following the methods of \cite{MazurX0,Momose,AraiI}.
Let $k$ be a Galois number field containing $F$ and assume that $F$ is Galois.
The results for $d = 1$ is due to \cite{AraiII,AraiIII}.

\begin{theorem}[\cref{M_0^B(p)(k)_is_empty}] \label{intro:M_0^B(p)(k)_is_empty}
Assume the following four conditions:
\begin{enumerate}
\item The class number $h_k$ of $k$ is relatively prime to $d$.

\item The field $k$ does not contain the Hilbert class fields of any imaginary quadratic number fields.

\item There exists a rational prime $\ell$ not dividing $n_{\lcm} \Delta$ whose inertia degree in $k$ is odd such that
the field $F(\sqrt{-\ell})$ does not split $B$.

\item Either $d=1$, or $d = 2$ and the exponent of $\Gal(k/\Q)$ is $2$.
\end{enumerate}
Then there exists a finite set $N_3(k)$ of primes of $F$, which also depends on B, such that
for every $p_F \not\in N_3(k)$ which splits totally over $\Q$, we have that $M^B_0(p_F)(k) = \varnothing$.
\end{theorem}

\begin{theorem}[\cref{M_0^B(p)(k)_is_empty:variant}] \label{intro:M_0^B(p)(k)_is_empty:variant}
Assume the following two conditions:
\begin{enumerate}
\item There exists a rational prime $\ell_1$ not dividing $n_{\lcm} h_k$ which splits in $k$ totally such that $k$ does not contain any elements of the finite set $FR(\ell_1)$.

\item There exists a rational prime $\ell_2$ not dividing $n_{\lcm} \Delta$ whose inertia degree in $k$ is odd such that the field $F(\sqrt{-\ell_2})$ does not split $B$.
\end{enumerate}
Then there exists a finite set $N_4(k)$ of primes of $F$, which also depends on B, such that
for every $p_F \not\in N_4(k)$ which splits totally over $\Q$, we have that $M^B_0(p_F)(k) = \varnothing$.
\end{theorem}

For precise definitions of notations see below.

Note that by \cite[Theorem 0]{Shimura}, the varieties $M_0^B(p_F)$ have no real points,
and hence the theorems are trivial for real number fields.
(We explain it in \cref{section:quaternionic_Shimura_varieties}.)

In \cref{section:QM-abelian_varieties:lemmata} we define QM-abelian varieties and show fundamental lemmata about the conditions defining them.
In \cref{section:quaternionic_Shimura_varieties}, we recall the basic facts about Shimura varieties, and define the moduli stacks classifying QM-abelian varieties, which induce certain types of Shimura varieties as their coarse moduli spaces.
In \cref{section:QM-abelian_varieties:general_theory}, we discuss some standard properties about QM-abelian varieties and Shimura varieties classifying them.
In \cref{section:local}, we define the character induced from a $*$-polarized QM-abelian variety and its level-$\Gamma_0(p)$ structure, and study its local behavior.
At last of this section using its local information we deduce \cref{intro:M_1^B(p)(k)_is_empty}.
In \cref{section:global}, piecing together its local information, we study global behavior of the character.
Finally using it we deduce \cref{intro:M_0^B(p)(k)_is_empty,intro:M_0^B(p)(k)_is_empty:variant}.

\textit{Notation.}
Unless otherwise stated, we use the following notations:
For a number field or a local field $K$ we denote the ring of integers of $K$ by $O_K$.
For a prime ideal (or for a nonarchimedean valuation) $\mathfrak{p}$ of $K$,
the symbol $\F_\mathfrak{p}$ denotes the residue field of $K$ at $\mathfrak{p}$.
For a local field $K$ let $\F_K$ be the residue field.
For an extension of global fields $L/K$ we denote the discriminant and the different by $\mathfrak{d}_{L/K}$ and $\mathfrak{D}_{L/K}$ respectively, and for a place (or for a prime) $v$ of $K$, we denote $L \otimes_K K_v$ by $L_v$.
$F$ is a totally real number field of degree $d$,
$B$ is a totally indefinite division quaternion algebra over $F$,
$\Nrd$ and $\Trd$ is the reduced norm and the reduced trace of $B/F$ respectively,
$b \mapsto b^*$ is a positive involution on $B$,
i.e., an involution on $B$ satisfying for all $b \in B$
that $\tr_{B/\Q} (bb^*) > 0$.
In this case by the theorem of Albert (for example see \cite[Chapter IV, Section 21, Theorem 2]{MumfordAV}) we have that
$F$ is fixed by $*$ and that $(B, *) \otimes_\Q \R \isom (M_2(\R), t)^d$ as $\R$-algebras with involutions, where $t$ is the transpose of matrices.
We fix it throughout this paper.
$O_B$ is a maximal order of $B$ which is invariant under the involution $*$.
Note that obviously we have $O_B \cap F = O_F.$
$\operatorname{Ram}(B/F)$ is the set of ramified primes of $B/F$.
$\Delta$ is the product of the rational prime numbers $p$ at which $B$ is ramified, i.e., the rational prime numbers $p$ such that $F$ is ramified at $p$ or that $B$ is ramified at a place of $F$ above $p$.
Also $\Delta'$ is the product of the rational prime numbers $p$ satisfying that $B$ is ramified at a place of $F$ above $p$.
Throughout this paper, for each prime $\mathfrak{p}$ of $F$ which splits $B$, we fix an isomorphism from $O_B \otimes_{O_F} O_{F,\mathfrak{p}}$ to $M_2(O_{F,\mathfrak{p}})$, and identify them.

For an algebra $R$ and its subset $S$, the symbol $C_R(S)$ denotes the centralizer of $S$ in $R$.

We denote the ring of adeles (respectively the ring of finite adeles) over the rationals $\Q$
by $\A$ (respectively by $\A_f$).
For an abelian variety $A$ over a field $k$ and for a prime $\ell$, we denote $V_\ell A = T_\ell A \otimes_{\Z_\ell} \Q_\ell$,
$\hat{T} A :=  \plim A[m](\kbar) = \prod_p T_p A$, and similar for $\hat{V} A$.

All rings have an identity element, and all ring homomorphisms take the identity element to the identity element.

We denote the Deligne torus, i.e. the Weil restriction of $\G_m$ from $\C$ to $\R$, by $\DeligneS$.

\section{Definition of QM-abelian varieties} \label{section:QM-abelian_varieties:lemmata}

For an abelian scheme $A$ over a scheme $S$,
the symbol $\Lie(A/S)$ denotes the pullback of the tangent sheaf $\mathscr{T}_{A/S} = (\Omega_{A/S}^1)^\vee$ along with the $0$-section,
which is a locally free $\mathscr{O}_S$-module of rank $\dim A/S$.
On $\Lie(A/S)$ the ring $\End(A/S)$ acts from left canonically.
Let $A^\vee$ denote the dual abelian scheme of $A$ over $S$.
Note that in general the dual abelian scheme does exist as a scheme even if the abelian scheme is not projective,
see \cite[Chapter I, \S 1]{FaltingsChai}.

\begin{definition} \label{def:QMAV}
Let $S$ be a scheme.
\begin{enumerate}
\item A QM-abelian scheme by $O_B$ over $S$ is a pair $(A, i)$, where $A$ is an abelian scheme over $S$ and
$i \colon O_B \to \End(A/S)$ is a ring homomorphism,
such that $\Lie(A/S)$ is locally free of rank $2$ over $O_F \otimes_\Z \mathscr{O}_S$.

\item Let $A/S$ be an abelian scheme and $i \colon O_B \to \End(A/S)$ a ring homomorphism.
A $*$-polarization on $(A,i)$ by $O_B$ over $S$ or, to be more precise,
a $*$-polarization on $(A,i)$ of type $(O_B, b \mapsto b^*)$ over $S$ is a polarization over $S$
which commutes with the action of $O_B$, i.e., a polarization $\lambda$ of $A$ such that for every $b \in O_B$ the diagram
\begin{equation*}
\begin{aligned}
\xymatrix{
A \ar[r]^\lambda \ar[d]_{i(b^*)} & A^\vee \ar[d]^{i(b)^\vee} \\
A \ar[r]^\lambda & A^\vee
}
\end{aligned}
\end{equation*}
is commutative.

\item A $*$-polarized QM-abelian scheme by $O_B$ or, to be more precise,
a $*$-polarized QM-abelian scheme of type $(O_B, b \mapsto b^*)$ over $S$ is a pair $(A, i , \lambda)$,
where $(A,i)$ is a QM-abelian variety and $\lambda$ is a $*$-polarization on it.

\item Let $A, A'$ be abelian schemes and $i \colon O_B \to \End_S A, i' \colon O_B \to \End_S A'$ be ring homomorphisms.
Let $X = (A,i)$ and $X' = (A', i')$.
We define
\begin{equation*}
\Hom_S( X, X' ) = \{ f \in \Hom_S (A,A') | i'(b) \circ f = f \circ i(b) \text{ for every } b \in O_B \}.
\end{equation*}
We define $\End_S X, \Isom_S ( X, X' )$, and $\Aut_S X$ in the same way.
Let $\lambda$ and $\lambda'$ be polarizations on $A$ and on $A'$ respectively,
and let $Y = (A,i,\lambda)$ and $Y' = (A',i',\lambda')$.
We define
\begin{equation*}
\Isom_S( Y, Y' ) = \{ f \in \Isom_S (X,X') | \lambda = f^\vee \circ \lambda' \circ f \},
\end{equation*}
and define $\Aut_S Y$ in the same way.

\end{enumerate}
We often simply call them a QM-abelian scheme and a $*$-polarized QM-abelian scheme,
and also write $A$ instead of $(A,i)$ and $(A,i,\lambda)$ when no confusion can arise.
\end{definition}

Note that by the condition on Lie algebras, QM-abelian schemes are of dimension $2d$.
Also note that it is easy to see that if $F = \Q$ then every abelian scheme of dimension $2$ with an action of $O_B$
is a QM-abelian scheme.
Such abelian schemes are called QM-abelian surfaces.

For later use we need fundamental lemmata, especially we need \cref{lf_iff_matrices,lf_iff_reduced_norm} for later sections.
More concretely, we use \cref{lf_iff_matrices} in \cref{section:QM-abelian_varieties:general_theory}, and we use \cref{lf_iff_reduced_norm} in \cref{section:quaternionic_Shimura_varieties}.
In the rest of this section we give their proofs for the sake of completeness.

First we consider a more general setting than in the introduction.

\begin{lemma} \label{fiber_by_fiber}
Let $r$ be a nonnegative integer, $R$ a flat commutative ring over $\Z$, $S$ a scheme, $\mathscr{M}$ a locally free $\mathscr{O}_S$-module of finite rank,
and $i \colon R \to \End_{\mathscr{O}_S} (\mathscr{M})$ be a ring homomorphism.
Then $\mathscr{M}$ is locally free of rank $r$ over $R \otimes_\Z \mathscr{O}_S$ if and only if
for every point $s \in S$, the module $\mathscr{M} \otimes k(s)$ is locally free of rank $r$ over $R \otimes_\Z k(s)$,
where $k(s)$ is the residue field of $S$ at $s$.
\end{lemma}

\begin{proof}
First since $\mathscr{M}$ is a finitely presented $\mathscr{O}_S$-module and since the statement is Zariski local, we may assume that $S$ is locally noetherian.
Applying \cite[Chapter 8, Section 20, Application 3]{MatsumuraCA} to $\mathscr{O}_S \to R \otimes_\Z \mathscr{O}_S$, 
since $\mathscr{M}$ is flat over $\mathscr{O}_S$ and is also finitely generated as an $R \otimes \mathscr{O}_S$-module,
we have that the flatness of the $R \otimes_\Z \mathscr{O}_S$-module $\mathscr{M}$ is equivalent to the one of
the $R \otimes_\Z k(s)$-module $\mathscr{M} \otimes k(s)$ for every $s \in S$.
Now the finitely generation of the $R \otimes \mathscr{O}_S$-module $\mathscr{M}$ implies that its locally freeness is equivalent to the flatness.
This shows the statement.
\end{proof}

Using this lemma we can check the condition on Lie algebras as in \cref{def:QMAV} fiber-by-fiber.
Next we show some fundamental equivalent conditions related to the condition on Lie algebras as in \cref{def:QMAV}.

For a topological space $X$, for a sheaf of rings $\mathscr{A}$ on $X$,
and for a sheaf of $\mathscr{A}$-modules $\mathscr{M}$ on $X$,
let the symbol $\mathscr{E}nd_\mathscr{A} (\mathscr{M})$ denote the sheaf on $X$ which takes $U$ to
$\End_{\mathscr{A}|_U} (\mathscr{M}|_U)$.

\begin{lemma} \label{lf_iff_matrices}
Let $F$ be a number field of degree $d$, $B$ a central simple algebra of degree $r$ over $F$, $O_B$ a maximal order of $B$,
$S$ a scheme, $\mathscr{M}$ a locally free $\mathscr{O}_S$-module of finite rank, $i \colon O_B \to \End_{\mathscr{O}_S} (\mathscr{M})$ a ring homomorphism,
and  $\Delta'$ be the product of the prime numbers $p$ satisfying that there exists a place $v$ of $F$ above $p$
such that $B$ does not split at $v$.
Then $\mathscr{M}$ is locally free of rank $r$ over $O_F \otimes \mathscr{O}_S$ if
the canonical map $O_B \otimes_\Z \mathscr{O}_S \to \mathscr{E}nd_{O_F \otimes_\Z \mathscr{O}_S} \mathscr{M}$ is an isomorphism.
When $\Delta'$ is invertible on $S$, then the converse also holds.
\end{lemma}

\begin{proof}
Suppose that the canonical map
$O_B \otimes_\Z \mathscr{O}_S \to \mathscr{E}nd_{O_F \otimes_\Z \mathscr{O}_S} \mathscr{M}$ is an isomorphism.
In order to show the locally freeness of $\mathscr{M}$ as an $O_F \otimes \mathscr{O}_S$-module, by \cref{fiber_by_fiber} we may assume that $S = \Spec k$ is the spectrum of an algebraically closed field.
Let $R$ be a factor of the Artin ring $O_F \otimes_\Z k$.
Since it is a quotient ring of a discrete valuation ring, every finitely generated module over it is isomorphic to a direct sum of cyclic modules.
In particular the $R$-module $ N := \mathscr{M} \otimes_{O_F \otimes_\Z k} R$ is isomorphic to such a module.
Since for $a \in R$ the endomorphism ring $\End_R R/a$ is isomorphic to $R/a$,
and since $\End_R N$ is free over $R$ by the assumption, it follows that $N$ is free over $R$.
Again by the assumption, its rank is $r$.
Hence so is $\mathscr{M}$ over $O_F \otimes_\Z k$, which is what we want to show.

Conversely suppose that $\Delta'$ is invertible on $S$ and
that $\mathscr{M}$ is locally free of rank $r$ over $O_F \otimes \mathscr{O}_S$.
We show that the canonical map
$O_B \otimes_\Z \mathscr{O}_S \to \mathscr{E}nd_{O_F \otimes_\Z \mathscr{O}_S} \mathscr{M}$ is an isomorphism.
Since both of $O_B \otimes_\Z \mathscr{O}_S$ and $\mathscr{E}nd_{O_F \otimes_\Z \mathscr{O}_S} \mathscr{M}$ are
locally free of rank $r^2$ as $O_F \otimes \mathscr{O}_S$-modules,
it suffices to show that
the canonical map is surjective, and so by Nakayama, for a field $k$ and for a surjection $O_F \otimes_\Z \mathscr{O}_S \to k$,
it suffices to show the surjectivity of the map after tensoring with $k$ over $O_F \otimes_\Z \mathscr{O}_S$.
Since now we are assuming that the characteristic of $k$ is relatively prime to $\Delta'$,
two algebras $(O_B \otimes_\Z \mathscr{O}_S) \otimes_{O_F \otimes_\Z \mathscr{O}_S} k$ and
$(\End_{O_F \otimes_\Z \mathscr{O}_S} \mathscr{M}) \otimes_{O_F \otimes \mathscr{O}_S} k$ are simple algebras of the same rank.
Therefore the morphism
$(O_B \otimes_\Z \mathscr{O}_S) \otimes_{O_F \otimes_\Z \mathscr{O}_S} k
\to (\End_{O_F \otimes_\Z \mathscr{O}_S} \mathscr{M} ) \otimes_{O_F \otimes_\Z \mathscr{O}_S} k$
is an isomorphism.
This is what we need to show.
\end{proof}

Choose a basis $\alpha_1, \dots, \alpha_t$ of $O_B$ as a $\Z$-module
and let $\mathscr{O}_S [X_1, \dots X_t]$ be the polynomial ring of $t$ variables over $\mathscr{O}_S$.

\begin{lemma} \label{lf_iff_reduced_norm}
Keep the notations as in \cref{lf_iff_matrices}.
Let $j \colon \Z \to \mathscr{O}_S$ be the natural map.
Consider the following statements:
\begin{enumerate}[(1)]
\item The $O_F \otimes_\Z \mathscr{O}_S$-module $\mathscr{M}$ is locally free of rank $r$.\label{lf_iff_reduced_norm:con:lff}
\item The following equation holds in $\mathscr{O}_S[X_1, \dots, X_t]$:
\begin{equation*}
\det (\sum_i X_i \alpha_i \text{ on } \mathscr{M} / \mathscr{O}_S) = j( N_{O_F/\Z} \Nrd_{O_B/O_F} (\sum_i X_i \alpha_i ) ),
\end{equation*}
where $\Nrd_{O_B/O_F} = \Nrd|_{O_B}$, which takes $O_B[X_1, \dots, X_t]$ to $O_F[X_1, \dots, X_t]$.\label{lf_iff_reduced_norm:con:det}
\end{enumerate}
If $\mathfrak{d}_{F/\Q}$ is invertible on $S$, then we have $\eqref{lf_iff_reduced_norm:con:det} \Rightarrow \eqref{lf_iff_reduced_norm:con:lff}$.
Conversely, if $\Delta'$ is invertible on $S$, then we have $\eqref{lf_iff_reduced_norm:con:lff} \Rightarrow \eqref{lf_iff_reduced_norm:con:det}$.
In particular, if $\Delta$ is invertible on $S$, then these two statements are equivalent.
\end{lemma}

\begin{proof}
First suppose that $\mathfrak{d}_{F/\Q}$ is invertible on $S$ and that the equation in the statement holds.
In order to show the locally freeness of $\mathscr{M}$ as an $O_F \otimes \mathscr{O}_S$-module, by \cref{fiber_by_fiber}
we may assume that $S$ is the spectrum of an algebraically closed field $k$.
Taking the $r$-th power of the equation, we get
$\det (\sum_i X_i \alpha_i \text{ on } \mathscr{M}^r / k) = j( N_{O_F/\Z} \Nrd_{O_B/O_F} (\sum_i X_i \alpha_i )^r ) $.
The latter polynomial is, since $\Nrd^r = N_{B/F}$,
equal to $\det (\sum_i X_i \alpha_i \text{ on } O_B \otimes_\Z k / k)$.
Hence we have, for example by \cite[Proposition 1.1.2.20]{Lan},
that $\mathscr{M}^r \isom O_B \otimes_\Z k$ as $O_F \otimes_\Z k$-modules.
(By the assumption that $\mathfrak{d}_{F/\Q}$ is invertible on $k$, the algebra $O_F \otimes k$ satisfies the assumption of \cite[Proposition 1.1.2.20]{Lan}.)
In particular, the $O_F \otimes_\Z k$-module $\mathscr{M}$ is a direct summand of a free module,
and thence it is locally free, whose rank is $r$ since the rank of $O_B$ as an $O_F$-module is $r^2$.
This is the result.

Conversely, suppose that $\Delta'$ is invertible on $S$ and that $\mathscr{M}$ is locally free of rank $r$ over $O_F \otimes \mathscr{O}_S$.
We may assume that $\mathscr{M}$ is moreover free over $O_F \otimes \mathscr{O}_S$.
By \cref{lf_iff_matrices}, the canonical map
$O_B \otimes_\Z \mathscr{O}_S \to \mathscr{E}nd_{O_F \otimes_\Z \mathscr{O}_S} \mathscr{M}$ is an isomorphism.
Identifying $O_B \otimes_\Z \mathscr{O}_S$ and $M_r (O_F \otimes_\Z \mathscr{O}_S)$ via the canonical map, the $O_B \otimes_\Z \mathscr{O}_S$-modules $\mathscr{M}$ and $(O_F \otimes_\Z \mathscr{O}_S)^r$ are isomorphic.
Hence the result.
\end{proof}

\section{Preliminaries on Shimura varieties parametrizing QM-abelian varieties} \label{section:quaternionic_Shimura_varieties}

Next we give some preliminaries on the Shimura varieties of PEL type associated with our $B$.
Before explaining it we recall some linear-algebraic objects from \cite[Section 1.1.4]{Lan}.

Let $R$ be a commutative ring and $M$ be a finitely generated left $O_B \otimes_\Z R$-module.
An $(O_B \otimes_\Z R, *)$-pairing on $M$ is an $R$-bilinear map $\varphi \colon M \times M \to R$ which satisfies, for all $x, y \in M$ and for all $b \in O_B \otimes_\Z R$, that $\varphi(bx, y) = \varphi(x, b^*y)$.
A symplectic $O_B \otimes_\Z R$-module is a pair $(M,\varphi)$, where $M$ is a finitely generated $O_B \otimes_\Z R$-module and where $\varphi$ is an alternating $(O_B \otimes_\Z R, *)$-pairing.
For symplectic $O_B \otimes_\Z R$-modules $(M, \varphi)$ and $(N, \psi)$, a symplectic morphism is a pair $(f, \nu(f))$ of an $O_B \otimes R$-morphism $f \colon M \to N$ and $\nu(f) \in R$ such that $\psi(f(x), f(y)) = \nu(f) \varphi(x,y)$ for all $x,y \in M$.
Note that if $\varphi$ is nondegenerate and if $M$ is faithful as an $R$-module, then $\nu(f)$ is uniquely determined by $f$.
A symplectic morphism $(f, \nu(f))$ is a symplectic isomorphism if $f$ is an isomorphism and if $\nu(f) \in R^*$.
If no confusion can arise we write $M$ and $f$ instead of $(M, \varphi)$ and $(f, \nu(f))$ respectively.

For symplectic $O_B \otimes_\Z R$-modules $M$ and $N$, we denote the set of symplectic isomorphisms from $M$ to $N$ by $\operatorname{SymIsom}_{O_B \otimes_\Z R}(M,N)$.
Similarly for a scheme $S$, for sheaves of finitely generated left $O_B \otimes_\Z R$-modules $X$ and $Y$ on the etale topology of the category of $S$-schemes,
and for alternating $(O_B \otimes_\Z R, *)$-pairings on them,
let the symbol $\operatorname{SymIsom}_{O_B \otimes_\Z R}(X,Y)$ denote
the set of symplectic isomorphisms from $X$ to $Y$.
The symbol $\mathscr{S}ym\mathscr{I}som_{O_B \otimes_\Z R}(X,Y)$ denotes the sheaf on the etale topology of the category of $S$-schemes
which sends a scheme $T$ to the set $\operatorname{SymIsom}_{O_B \otimes_\Z R}(X_T,Y_T)$.
Finally we write similarly for the symplectic automorphisms, for example $\operatorname{SymAut}_{O_B \otimes_\Z R}(M) = \operatorname{SymIsom}_{O_B \otimes_\Z R}(M,M)$.

Let $\Lambda$ be a free left $O_B$-module of rank one and $\psi$ a nondegenerate alternating $(O_B, *)$-pairing on $\Lambda$.
(Such a form always does exist by \cite[Lemma 1.1]{Milne}.)
We fix them in the rest of the paper.
First we claim the following:

\begin{lemma} \label{psi(gg)}
Let $R$ be a commutative ring on which $\Delta'$ is invertible.
Then for $x, y \in \Lambda \otimes R$ and for $g \in \GL_{O_B \otimes R} (\Lambda \otimes R)$, we have that $\psi(g(x), g(y)) = \psi(x, \Nrd(g) y)$, where $\Nrd(g)$ is the reduced norm of $(\alpha \circ g \circ \alpha^{-1})(1)$ as an element of the Azumaya algebra $O_B \otimes R$ for an isomorphism $\alpha \colon \Lambda \otimes R \to O_B \otimes R$ of $O_B \otimes R$-modules, which is independent of the choice of $\alpha$.
\end{lemma}

\begin{proof}
Let $H$ be the group scheme $\operatorname{Res}_{O_F/\Z} \GL_{O_B} \Lambda$ over $\Z[1/\Delta']$.
Define the subgroup scheme $G$ of $H$ over $\Z[1/\Delta']$ by
\begin{equation*}
G(R) = \{ g \in \GL_{O_B \otimes_\Z R} (\Lambda \otimes_\Z R) | \psi(g-,g-) = \psi(-,\Nrd g -) \}
\end{equation*}
for a ring $R / \Z[1/\Delta']$.
We show that $G = H$.
First by \cite[Lemma 1.1 (d)]{Milne} and by its proof we have $G(\Q) = H (\Q)$.
Hence by \cite[Theorem 17.93]{MilneAG}, $G_\Q = H_\Q$.

Now since the reduced norms of Azumaya algebras are compatible with base changes, for a flat scheme $S$ over $\Z[1/\Delta']$, we have that $G(S) = H(S)$.
In particular for $S = H$ itself this equation holds.
Hence the identity map of $H$ factors through the inclusion $G \injects H$.
Thus this inclusion is also an epimorphism, i.e., an isomorphism.
\end{proof}

Define the subgroup scheme $G_1$ of $\operatorname{Res}_{O_F/\Z} \GL_{O_B} \Lambda \times \operatorname{Res}_{O_F/\Z} \G_{m,O_F}$ over $\Z$ by
\begin{equation*}
G_1(R) = \{ (g, \nu) \in \GL_{O_B \otimes_\Z R} (\Lambda \otimes_\Z R) \times (O_F \otimes_\Z R)^* | \psi(g-,g-) = \psi(-, \nu -) \}
\end{equation*}
for a ring $R$.
As we have noted, if $R$ is flat over $\Z$ or if $[\Lambda^\vee : \Lambda]$ is invertible in $R$, then the canonical projection $G_1(R) \to \GL_{O_B \otimes_\Z R} (\Lambda \otimes_\Z R)$ is injective.
Under this injection, by \cref{psi(gg)}, if $R$ is flat over $\Z$ or if $[\Lambda^\vee : \Lambda]\Delta'$ is invertible in $R$, then the group $G_1(R)$ is isomorphic to $\{ g \in \GL_{O_B \otimes R} (\Lambda \otimes R) | \Nrd(g) \in (O_F \otimes R)^* \}$.
(Here, if $R$ is flat over $\Z$, for $g \in \GL_{O_B \otimes R} (\Lambda \otimes R)$ we define $\Nrd(g)$ as the reduced norm of $g$ as an element of $\GL_{O_B \otimes R_\Q} (\Lambda \otimes R_\Q)$.)
In particular over $\Z[1/[\Lambda^\vee : \Lambda]\Delta']$ the canonical map $G_1 \to \operatorname{Res}_{O_F/\Z} \GL_{O_B} \Lambda$ is an isomorphism.
It also follows that the canonical map $G_1 \to \operatorname{Res}_{O_F/\Z} \G_{m,O_F}$ is the reduced norm.

Let $G_0 = \operatorname{SymAut}_{O_B}(\Lambda)$, i.e., the group scheme over $\Z$ satisfying for a ring $R$ that $G_0(R) = \operatorname{SymAut}_{O_B \otimes_\Z R}(\Lambda \otimes R)$.
As in the case of $G_1$, if $R$ is flat over $\Z$ or if $[\Lambda^\vee : \Lambda]$ is invertible in $R$, then the canonical projection $G_0(R) \to \GL_{O_B \otimes_\Z R} (\Lambda \otimes_\Z R)$ is injective, and under this injection, by \cref{psi(gg)}, if $R$ is flat over $\Z$ or if $[\Lambda^\vee : \Lambda]\Delta'$ is invertible in $R$, then the group $G_0(R)$ is isomorphic to $\{ g \in \GL_{O_B \otimes R} (\Lambda \otimes R) | \Nrd(g) \in R^* \}$.
In particular, we have the following Cartesian diagram of group schemes over $\Z[1/[\Lambda^\vee:\Lambda]\Delta']$:
\begin{equation*}
\begin{aligned}
\xymatrix{
G_0 \ar[rr]^-{\text{projection}} \ar[d]^{\nu} & & \operatorname{Res}_{O_F/\Z} \GL_{O_B}\Lambda \ar[d]^{\Nrd} \\
\G_m \ar[rr]^-{\text{canonical}} & & \operatorname{Res}_{O_F/\Z} \G_{m,O_F}.
}
\end{aligned}
\end{equation*}
For every place $\mathfrak{p}$ of $F$ splitting $B$ and relatively prime to $[\Lambda^\vee : \Lambda]$, the base change of the kernel of the reduced norm $\GL_{O_B}\Lambda \to \G_{m,O_F}$ from $O_F$ to $O_{F,\mathfrak{p}}$ is $\SL_2 O_{F, \mathfrak{p}}$.
Thus the map $G_0 \to \G_m$ is smooth over $\Z[1/[\Lambda^\vee:\Lambda]\Delta']$, in particular the group scheme $G_0$ is smooth over $\Z[1/[\Lambda^\vee:\Lambda]\Delta']$.

By the Skolem--Noether theorem, there exists an element $t \in B$ such that $b^* = t^{-1} \overline{b} t$ for every $b \in B$, where $\overline{b}$ is the standard involution of $B$, and which is mapped to an element
\begin{equation*}
(c_i \left( \begin{matrix} 0 & -1 \\ 1 & 0 \end{matrix} \right) )_i
\end{equation*}
for positive real numbers $c_i$ under our fixed isomorphism $B \otimes \R \isom M_2(\R)^d$.
Fixing an isomorphism $\Lambda \isom O_B$ and identifying them, let $\psi^{\operatorname{st}}$ be a nondegenerate alternating $(O_B, *)$-pairing on $\Lambda$ satisfying all properties as in \cite[Lemma 1.1]{Milne}.
Then by \cite[Lemma 1.1 (d)]{Milne}, for an element $\theta \in F$ we have that $\psi(u,v) = \psi^{st}(u, \theta v)$ for all $u,v \in \Lambda$.
Let $(\epsilon_i)_i$ be the image of $\theta$ under $F^* \to (F \otimes \R)^* \to \{ \pm 1 \}^d$, where the second map is the sign map.
Using our fixed isomorphisms $\Lambda \isom O_B$ and $B_\R \isom M_2(\R)^d$, define $h \colon \DeligneS \to G_{1,\R} \isom (\GL_2 \R)^d$ by 
\begin{equation*}
h(a + b\sqrt{-1}) = ( \left( \begin{matrix} a & \epsilon_i b \\ - \epsilon_i b & a \end{matrix} \right) )_i.
\end{equation*}
Then it is easy to see that it gives the map $\DeligneS \to G_{0,\R}$, which again we denote by $h$.
Since $\psi^{st}(u, v) = \tr_{B/\Q}(utv^*)$ for all $u,v \in \Lambda_\Q$ under our identification $\Lambda_\Q \isom B$ (see the proof of \cite[Lemma 1.1]{Milne}), since $*$ corresponds to the transpose of matrices under our identification $B_\R \isom M_2(\R)^d$, and since $\tr_{B/\Q}(b b^*)$ is positive for all $b \in B_\R$, it is also easy to see that $h$ satisfies the properties described in \cite[Proposition 8.14]{MilneSh} (which cites \cite[Lemma 3.1]{Zink}), i.e., $\psi(u, h(\sqrt{-1})u) > 0$ for all nonzero $u \in \Lambda_\R$.
(See \cite[Example 2.12]{MilneSh}.)
Let $X_0$ be the $G_0(\R)$-conjugacy class of $h$, which we can consider as a finite disjoint union of hermitian symmetric domains canonically.
Then $(G_{0,\Q}, X)$, which is called the PEL Shimura datum associated with $(B, *, \Lambda \otimes \Q, \psi_\Q)$, is a Shimura datum.
For more details, see, for example \cite[Section 8]{MilneSh}.

Decompose $\Lambda \otimes_\Z \C = V_0 \oplus V_0^c$ as left $O_B \otimes_\Z \C$-modules, where $V_0$ (respectively $V_0^c$)
is the subspace of $\Lambda \otimes_\Z \C$ on which, for every $z \in \C^*$,
the endomorphism $h(z)$ is equal to the scalar multiplication by $1 \otimes z \in O_B \otimes_\Z \C$
(respectively $1 \otimes z^c$, where $z^c$ is the complex conjugate).
Identifying $B \otimes \C$ with $M_2(F \otimes \C)$, we obtain that $V_0 \isom (F \otimes \C)^2$ as $B \otimes \C$-modules.

Note that the reflex field of this Shimura datum
is $\Q(\{ \tr_{V_0/\C} (b) | b \in B \})$
(see \cite[Example 12.4.c]{MilneSh}, which cites an unpublished paper of Deligne), which is $\Q$.

For an open compact subgroup $K \le G_0(\hat{\Z})$, let $M_K^B$ be the Shimura variety
$G_0(\Q) \backslash G_0(\A_f) \times X / K$.
We can consider it as a normal quasi-projective scheme over $\C$ by the theorem of Baily and Borel, for example see Theorem 3.12, Remark 3.13.(a), and Lemma 5.13 of \cite{MilneSh}.
It is canonically defined over its reflex field, which is $\Q$, by \cite[Corollaire 2.7.21]{DeligneSV}.
We denote the canonical model over $\Q$ of $M_K^B$ by the same symbol $M_K^B$.
Moreover since $B$ is division $M_K^B$ is proper \cite[Theorem 3.3.(b)]{MilneSh}, and hence is projective.
If $K$ is sufficiently small (for example if $K$ is neat) then it is smooth.

The real points of our Shimura varieties $M_K^B$ are studied by Shimura.

\begin{theorem} \cite[Theorem 0]{Shimura} \label{no_real_points}
Let $K \le G_0(\hat{\Z})$ be an open compact subgroup.
Then $M_K^B(\R) = \varnothing$.
\end{theorem}

\begin{proof}
Let $X_1$ be the $G_1(\R)$-conjugacy class of $h$.
Then $(G_{1,\Q}, X_1)$ is a Shimura datum, and $X_1 \isom (\C - \R)^d$ as complex manifolds with the action of $G_1(\R)$.
(For example see \cite[Example 5.6]{MilneSh}.)
Since canonically there exists a map of Shimura data from $(G_{0,\Q},X)$ to $(G_{1,\Q}, X_1)$, the non-existence of the real points of $M_K^B$ follows from the one of the Shimura variety associated with $(G_{1,\Q}, X_1)$ of level $G_1(\hat{\Z})$.
This is \cite[Theorem 0]{Shimura}.
\end{proof}

For an open compact subgroup $K \le G_0(\hat{\Z})$, let $\mathscr{M}_K^B$ be the fibered category over $\Q$, which associates to each scheme $S$
the groupoid of pairs $(A, i, \lambda, \eta)$, where $(A, i, \lambda)$ is a $*$-polarized QM-abelian variety over $S$ and
$\eta$ is a level-$K$ structure.
(Note that since $V_0 \isom (F \otimes_\Q \C)^2$ as above,
by \cref{lf_iff_reduced_norm} our Lie algebra condition as in \cref{def:QMAV} is equivalent to the Kottwitz' determinantal condition \cite[Definition 1.3.4.1]{Lan}, which is introduced by \cite{Kottwitz}.)
For the definition of level-$K$ structures, see \cite[Definition 1.3.7.6]{Lan}.
This is a smooth Deligne--Mumford stack \cite[Theorem 1.4.1.12]{Lan}
of dimension $d$ (by \cite[Theorem 2.2.4.13]{Lan} and by the fact that $V_0 \isom (F \otimes \C)^2$),
and is proper by \cite[III]{BreenLabesse}.
Its coarse moduli space is a scheme by \cite[Corollary 7.2.3.10]{Lan}, and is the canonical model of $(M^B_K)_\C$ by \cite[Theorem 3.31]{MilneSVM}, i.e., $M_K^B$ is the coarse moduli space of $\mathscr{M}_K^B$.
If $K$ is sufficiently small (for example if $K$ is neat)
then $\mathscr{M}_K^B$ is represented by a scheme, which is $M_K^B$.
Note that after a slight modification,
we can consider them over a suitable localization of the ring of inters canonically,
although we do not use them in this paper.

In our case we can describe the level structures in an easier way.

\begin{lemma} \label{level-n_structures}
Let $S$ be a scheme of characteristic zero,
$X = (A, i, \lambda)$ a $*$-polarized QM-abelian variety over $S$,
$n$ a prime-to-$[\Lambda^\vee : \Lambda]\Delta$ integer,
$K$ an open compact subgroup of $G_0(\hat{\Z})$ containing
$\Gamma(n) = \ker( G_0(\hat{\Z}) \to G_0(\Z/n))$,
and let $K_n = K/\Gamma(n)$.
Fixing an isomorphism between the projective limit $\plim \mu_m(\Qbar)$ and $\hat{\Z}$,
for each geometric point $\overline{s} \to S$ and for each $\overline{s} \to \Spec \Qbar$,
consider $\hat{T} A_{\overline{s}}$ with the Weil pairing as a symplectic $O_B \otimes \hat{\Z}$-module.
Then we have the following:
\begin{enumerate}
\item If there exists a level-$K$ structure on $X$,
then for every geometric point $\overline{s} \to S$ and for every $\overline{s} \to \Spec \Qbar$
there exists a symplectic isomorphism from $\Lambda \otimes \hat{\Z}$ to $\hat{T} A_{\overline{s}}$.
\item Assume that for every geometric point $\overline{s} \to S$ and for every $\overline{s} \to \Spec \Qbar$
there exists a symplectic isomorphism from $\Lambda \otimes \hat{\Z}$ to $\hat{T} A_{\overline{s}}$.
Then we have that the level-$K$ structures on $X$ are exactly the global sections of the sheaf
$\mathscr{S}ym\mathscr{I}som_{O_B}( (\Lambda/n)_S, A[n] )/K_n$
on $S$.
\end{enumerate}
\end{lemma}

\begin{proof}
The first statement is trivial by the definition.
For the second statement, since the algebraic stack $\mathscr{M}_K$ is locally of finite presentation over the base \cite[Lemma 1.4.1.10]{Lan},
we may assume that $S$ is locally noetherian.
Hence we may assume that $S$ is connected,
since the statement is a Zariski local question.

Fix a geometric point $\overline{s} \to S$.
Then by \cite[Lemma 1.3.6.6]{Lan} and by the equivalence
between the category of smooth $\Z_\ell$-sheaves
and the category of finitely generated $\Z_\ell$-modules on which $\pi_1(S, \overline{s})$ acts continuously (for example
see \cite[A.1.8]{FK}),
the set of principal level-$n$ structures on $X$ is
the intersection of
$\operatorname{SymIsom}_{O_B, \pi_1(S, \overline{s})}(\Lambda/n, A[n]_{\overline{s}})$, which is the set of $\pi_1(S, \overline{s})$-invariant elements of
$\operatorname{SymIsom}_{O_B}(\Lambda/n, A[n]_{\overline{s}})$,
and the image of the modulo $n$ map from
$\operatorname{SymIsom}_{O_B}(\Lambda \otimes \hat{\Z}, \hat{T} A_{\overline{s}})$ to
$\operatorname{SymIsom}_{O_B}(\Lambda/n, A[n]_{\overline{s}})$.
We claim that this modulo $n$ map is surjective.
Since now we are assuming that there exists
a symplectic isomorphism from $\Lambda \otimes \hat{\Z}$ to $\hat{T} A_{\overline{s}}$,
in order to show it, it suffices to show the surjectivity
of $\operatorname{SymAut}_{O_B}(\Lambda \otimes \hat{\Z}) \to \operatorname{SymAut}_{O_B}(\Lambda/n)$, which is equal to $G_0(\hat{\Z}) \to G_0(\Z/n)$.
Thus the desired surjectivity follows from the smoothness of $G_0$ over $\Z[1/[\Lambda^\vee : \Lambda]\Delta]$.
Therefore, again using \cite[A.1.8]{FK},
it follows that the set of principal level-$n$ structures
on $X$ is equal to $\operatorname{SymIsom}_{O_B}( (\Lambda/n)_S, A[n])$.
Thus by definition of level-$K$ structures
the result follows.
\end{proof}

In some cases we can describe the level structures in a
much easier way than \cref{level-n_structures}.
Recall that $G_1$ is equal to $\operatorname{Res}_{O_F/\Z} \GL_{O_B}\Lambda$ over $\Z[1/[\Lambda^\vee : \Lambda]\Delta']$.

\begin{lemma} \label{level-K_structures}
Keep the notation as in \cref{level-n_structures}.
Assume that for every geometric point $\overline{s} \to S$
there exists a symplectic isomorphism from $\Lambda \otimes \hat{\Z}$ to $\hat{T} A_{\overline{s}}$,
and that there exists a subgroup $\Gamma$ of
$G_1(\Z/n)$ whose reduced norm is the whole of $(O_F/n)^*$ and whose preimage under the canonical map $G_0(\hat{\Z}) \to G_1(\Z/n)$ is $K$.
Then the level-$K$ structures on $X$ are exactly the global sections of the sheaf
$\mathscr{I}som_{O_B}((\Lambda/n)_S, A[n])/\Gamma$ on $S$.
\end{lemma}

This lemma follows from \cref{level-n_structures} easily.
Using \cref{level-K_structures} we can translate some level structures into more convenient forms as in the case of modular curves.
We discuss it later in the next section.

For a prime-to-$\Delta$ integral ideal $\mathfrak{n}$ of $F$, we define the groups $\Gamma_0(\mathfrak{n})$ and $\Gamma_1(\mathfrak{n})$ to be the preimages of
\begin{equation*}
\left\{ \left( \begin{matrix} a & b \\ 0 & d \end{matrix} \right) \bigg | a,d \in (O_F/\mathfrak{n})^*, b \in O_F/\mathfrak{n} \right\}
\text{ and }
\left\{ \left( \begin{matrix} a & b \\ 0 & 1 \end{matrix} \right) \bigg | a \in (O_F/\mathfrak{n})^*, b \in O_F/\mathfrak{n} \right\}
\end{equation*}
under the canonical map $G(\hat{\Z}) \to \GL_{O_B/\mathfrak{n}} \Lambda/\mathfrak{n} \isom \GL_2(O_F/\mathfrak{n})^\opp$, where the last map is induced by our fixed isomorphisms $\Lambda \isom O_B$ and $O_B/\mathfrak{n} \isom M_2(O_F/\mathfrak{n})$, respectively.
Similarly we define the group $\Gamma(\mathfrak{n})$ to be the kernel of the canonical map $G(\hat{\Z}) \to \GL_2(O_F/\mathfrak{n})^\opp$.
For $K = G(\hat{\Z})$, we simply denote $\mathscr{M}_K^B$ and $M_K^B$ by $\mathscr{M}^B$ and $M^B$ respectively.
For a prime-to-$\Delta$ integral ideal $\mathfrak{n}$ of $F$, and for $K= \Gamma_i(\mathfrak{n})$ ($i = 0$ or $1$),
we denote $\mathscr{M}_K^B$ and $M_K^B$ by $\mathscr{M}_i^B(\mathfrak{n})$ and $M_i^B(\mathfrak{n})$ respectively.
Similarly for the group $K = \Gamma(\mathfrak{n})$, we denote $\mathscr{M}_K^B$ and $M_K^B$
by $\mathscr{M}^B(\mathfrak{n})$ and $M^B(\mathfrak{n})$ respectively.
By \cite[Section 21, Theorem 5]{MumfordAV}, for a rational integer $n \ge 3$ the algebraic stack $\mathscr{M}^B(n)$ is a scheme.

\section{General theory on QM abelian varieties} \label{section:QM-abelian_varieties:general_theory}

In this section we show fundamental statements about QM-abelian varieties.

For an abelian scheme $A$ over a scheme $S$, for a ring $R$, and for a ring homomorphism $i \colon R \to \End (A/S)$,
we write $\End_R (A/S)$ for the set of endomorphisms of $A/S$ which commute with the action of $R$
and similarly for $\Aut_R (A/S)$.
Similarly we write $\End^0_{R \otimes_\Z \Q} (A/S) = \End_R (A/S) \otimes_\Z \Q$.
When no confusion can arise,
we simply denote them by $\End_R A$, by $\Aut_R A$ and by $\End^0_{R \otimes_\Z \Q} A$ respectively.
Similar for $\Hom_R(A, B)$ and for $\Hom^0_{R \otimes_\Z \Q}(A, B)$, for two abelian schemes and for actions of $R$ on them.

For classification of endomorphism rings of abelian varieties we use the notation as in the table in \cite[Chapter IV, Section 21, p202]{MumfordAV}.
Note that although the base field is assumed to be algebraically closed in \cite{MumfordAV},
we do not need this assumption.

Let $X$ be a simple abelian variety over a field $k$,
and let $D = \End^0 X$ and $K$ the center of $D$.
Then by \cite[Chapter IV, Section 19, Corollary, p182]{MumfordAV}, we have that
$\sqrt{[D:K]} [K : \Q]$ divides $2 \dim X$.
Recall that $X$ is called of CM-type if $\sqrt{[D:K]} [K : \Q] = 2 \dim X$.
Also recall that, in general an abelian variety is called of CM-type if all of its simple factors are of CM-type.

\begin{lemma} \label{H_1/O_B}
Let $A$ be a QM-abelian variety over $\C$.
Then the first singular homology group $H_1(A(\C), \Q)$ is free of rank $1$ over $B$.
\end{lemma}

\begin{proof}
By the definition the $F \otimes_\Q \C$-module $\Lie(A)$ is free of rank $2$.
Since $H_1(A(\C), \R) \isom \Lie A$ as $B \otimes_\Q \R$-modules,
the homology group $H_1(A(\C), \Q)$ is $4$-dimensional over $F$.
Thus it is free of rank $1$ over $B$.
\end{proof}

For QM-abelian surfaces (i.e., QM-abelian varieties in the case of $F = \Q$),
the Tate modules at a prime $\ell$ are free of rank $1$ over the completion of $O_B$ at $\ell$
(for example see \cite[Proposition 1.1 (1)]{Ohta}, which cites the unpublished master thesis of Y. Morita).
This also holds for general $F$.

\begin{proposition} \label{V_l/F_l}
Let $k$ be a field, $A$ an abelian variety over $k$ on which an order of a field $L$ acts,
and let $\ell$ be a prime.
Then the Tate module $V_\ell A$ is free over $L_\ell$.
\end{proposition}

\begin{proof}
We may assume that $V_\ell A \ne 0$.
The ring $L_\ell$ is the product $\prod_{v|\ell} L_v$ of the completions of $L$ at places above $\ell$.
For each place $v$ define $V_v = V_\ell A \otimes_{L_\ell} L_v$.
Write $\dim_{L_v} V_v = n_v$.
For any $\alpha \in \End^0 A$, the trace of $V_\ell \alpha$ on $V_\ell A$ is a rational number.
In particular for $\theta \in L$, the trace of $V_\ell \theta$ on $V_\ell A$, which is equal to $\sum_v n_v \tr_{L_v/\Q_\ell}(\theta)$, is a rational number.
We show that this forces that the integers $n_v$ do not depend on $v$.
Since the $L$-vector space of $\Q$-linear maps from $L$ to $\Q$ is one-dimensional and is generated by $\tr_{L/\Q}$,
we have that there exists $a \in L$ such that $\sum_v n_v \tr_{L_v/\Q_\ell} = a\tr_{L/\Q}$.
Since $\Hom(L, \Q) \otimes \Q_\ell  = \oplus \Hom(L_v, \Q_\ell)$ and each of the direct summand of the right hand side is also a one-dimensional vector space generated by $\tr_{L_v/\Q_\ell}$ over $L_v$, we have that $a = n_v$,
in particular these $n_v$ are the same.
This shows the freeness of the $L_\ell$-module $V_\ell A$.
\end{proof}

\begin{proposition} \label{V_l/B_l}
Let $k$ be a field, $A$ an abelian variety over $k$ on which an order of $B$ acts,
and let $\ell$ be a prime.
Then the Tate module $V_\ell A$ is free over $F_\ell$.
Moreover we have the following:
\begin{enumerate}
    \item
    If $\ell \ne \ch k$ and if the dimension of $A$ is divisible by $2d$, then $V_\ell A$ is free over $B_\ell$.
    \item
    If $\ell = \ch k$ and if the dimension of $A$ is $2d$,
    then $V_\ell A$ is free of rank $0$ or $2$ over $F_\ell$, and the latter case occurs only when $\ell$ does not divide $\Delta'$.
\end{enumerate}
\end{proposition}

\begin{proof}
We may assume that $V_\ell A \ne 0$.
The freeness of $V_\ell A$ as the $F_\ell$-module is just \cref{V_l/F_l}.
Let $n$ be the rank of the $F_\ell$-module $V_\ell A$ and $r$ the dimension of $V_\ell A$ over $\Q_\ell$.
Then $n = r/d$.
Let $v$ be a place of $F$ dividing $\ell$ and let $V_v = V_\ell A \otimes_{F_\ell} F_v$.
If $v$ splits $B$, then $V_v$ is isomorphic to a direct sum of some copies of $F_v^2$, and hence $n$ is even.
Moreover in this case if $n/2$ is even then $V_v$ is free over $B_v$ since $(F_v^2)^2 \isom M_2(F_v) = B_v$ as $B_v$-modules.
If $v$ does not split $B$, then $V_v$ is free over $B_v$ and in particular $n$ must be divisible by $4$.
Since we know that $r = 2 \dim A$ if $\ell \ne \ch k$,
and that $0 \le r \le \dim A$ if $\ell = \ch k$,
this finishes the proof.
\end{proof}

In the proof above, furthermore we have shown that the dimension of $A$ is divided by $d$, and moreover that if there exists a prime $\ell \ne \ch k$ which does divide $\Delta'$, i.e., if $\ch k \ne \Delta'$, then the dimension of $A$ is divided by $2d$.
In \cref{V_l/B_l}, if the order acting on $A$ is maximal, then we have the following obvious corollary.

\begin{corollary} \label{T_l/O_B_l}
Let $k$ be a field, $A$ an abelian variety over $k$ on which $O_B$ acts,
and let $\ell$ be a prime.
Then the Tate module $T_\ell A$ is free over $O_{F,\ell}$.
Moreover we have the following:
\begin{enumerate}
    \item
    If $\ell \ne \ch k$ and if the dimension of $A$ is divisible by $2d$, then $T_\ell A$ is free over $O_{B,\ell}$.
    \item
    If $\ell = \ch k$ and if the dimension of $A$ is $2d$,
    then $T_\ell A$ is free of rank $0$ or $2$ over $O_{F,\ell}$, and the latter case occurs only when $\ell$ does not divide $\Delta'$.    
\end{enumerate}
\end{corollary}

\begin{proof}
    The statements about $O_{F,\ell}$-modules are trivial by \cref{V_l/B_l}, since $O_{F,\ell}$ is a product of discrete valuation rings.
    Assume that $\ell \ne \ch k$ and that the dimension of $A$ is divided by $2d$, say $2d\delta$.
    We show that $T_\ell A$ is free over $O_{B,\ell}$.
    Let $v$ be a place of $F$ above $\ell$ and $T_v := T_\ell A \otimes_{O_{F,\ell}} O_v$.
    Note that the rank of the free $O_{F,v}$-module $T_v$ is $4\delta$.
    
    First we claim that $T_v \otimes_{O_{F,v}} \F_v$ is free over $O_B \otimes_{O_F} \F_v$.
    If $v$ splits $B$, then since $O_B \otimes \F_v$ is simple and since the rank of $T_v$ as an $O_{F,v}$-module is divided by $4$, we obtain the claim.
    If $v$ does not split $B$, then every one-sided ideal of $O_{B,v}$ is in fact two-sided and principal by \cite[13.3.10]{Voight}, $O_{B,v}$ has only one maximal ideal $\mathfrak{P}$, which satisfies $\mathfrak{P}^2 = \mathfrak{p}_v O_{B,v}$, where $\mathfrak{p}_v$ is the prime of $F$ corresponding to $v$, and $\F_\mathfrak{P} := O_{B,v}/\mathfrak{P}$ is a quadratic field of $\F_v$ by \cite[Theorem 13.3.11]{Voight}.
    Thus $\dim_{\F_v} T_v \otimes \F_v = 2 \dim_{\F_v} T_v \otimes_{O_{B,v}} \F_\mathfrak{P}$, and hence $\dim_{\F_\mathfrak{P}} T_v \otimes \F_\mathfrak{P} = \delta$.
    Therefore applying Nakayama's lemma to the noncommutative local ring $O_B \otimes \F_v$, we obtain a surjection $O_B^\delta \otimes \F_v \to T_v \otimes \F_v$ of $O_B$-modules.
    Comparing their dimensions over $\F_v$, we see that it is an isomorphism, which shows the claim.
    As a consequence, in any case the $O_B \otimes \F_v$-module $T_v \otimes \F_v$ is free of rank $\delta$.
    Hence again by Nakayama's lemma there exists a surjective $O_B$-homomorphism $O_{B,v}^\delta \to T_v$.
    Again comparing their ranks over $O_{F,v}$, we see that it is an isomorphism.
    This shows the statement.
\end{proof}

Next we show that actions of orders of quaternion algebras over totally real number fields on abelian varieties force  their dimensions to be big
by considering their endomorphism rings.
In order to show it we need a fundamental lemma about central simple algebras.

\begin{lemma} \label{C_E(B)_and_C_E(F)}
    Let $L/K$ be a field extension of degree $f$, $W/L$ and $D/K$ central simple algebras, and $E := M_m(D)$ for a positive integer $m$ divided by $f$.
    Fix a $K$-algebra homomorphism $W \to E$.
    Then $C_E(W)$ and $C_E(L)$ are central simple algebras over $L$ and we have $W \otimes_L C_E(W) \isom C_E(L) \isom M_{m/f}(D \otimes_K L)$.
\end{lemma}

\begin{proof}
Take a $K$-algebra homomorphism $L \to M_f (K)$.
Since $f$ divides $m$, by tensoring it with $D$ we obtain a morphism $L \to M_m(D) = E$.
By the Skolem--Noether theorem, it differs from the map $L \to W \to E$, where the latter map is our fixed map,
by an inner automorphism of $E$.
Thus we can compute the centralizer $C_E(L)$ using the map $L \to M_f(K) \to E$ obtained above.
Therefore we have $C_E(L) = M_{m/f}(D \otimes_K L)$.
In particular $C_E(L)$ is a central simple algebra over $L$.
Consequently, $C_E(W) = C_{C_E(L)}(W)$ is a central simple algebra over $L$ by \cite[Section 8, Theorem 9]{SerreEx} and one gets $W \otimes_L C_{C_E(L)}(W) = C_E(L) = M_{m/f}(D \otimes_K L)$ by \cite[Section 8, Corollaire 1]{SerreEx}.
\end{proof}

Let $k$ be a field, $O$ an order of $B$, and $A$ be a nonzero $k$-abelian variety on which $O$ acts.
We say that $A$ is $B$-simple if there are no nontrivial abelian subvarieties $A'$ of $A$ on which orders $O'$ of $B$ act so that the inclusions $A' \to A$ are compatible with the actions of $O \cap O'$.
Note that since for an abelian subvariety $A'$ of $A$, the abelian variety $A$ is isogenous to the product $A' \times A/A'$, and since if an order of $B$ acts on $A'$ so that the inclusion is compatible with the actions then there exists an order of $B$ which acts on $A/A'$, we have that
there exists an isogeny between $A$ and a product $\oplus_i A_i$, where each $A_i$ is an abelian variety on which an order of $B$ acts and is $B$-simple.
We show several fundamental statements about $B$-simple abelian varieties.

\begin{lemma} \label{simple_power}
Let $k$ be a field and $A$ a nonzero abelian variety over $k$ on which an order of $B$ acts.
If $A$ is $B$-simple, then $A$ is isogenous to a power of a simple abelian variety.
\end{lemma}

\begin{proof}
Assume that $A$ is $k$-isogenous to $\oplus_{i=1}^r A_i^{n_i}$,
where $A_i$ are pairwise non-isogenous simple abelian varieties over $k$ and $n_i > 0$.
Then it follows that as algebras $\End^0 A \isom \prod_i \End^0 (A_i^{n_i})$.
Hence an order of $B$ acts on each of $A_i^{n_i}$.
Thus, since $A$ is $B$-simple, the result follows.
\end{proof}

\begin{lemma} \label{isogeny}
Let $k$ be a field
and $A$ and $A'$ abelian varieties of the same dimension over $k$ on which an order of $B$ acts.
Let $f$ be a nonzero homomorphism from $A$ to $A'$ which commutes with the action of an order of $B$.
Then if $A$ is $B$-simple, then $f$ is an isogeny.
\end{lemma}

\begin{proof}
Since $f$ commutes with the action of an order of $B$,
an order of $B$ also acts on $\ker f$, and hence on $(\ker f)^0$, the identity component,
and thus on $H := (\ker f)^0_\red$, which is an abelian variety.
Since $f$ is nonzero, by the $B$-simplicity, we have $H = 0$.
Therefore $\ker f$ is of dimension zero, which exactly means that $f$ is an isogeny.
\end{proof}

\begin{lemma} \label{FK_is_a_field}
Let $k$ be a field and $A$ a nonzero abelian variety over $k$ on which an order of $B$ acts.
Let $K$ be the center of $\End^0 A$.
If $A$ is $B$-simple, then the composition algebra $FK$ of $F$ and $K$ in $\End^0 A$ is a field.
\end{lemma}

\begin{proof}
Both of $F$ and $K$ are contained in $\End_B^0 A$, which is a finite dimensional division algebra over $\Q$ by \cref{isogeny}, and $FK$ is commutative since $K$ is the center.
Hence $FK$ is a field.
\end{proof}

\begin{proposition} \label{dim_gt_d}
Let $k$ be a field and $A$ a nonzero abelian variety over $k$ on which an order of $B$ acts.
Then the dimension of $A$ is divisible by $d$ and not equal to $d$.
\end{proposition}

Although the first part is shown in \cref{V_l/B_l}, in preparation for the proof of the second part, we begin with another proof of the first part.

\begin{proof}
Considering a decomposition of $A$ into a product of abelian varieties on which orders of $B$ act and which are $B$-simple,
we may assume that $A$ itself is $B$-simple.
Thus by \cref{simple_power}, $A$ is $k$-isogenous to $X^m$, where $X$ is a simple abelian variety over $k$ and $m > 0$.

Let $D = \End^0 X, E = \End^0 A = M_m D$, and $K$ be the center of $D$.
By \cref{FK_is_a_field} the composition $FK$ in $E$ is a field.
The sub algebra $BK$ of $E$ is isomorphic to $B \otimes_F FK$, in particular is also a central simple algebra over $FK$,
for, $B \otimes_F FK$ is a simple algebra and it surjects onto $BK$.
Take a quadratic extension $L$ of $FK$ contained in $BK$.
Since $L$ is commutative we have $2[FK:K] = [L:K] \mid \sqrt{[E : K]}$,
and hence $d[FK:F] \mid \dim A$ by \cite[Chapter IV, Section 19, Corollary, p182]{MumfordAV}.
It is the first part of the statement.

Next we show that $\dim A$ is greater than $d$.
Suppose the contrary, i.e., $\dim A = d$.
In this case we have $2[FK:K] = \sqrt{[E:K]}$, $[FK:F] = 1$, and that $A$ has CM.
In this case $K$ is totally real and hence $D$ is of type III by the table in \cite[Chapter IV, Section 21, p202]{MumfordAV}
(and it follows that the characteristic of $k$ must be positive, although we do not use it),
in particular $D/K$ is a quaternion.
Thus for $f = [F : K]$ we have $f = m$ and $[K:\Q] = \dim X$.
Applying \cref{C_E(B)_and_C_E(F)}, we have that $B \otimes_F C_E(B) = C_E(F) = D \otimes_K F$.
Comparing their dimensions over $F$, these isomorphisms yield that $B \isom D \otimes_K F$.
However, since $B/F$ is totally indefinite while $D/K$ is totally definite, this is a contradiction, which completes the proof.
\end{proof}

\begin{corollary} \label{B-simple}
Let $k$ be a field and $A$ an abelian variety of dimension $2d$ over $k$ on which an order of $B$ acts.
Then $A$ is $B$-simple.
\end{corollary}

\begin{proof}
If $A$ has a nonzero abelian subvariety on which an order of $B$ acts, then \cref{dim_gt_d} forces that its dimension must equal to $2d$, i.e., it is $A$ itself,
which is what we want.
\end{proof}

From now we classify the endomorphism rings of abelian varieties on which orders of $B$ act.
Before starting the classification we show a statement about the composition of $F$ and the center of the endomorphism ring.

\begin{proposition} \label{[FK:K]_divides_2}
Let $k$ be a field, $A$ a $2d$-dimensional $k$-abelian variety on which an order of $B$ acts, $K$ the center of $\End^0 A$, and let $FK$ be the composite of $F$ and $K$ in $\End^0 A$.
Then $FK$ is a field and $[FK:F] = 1$ or $2$.
Moreover it is $1$ if and only if $K$ is totally real, and it is $2$ if and only if $K$ is totally imaginary.
\end{proposition}

\begin{proof}
At first we note that $K$ is totally real or totally imaginary since it is the center of the endomorphism ring of an abelian variety, and that $FK$ is a field by \cref{FK_is_a_field,B-simple}.
By \cref{B-simple,simple_power} we have that $A$ is $B$-simple and is $k$-isogenous to $X^m$ for a simple abelian variety $X$.
Trivially $2d = \dim A = m \dim X$.
Let $D = \End^0 X$ and $E = \End^0 A$, which is equal to $M_m(D)$.
Taking a quadratic extension $L$ of $FK$ in $BK$, we obtain that $2[FK:K] = [L:K] \mid \sqrt{[E:K]}$, and hence $2d [FK:F] \mid \sqrt{[E:K]} [K : \Q]$.
The latter divides $2 \dim A = 4d$.
It shows the first part of the statement.
It is trivial that if $[FK : F] = 1$ then $K$ is totally real.
Thus from now we assume that $K$ is totally real and show that $[FK : F] = 1$.

If $A$ has no CM, then since $\sqrt{[E:K]} [K : \Q] \ne 4d$, we have that $FK = F$.
Next assume that $A$ has CM.
Then the algebra $D$ is a totally definite quaternion algebra over $K$ by the table in \cite[Chapter IV, Section 21, p202]{MumfordAV}.
Hence $f = [FK:K]$ divides $m$.
Therefore applying \cref{C_E(B)_and_C_E(F)} to $FK/K$, $BK/FK$, and $E/K$, we obtain that
$BK \otimes_{FK} C_E(BK) = M_{m/f}(D \otimes_K FK)$.
By comparing their dimensions over $FK$ and by the fact that $BK$ is totally indefinite but $D$ is not,
this isomorphism implies that $f < m$.
Therefore $[FK:F] = 1$.
\end{proof}

First we consider the non-CM case.

\begin{proposition} \label{End:ch0:nonCM}
Let $k$ be a field and let $A$ be a non-CM $2d$-dimensional $k$-abelian variety on which an order of $B$ acts.
Then we have that $\End^0_B (A/k) = F$, $\End^0_F (A/k) = B$,
and $\End^0 (A/k)$ is a matrix algebra of a totally indefinite quaternion algebra over a subfield of $F$.
\end{proposition}

\begin{proof}
By \cref{B-simple,simple_power} we have that $A$ is $B$-simple and is $k$-isogenous to $X^m$ for a simple abelian variety $X$.
Trivially $2d = \dim A = m \dim X$.
Let $D = \End^0 X, E = \End^0 A$, which is equal to $M_m(D)$, $K$ the center of $D$, and $C = C_E(B)$.
Since $A$ has no CM, the center $K$ is totally real.
Thus by \cref{[FK:K]_divides_2} we have $K \subseteq F$.
Let $f = [F : K]$.
Then taking a quadratic extension of $F$ in $E$, we obtain that $2f \mid \sqrt{[E:K]} = m\sqrt{[D:K]}$.
The latter divides $4f$, and is, since now we are assuming that $A$ has no CM, not equal to $4f$ by \cite[Chapter IV, Section 19, Corollary, p182]{MumfordAV}.
Therefore $m \sqrt{[D:K]} = 2f$.
Now from the fact that $K$ is totally real it follows that $D$ is of type I, II, or III, in particular $\sqrt{[D:K]} = 1$ or $2$ by the table in \cite[Chapter IV, Section 21, p202]{MumfordAV}.
Thus $f \mid m$.
Hence we can apply \cref{C_E(B)_and_C_E(F)}, and get $B \otimes_F C = C_E(F) = M_{m/f}(D \otimes_K F)$.
Suppose that $\sqrt{[D:K]} = 1$, i.e., $m = 2f$.
In this case $B \otimes_F C = M_2(F)$.
Thus $C = F$ and $B = M_2(F)$, which is a contradiction.
Hence we have $\sqrt{[D:K]} = 2$, i.e., $m = f$.
In this case $B \otimes_F C = D \otimes_K F$.
Thus we have that $C = F$ and $B = D \otimes_K F$.
Since we are assuming that $B$ is totally indefinite over $F$,
the quaternion $D$ must be so over $K$,
i.e., $D$ is of type II.
This is the result.
\end{proof}

In the proof above, we obtain more results on the simple factor of $A$: 

\begin{corollary} \label{number_of_simple_factors:ch0:nonCM}
Assume that over $k$, the abelian variety $A$ is isogenous to $X^m$ for a simple abelian variety $X$ and for an integer $m$.
Then there exists a subfield $K$ of $F$ and there exists a quaternion algebra $D/K$ such that
$[F :K] = m$, $B = D \otimes_K F$, and $\End^0 X = D$.
\end{corollary}

\begin{proof}
By the proof of \cref{End:ch0:nonCM}.
\end{proof}

In the case of $F = \Q$, i.e., for QM-abelian surfaces, it is probably well-known to experts that
every non-CM QM-abelian surface over a field of characteristic zero is simple, but we could not find explicit proofs in the literature.
The following, which is an obvious corollary of \cref{number_of_simple_factors:ch0:nonCM}, corresponds to it for general $F$.

\begin{corollary}
Assume that $B$ does not descend to any proper subfields of $F$,
i.e., assume that there are no proper subfields $K$ of $F$
and no quaternion algebras $D$ over $K$ such that $B = D \otimes_K F$.
Then $A$ is simple over $k$ and $\End^0 A = B.$
\end{corollary}

\begin{corollary} \label{Aut:nonCM}
The only finite order elements of $\Aut_B A$ are $\pm 1$.
\end{corollary}

\begin{proof}
Since $\Aut_B A$ is contained in an order of $F$ by \cref{End:ch0:nonCM}, the statement follows from $\mu(F) = \{ \pm 1 \}$.
\end{proof}

Note that, in particular, since the automorphism groups of polarized abelian varieties are finite (for example see \cite[Section 21, Theorem 5]{MumfordAV}), the automorphism group of a $*$-polarized non-CM QM-abelian variety is $\{ \pm 1 \}$ by \cref{Aut:nonCM}.

Next we classify the endomorphism rings of CM abelian varieties on which orders of $B$ act.

\begin{proposition} \label{End:ch0:CM}
Let $k$ be a field and let $A$ be a CM $2d$-dimensional $k$-abelian variety on which an order of $B$ acts.
Then we have the following possibilities:

\begin{enumerate}
\item There exists a totally imaginary quadratic extension $F'$ of $F$ such that $\End^0_B (A/k) = F'$
and $\End^0_F (A/k) = B \otimes_F F'$.
Moreover if the characteristic of $k$ is not dividing $\Delta'$, then $F'$ splits $B$,
in particular $\End^0_F (A/k) = M_2(F')$.
\item There exists a totally definite quaternion algebra $D$ over a subfield $K$ of $F$ such that
$\End^0(A/k)$ is a matrix algebra of $D$, $\End^0_F (A/k) = M_2(D \otimes_K F)$,
$\End^0_B (A/k)$ is a totally definite quaternion algebra over $F$,
and we have $\End^0_B (A/k) \otimes_F B = \End^0_F (A/k)$.
\end{enumerate}
The latter case only occurs when the characteristic of $k$ is positive.
\end{proposition}

\begin{proof}
By \cref{simple_power} we have that $A$ is $k$-isogenous to $X^m$ for a simple abelian variety $X$.
Trivially $2d = \dim A = m \dim X$.
Let $D = \End^0 X, E = \End^0 A$, which is equal to $M_m(D)$, $K$ the center of $D$, and $C = C_E(B)$.
By \cref{[FK:K]_divides_2}, $FK$ is a field, $K$ is totally real if and only if $[FK:F] = 1$,
and $K$ is totally imaginary if and only if $[FK:F] = 2$.

We divide into two cases.
First we consider the case that $K$ is totally real, i.e., $[FK:F] = 1$.
Then since $A$ has CM, the algebra $D$ is of type III and the characteristic of $k$ must be positive by the table in \cite[Chapter IV, Section 21, p202]{MumfordAV}.
Let $f = [F : K]$.
Then since $\sqrt{[E:K]} = 4f$ we have $m = 2f$.
Hence by \cref{C_E(B)_and_C_E(F)}, we obtain that $C_E(F) = M_2(D \otimes_K F)$ and $B \otimes_F C = M_2(D \otimes_K F)$, and in particular $C$ is a totally definite quaternion algebra.

Secondly we consider the case that $K$ is totally imaginary, i.e., $[FK:F] = 2$.
We show that $C=FK$.
For a rational prime $\ell$ not equal to $\ch k$, we can embeds $BK \otimes \Q_\ell$ into $ \End_{FK \otimes \Q_\ell} V_\ell A$.
By \cref{V_l/F_l}, the $FK \otimes \Q_\ell$-module $V_\ell A$ is free of rank two.
Thus comparing their dimensions over $\Q_\ell$, we have that $BK \otimes \Q_\ell = \End_{FK \otimes \Q_\ell} V_\ell A = M_2(FK \otimes \Q_\ell)$, and hence taking their centers, we obtain that $C=FK$.

Finally we show, in this case, that $C_E(F) = BK$ and that if $\Delta'$ is invertible on $k$ then $FK$ splits $B$.
Take a quadratic extension $L$ of $FK$ in $BK$.
Then $4d = 2[FK:\Q] = [L:\Q]$ divides $\sqrt{[E:K]}[K:\Q]$, which is equal to $2 \dim A = 4d$ since now $A$ has CM,
and hence $2[FK:K] = \sqrt{[E:K]}$.
By \cite[Th\'eor\`eme 9]{SerreEx} we have that $[C_E(FK):FK] = 4$.
Since $BK \subseteq C_E(FK)$ and since $BK$ is a quaternion algebra over $FK$, it implies that $C_E(F) = BK$.
We have shown that $(BK)_\ell = M_2((FK)_\ell)$ for $\ell \ne \ch k$.
Therefore if the characteristic of $k$ does not divide $\Delta'$, then $BK$ splits at every place, i.e., $BK = M_2(FK)$.
\end{proof}

If the base field $k$ is finite, then $A$ always has CM.
Moreover in the case that $\End^0_B (A/k)$ is a quaternion algebra, using Tate's theory \cite{Tate}, we can compute its invariant at each place.
See \cite[Proposition 5.2]{Milne}.

Throughout this paper let $n_{\lcm} = \lcm \{ m : [F(\zeta_m) : F] \le 2 \}$, where $\zeta_m$ is a primitive $m$-th roof of $1$.
It divides an easy-to-compute integer $\lcm\{ m : \phi(m) \mid 2d \}$, which depends only on $d$, but not on $F$.
For example, if $d = 1$ then $n_{\lcm} = 12$, and if $d = 2$ then $n_{\lcm}$ divides $120$.
More precisely, $n_{\lcm}$ is $60$ if $F = \Q(\sqrt{5})$, is $24$ if $F = \Q(\sqrt{2})$, and is $12$ if $F$ is an other real quadratic field.
We note some elementary properties about $n_{\lcm}$.
First since $[\Q(\zeta_4) : \Q] = [\Q(\zeta_3) : \Q] = 2$, the integer $n_{\lcm}$ is divisible by $12$.
Next for a prime divisor $p$ of $n_{\lcm}$, we have that $p-1$ divides $2d$, and for such $p$, the $p$-adic value $v_p(n_{\lcm})$ of $n_{\lcm}$ is less than or equal to $v_p(2d) + 1$.

\begin{corollary} \label{Aut:CM}
Let $H$ be a finite subgroup of $\Aut_B A$.
Then its order divides $20n_{\lcm}$, and for every element $f \in H$
we have that $[F(\zeta_m) : F] \le 2$, where $m$ is the order of $f$ and $\zeta_m$ is a primitive $m$-th root of $1$.
Moreover if the characteristic of $k$ is zero then $H$ is cyclic.
\end{corollary}

\begin{proof}
By the proposition, $\End_B^0 A$ is either an imaginary quadratic extension of $F$
or a totally definite quaternion algebra over $F$,
and the latter case only occurs when the characteristic of $k$ is positive.
In the former case the statement is trivial.
In the latter case, let $f \in H$.
Then since the ring $F[f]$ is commutative and since $f$ has finite order, the order $m$ of $f$ satisfies $[F(\zeta_m) : F] \le 2$.
Moreover by \cite[Chapitre 1, Theoreme 3.7]{Vigneras}, the possibility of our group $H$ is:
a cyclic group, a dihedral group, or groups of order $24, 48$, or $120$.
Therefore since $12 \mid n_{\lcm}$ the order of $H$ divides $20n_{\lcm}$.
\end{proof}

\begin{remark} \label{descent}
By \cref{Aut:nonCM,Aut:CM}, in particular, we have, for a field $k$ of characteristic zero, for an open compact subgroup $K$ of $G(\hat{\Z})$, and for an object $X \in \mathscr{M}_K^B(k)$, that $\Aut X$ is commutative.

Let $k$ be a field of characteristic zero, $K$ an open compact subgroup of $G(\hat{\Z})$, $x \in M_K^B(k)$ a rational point, and $X \in \mathscr{M}_K^B(L)$ a model of $x$ over a Galois extension $L$ of $k$.
Assume for every $\sigma \in \Gal(L/k)$ that there exists an isomorphism $\lambda_\sigma \colon X \to X^\sigma$ over $L$.
(For example $L = \kbar$.)
Using them, we define an action of $\Gal(L/k)$ on $\Aut (X/L)$ from right by $(f, \sigma) \mapsto \lambda_\sigma^{-1} f^\sigma \lambda_\sigma$.
By the commutativity of $\Aut X$ this action is well-defined and is independent of the choice of isomorphisms $\lambda_\sigma$ up to isomorphism.
Let $\operatorname{ob}(X, L/k) \in H^2(\Gal(L/k), \Aut (X/L))$, which we call the obstruction to descent of $X$ from $L$ to $k$, be the cohomology class of the $2$-cocycle $(\lambda_\sigma^{-1} (\lambda_\tau^\sigma)^{-1} \lambda_{\tau \sigma})_{\sigma, \tau}$.
We can show that $\operatorname{ob}(X, L/k)$ is independent of the choice of isomorphisms $\lambda_\sigma$.
Also we can show that the object $X$ has a model over $k$ if and only if $\operatorname{ob}(X, L/k)$ is trivial.
\end{remark}

Let $k$ be a field, $A$ a $k$-abelian variety of dimension $2d$ on which an order of $B$ acts, and let $\ell \ne \ch k$ be a prime.
Fixing an isomorphism $\alpha$ between $V_\ell A$ and $B_\ell$ as $B_\ell$-modules by \cref{V_l/B_l}, we obtain an isomorphism $\theta \colon \End_B V_\ell A \to B_\ell^\opp$.
Using this isomorphism, for $f \in \End_B V_\ell A$,
we define the reduced trace and reduce norm of $f$ by
$\Trd (\theta(f))$ and by $\Nrd (\theta(f))$ respectively.
These definitions are independent of the choice of $\alpha$, thus we simply denote them by $\Trd f$ and by $\Nrd f$ respectively.
Note that if the order acting on $A$ is $O_B$ and if $f \in \End_{O_B} T_\ell A$, then $\Trd f$ and $\Nrd f$ are in $ O_{F,\ell}$.

For $f \in \End^0_B A$, the values $\Nrd V_\ell f$ and $\Trd V_\ell f$ a priori lie in $F_\ell$, however, we can show that these lie in $F$.

\begin{proposition} \label{Nrd_lies_in_F}
Let $A$ be an abelian variety of dimension $2d$ over a field $k$ on which an order of $B$ acts,
$f \in \End^0_B A$,
and let $\ell \ne \ch k$.
Then the dimension of $F(f)$ over $F$ is at most $2$.
Moreover $\Trd V_\ell f = 2f$ or $\tr_{F(f)/F}(f)$ and $\Nrd V_\ell f = f^2$ or $N_{F(f)/F}(f)$ if $\dim_F F(f) = 1$ or $2$ respectively.
In particular these two values lie in $F$ and are independent of $\ell$, $f$ is a root of a polynomial $T^2 - \Trd(V_\ell f) T + \Nrd(V_\ell f)$, and the characteristic polynomial of $f$ is $N_{F/\Q}(T^2 - \Trd(V_\ell f) T + \Nrd(V_\ell f))^2$.
\end{proposition}

\begin{proof}
If $f \in F$ then it is trivial.
Thus we may assume the contrary.
Taking $V_\ell$, the algebra $F(f) \otimes \Q_\ell$ injects into $F_\ell(V_\ell f)$ since $\ell \ne \ch k$.
Since $\End_{B_\ell} V_\ell A \isom B_\ell^\opp$ by \cref{V_l/B_l}, and since $F(f) \otimes \Q_\ell$ is commutative and free of rank at least $2$ over $F_\ell$, we have that $F(f) \otimes \Q_\ell = F_\ell(V_\ell f)$, and in particular $F(f)$ is quadratic over $F$.
Thus $\Trd(V_\ell f) = \tr_{F_\ell(V_\ell f)/F_\ell}(V_\ell f) = \tr_{F(f)/F}(f)$, and the same for the reduced norm.
(For example see \cite[Proposition 2.6.3]{GS}.)
Finally since the characteristic polynomial of $f$ is $N_{B_\ell/\Q_\ell}(T - V_\ell f)$, and since $N_{B/F} = \Nrd^2$, we have the results.
\end{proof}

Applying this proposition to the relative Frobenius endomorphisms of abelian varieties over finite fields on which orders of $B$ act, we obtain a statement about the Weil numbers and the characteristic polynomials of them.

\begin{proposition} \label{Weil_number}
Let $k$ be the finite field $\F_q$ of characteristic $p$, $A$ an abelian variety of dimension $2d$ over $k$ on which an order of $B$ acts, and $\pi$ be the relative Frobenius endomorphism of $A/k$.
Then there exists a unique algebraic integer $b \in O_F$ such that $|b|_v \le 2 \sqrt{q}$ for every infinite place $v$ of $F$
and that $\pi^2 - b \pi + q = 0$.
Moreover the characteristic polynomial of $A/k$ equals to $N_{F/\Q}(T^2 - b T + q)^2$.
\end{proposition}

\begin{proof}
By \cref{Nrd_lies_in_F}, only what we need to show is, for $\ell \ne p$ and for every infinite place $v$ of $F$, that $| \Trd V_\ell \pi |_v \le 2 \sqrt{q}$ and $\Nrd V_\ell \pi = q$.
These are just the Weil conjectures since $F$ is totally real.
\end{proof}

As every elliptic curve has only one principal polarization,
it is well-known that every QM-abelian surface has only one principal polarization which commutes with the QM-structure.
A similar result also holds for general QM-abelian schemes,
and it is stated in a slightly different form in \cite[Lemma 1.1]{Milne}.

\begin{proposition} \label{unique_polarization}
Let $S$ be a scheme over $\Z[1/\Delta]$,
$A$ a QM-abelian variety over $S$, and $\lambda$ and $\mu$ be $*$-polarizations.
Then Zariski locally over $S$ there exists $c \in F$ such that $\mu = c \lambda$.
\end{proposition}

\begin{proof}
First assume that $S$ is the spectrum of a field $k$.
Let $\ell$ be a prime not dividing $\Delta'$ and not equal to $\ch k$.
Then $\lambda$ and $\mu$ give nondegenerate alternating $\Q_\ell$-bilinear forms
$\psi_\lambda, \psi_\mu \colon T_\ell A \times T_\ell A \to \Q_\ell$.
Since $T_\ell A$ is a free $O_{B,\ell}$-module by \cref{T_l/O_B_l}, by the proof of \cite[Lemma  1.1]{Milne} there exists $c \in F_\ell$ such that $\psi_\mu(-,-) = \psi_\lambda(-,c-)$.
Therefore for such $c$ we have $T_\ell \mu = c T_\ell \lambda$.
Let $f = \lambda^{-1} \circ \mu \in \End^0_B A$.
Then in $F_\ell$, we have the equality $2c = \Trd V_\ell f$, which implies that $c \in F$ by \cref{Nrd_lies_in_F}.
Therefore $f = c$.

Next we consider the general case.
Since now the statement is a Zariski local question and
since abelian schemes are of finite presentation over their base schemes,
we may assume that $S$ is locally noetherian, and hence we may assume that $S$ is connected.
In this case since at a geometric point of $S$ the two maps $\mu$ and $c \lambda$ are the same for some $c \in F$ by the argument above,
by the rigidity lemma \cite[Proposition 6.1]{MumfordGIT} they coincide over the whole of $S$.
(More precisely taking an integer $m$ so that $mc \in O_F$ apply the rigidity lemma to the maps $m\mu$ and $mc\lambda$.)
\end{proof}

\begin{corollary} \label{dual_isogeny}
Let $S$ be a scheme over $\Z[1/\Delta]$, $(A,i,\lambda)$ and $(A', i', \lambda')$ $*$-polarized QM-abelian schemes over $S$, and $f \in \Hom^0_B(A,A')$.
Let $f^t := \lambda^{-1} \circ f^\vee \circ \lambda'$.
Then Zariski locally $f^t \circ f$ is the multiplication by an element $c$ of $F$.
Moreover, if $(A,i,\lambda) = (A', i', \lambda')$, then we have that $\deg f = N_{F/\Q}(c)^2$.
\end{corollary}

\begin{proof}
The map $f^\vee \circ \lambda' \circ f$ is a $*$-polarization on $A$ over $S$
since $f$ and $\lambda'$ commute with the action of $B$.
Hence Zariski locally over $S$ it is equal to $c \lambda$ for an element $c \in F$ by \cref{unique_polarization}.
Thus $f^t \circ f = c$ Zariski locally.
In order to show the second statement we may assume that $S$ is the spectrum of a field $k$.
For a prime $\ell$ which is invertible in $k$, taking $V_\ell$ and $N_{B/\Q}$, we obtain that $\deg f^\vee \deg f = N_{F/\Q}(c)^4$ by \cref{Nrd_lies_in_F}.
Since $\deg f^\vee = \deg f$ by \cite[Section 15, Theorem 1]{MumfordAV} we have the results.
\end{proof}

Note that since for morphisms $f, g$ of abelian varieties we have $(f+g)^\vee = f^\vee + g^\vee$,
we also have that $(f+g)^t = f^t + g^t$.

It is well-known that the modular curve $X_1(n)$ is the fine moduli of the corresponding algebraic stack
for every $n \ge 4$,
and Buzzard \cite[Lemma 2.2]{Buzzard} proved the corresponding result for QM-abelian surfaces,
i.e., that the variety $M_1^B(n)$ is the fine moduli for every prime-to-$\Delta$ integer $n \ge 4$ in the case $F = \Q$.
Using the statements above we also show the fineness of $M^B_1(\mathfrak{n})$ for general $F$ by the almost same method.
Although the proof below works over not only $\Q$ but a suitable localization of $\Z$,
we only state it over the rationals because we did not define (and do not use) the Shimura varieties over rings of integers in this paper.

\begin{proposition} \label{M_1^B(n)_is_a_fine_moduli}
Let $\mathfrak{n}$ be a prime-to-$\Delta$ integral ideal of $F$ satisfying that $N(\mathfrak{n}) \ge 4^d$,
where $N(\mathfrak{n})$ is the absolute norm of $\mathfrak{n}$.
Then the algebraic stack $\mathscr{M}_1^B(\mathfrak{n})$ is represented by a scheme.
I.e., the Shimura variety $M_1^B(\mathfrak{n})$ is the fine moduli scheme of it.
\end{proposition}

\begin{proof}
As we have noted in \cref{section:quaternionic_Shimura_varieties}, in any case the coarse moduli space of $\mathscr{M}_1^B(\mathfrak{n})$ is a scheme, and is $M_1^B(\mathfrak{n})$.
Thus in order to show the statement it suffices to show that the algebraic stack $\mathscr{M}_1^B(\mathfrak{n})$
is an algebraic space.
Let $k$ be an algebraically closed field of characteristic zero, $X$ an object of $\mathscr{M}_1^B(\mathfrak{n})(k)$,
and $A$ the underlying QM-abelian variety of $X$.
It suffices to show that the automorphism group $\Aut X$ is trivial.
Let $f$ be an automorphism of $X$, and suppose that $f \ne 1$.
By \cref{isogeny} the map $f - 1$ is again an isogeny commuting with the action of $B$.
Hence $c := (f-1)^t \circ (f-1) \in F$ by \cref{dual_isogeny}.
Since $f$ preserves the polarization we have $f^t = f^{-1}$.
Thus $c = (f-1)^t \circ (f-1) = 2 - (f^t + f)$ and $c \in \End_{O_B} A$, and hence $c \in O_F$.
Write $a := 2 - c$.
Then $f$ satisfies the equation $x^2 - a x + 1 = 0$ in $\End^0_B A$.
Since $f$ has finite order by \cite[Section 21, Theorem 5]{MumfordAV},
we obtain that $a^2 - 4$ is zero or totally negative in $F$,
i.e., for every infinite place $v$ of $F$ we have that $| a |_v \le 2$.

Since $f$ fixes a level $\Gamma_1(\mathfrak{n})$-structure,
by exactly the same way as in the proof of \cite[Lemma 2.2]{Buzzard} we have that $c \in \mathfrak{n}$.
Hence $|N_{F/\Q}(c)| \ge N(\mathfrak{n})$.
Now since $|c|_v \le 4$ for every infinite place $v$ we have $|N_{F/\Q}(c)| \le 4^d$,
which shows the statement for $N(\mathfrak{n}) > 4^d$.
Finally if $N(\mathfrak{n}) = 4^d$ then $|N_{F/\Q}(c)| = 4^d$, and hence $a = -2$.
It follows that $f$ satisfies $(f+1)^2 = 0$.
Since its square is zero, the map $f + 1$ must not be an isogeny, and thus is zero itself
by \cref{isogeny}.
This implies that $f = -1$, and, since now $f$ fixes a level $\Gamma_1(\mathfrak{n})$-structure, we have that $2 \in \mathfrak{n}$.
This contradicts to the assumption that $N(\mathfrak{n}) \ge 4^d$, and therefore it completes the proof of the triviality of $\Aut X$ for $N(\mathfrak{n}) \ge 4^d$.
\end{proof}

For a modular curve $X$ and a point $x$ on $X$, we have that if $x$ is defined over a field $k$,
then there exists an elliptic curve with a level over $k$ corresponding to $x$.
In the case of Shimura varieties, in general this is false.
However, if the base field $k$ splits $B$, then the corresponding result holds also for our Shimura varieties.
The following proof is essentially due to Jordan \cite[Theorem 1.1]{Jordan}, who shows the case of $F = \Q$.

\begin{theorem} \label{thm:field_of_moduli}
Let $k$ be a field of characteristic zero containing the Galois closure of $F/\Q$.
Let $x \in M^B(k)$ and $X \in \mathscr{M}^B(\kbar)$ be a model of $x$.
Then $X$ descends to $k$ if and only if $B \otimes_\Q k \isom M_2(F \otimes_\Q k)$.
\end{theorem}

\begin{proof}
The “only if" part is shown in \cref{lf_iff_matrices}.
Assuming that $B \otimes_\Q k \isom M_2(F \otimes_\Q k)$ we show that $X$ descends to $k$.
Fix an isomorphism $B \otimes_\Q k \isom M_2(F \otimes_\Q k)$.
Take a finite Galois extension $L$ of $k$ so that $X$ has a model $(A,i,\lambda)$ over $L$,
and that for every $\sigma \in \Gal(L/k)$ one has that every isomorphism from $X$ to $X^\sigma$
is defined over $L$.
(Such $L$ does exist since if $X$ has a model over $L$, then for every Galois extension $L'$ of $L$
and for every $\tau \in \Gal(L'/L)$, we have that $X^\tau \isom X$ over $L'$.)
Since we are assuming that $k$ contains the Galois closure of $F/\Q$, we have that $F \otimes_\Q k \isom k^d$,
and thus we have an isomorphism of $\Gal(L/k)$-modules $F \otimes_\Q L \isom L^d$.

For every $\sigma \in \Gal(L/k)$ fix an isomorphism $\lambda_\sigma \colon X \to X^\sigma$.
Using them define the action of $\Gal(L/k)$ on $\Aut X/L$ as in \cref{descent}.
The cohomology class of the family of automorphisms $\alpha_{\sigma, \tau} := \lambda_\sigma^{-1} (\lambda_\tau^\sigma)^{-1} \lambda_{\tau \sigma}$ is the obstruction to descent of $X$ from $L$ to $k$ as in \cref{descent}.
In this proof we show that changing the base from $L$ to $\kbar$, the obstruction is trivial.
If it is true, then $X$ descends to $k$, and we will finish the proof.

Define the action of $\Gal(L/k)$ on $\End_{B \otimes L} \Lie (A/L)$ from right using $\{ \lambda_\sigma \}$ similarly, and let $\Gal(L/k)$ act on $B \otimes_\Q L$ by $(b \otimes a, \sigma) \mapsto b \otimes a^\sigma$.
By \cref{lf_iff_matrices} the canonical map $ r_{X/L} \colon B \otimes_\Q L \to \End_{F \otimes L} \Lie (A/L)$ is an isomorphism, and hence taking their centralizers of $B \otimes L$, this induces an isomorphism $F \otimes L \to \End_{B \otimes L} \Lie (A/L)$.
This preserves the action of $\Gal(L/k)$:
For, since $\lambda_\sigma$ is a morphism of $*$-polarized QM-abelian varieties,
for every $b \in O_B$, the following diagram commutes:
\begin{equation*}
\begin{aligned}
\xymatrix{
A \ar[r]^{\lambda_\sigma} \ar[d]_{i(b)} & A^\sigma \ar[d]^{i(b)^\sigma} \\
A \ar[r]^{\lambda_\sigma} & A^\sigma.
}
\end{aligned}
\end{equation*}
On the other hand, since $L$ acts on $\Lie (A/L)$ as the scalar multiplication, it follows that the isomorphism $F \otimes L \to \End_{B \otimes L} \Lie (A/L)$ preserves the action of $\Gal(L/k)$.
Obviously the map $\Aut X/L \to \Aut_{B \otimes L} \Lie (A/L)$ also commutes with $\Gal(L/k)$.

Let $e_1$ and $e_2$ be the elements of $B \otimes k$ corresponding to
\begin{equation*}
\left( \begin{matrix}1 & 0 \\ 0 & 0 \end{matrix} \right) \text{ and }
\left( \begin{matrix}0 & 0 \\ 0 & 1 \end{matrix} \right).
\end{equation*}
under our fixed isomorphism $B \otimes_\Q k \isom M_2(F \otimes_\Q k)$,
and let $e_i \otimes 1$ be the scalar extension of $e_i \in B \otimes k$ to $B \otimes L$.
For each $i$ denote the kernel of $r_{X/L}(e_i \otimes 1)$ in $\Lie (A/L)$ by $T_{X/L,i}$.
Then by the definition of $e_i$, we get that $\Lie (A/L) = T_{X/L,1} \oplus T_{X/L,2}$ as $F \otimes L$-modules,
and each direct summand is free of rank one.
For every $\sigma \in \Gal(L/k)$ we have 
\begin{equation*}
T_{X^\sigma/L,i} = T_{X/L,i} \otimes_{L, \sigma} L.
\end{equation*}
Hence if $\Lie (\lambda_\sigma)$ denotes the map $\Lie (A/L) \to \Lie (A^\sigma/L)$ induced by $\lambda_\sigma$, then
\begin{equation*}
\Lie (\lambda_\sigma) ( T_{X/L,i} ) = T_{X/L,i} \otimes_{L, \sigma} L.
\end{equation*}
Thus for every nonzero $\omega \in T_{X/L,i}$ and for every $\sigma \in \Gal(L/k)$,
there exists a unique $c_\sigma \in (F \otimes L)^*$ such that $\Lie (\lambda_\sigma)(\omega) = c_\sigma \omega^\sigma$.
We have seen that the canonical isomorphism $(F \otimes L)^* \to \Aut_{B \otimes L} \Lie (A/L)$ commutes with the action of $\Gal(L/k)$.
Write the canonical map $\Aut X/L \to \Aut_{B \otimes L} \Lie (A/L) = (F \otimes L)^*$ as $\xi$,
which is injective since now the characteristic is zero.
Then one gets $\xi(\alpha_{\sigma, \tau}) = c_\sigma^{-1} (c_\tau^\sigma)^{-1} c_{\tau \sigma}$.
In particular the family $(\xi(\alpha_{\sigma, \tau}))_{\sigma, \tau}$ is a 2-coboundary.

We know that $\Aut X/L$ is finite cyclic by \cref{Aut:nonCM,Aut:CM}.
From this fact, elementary we can find an endomorphism $\pi$ of the $\Gal(\kbar/k)$-module $\Aut_{B \otimes \kbar} \Lie(A_\kbar/\kbar)$ making the following sequence exact
\begin{equation*}
\begin{aligned}
\xymatrix{
0 \ar[r]  & \Aut X_\kbar/\kbar \ar[r]^-\xi & \Aut_{B \otimes \kbar} \Lie(A_\kbar/\kbar) \ar[r]^\pi & \Aut_{B \otimes \kbar} \Lie(A_\kbar/\kbar) \ar[r] & 0.
}
\end{aligned}
\end{equation*}
(See \cref{construction_of_pi} below.)
Since $\Aut_{B \otimes \kbar} \Lie(A_\kbar/\kbar) \isom (\kbar^*)^d$ as $\Gal(\kbar/k)$-modules, by Hilbert 90 we have that $H^1(\Aut_{B \otimes L} \Lie(A_\kbar/\kbar)) = 0$.
Thus by the cohomology long exact sequence induced by the above sequence, the cohomology map $\xi \colon H^2(\Aut A_\kbar/\kbar) \to H^2(\Aut_{B \otimes \kbar} \Lie(A_\kbar/\kbar))$ is injective.
As the family $(\xi(\alpha_{\sigma, \tau})_\kbar)$ is a 2-coboundary, it follows that the cohomology class of $(\alpha_{\sigma, \tau})_\kbar$ is trivial.
This is what we wanted to show, and we finish the proof.
\end{proof}

\begin{lemma} \label{construction_of_pi}
Let $d$ be a positive integer, $k$ a field, $\kbar$ a separable closure of $k$, $M$ a finite cyclic group on which $\Gal(\kbar/k)$ acts whose order is invertible in $k$, and let $\xi \colon M \to (\kbar^*)^d$ be an injection of $\Gal(\kbar/k)$-modules.
Then there exists an endomorphism $\pi$ of $\Gal(\kbar/k)$-module $(\kbar^*)^d$ which makes the following sequence exact:
\begin{equation*}
\begin{aligned}
\xymatrix{
0 \ar[r]  & M \ar[r]^-\xi & (\kbar^*)^d \ar[r]^\pi & (\kbar^*)^d \ar[r] & 0.
}
\end{aligned}
\end{equation*}
\end{lemma}

\begin{proof}
If for a sub $\Gal(\kbar/k)$-module $M'$ of $M$ and for $\xi|_{M'} \colon M' \to (\kbar^*)^d$ there exists an endomorphism $\pi'$ as in the statement, and if for $\pi' \circ \xi \colon M/M' \to (\kbar^*)^d$ there exists also an endomorphism $\pi''$ as in the statement, then the map $\pi'' \circ \pi'$ is a desired endomorphism for $M$.
Hence by induction on the order of $M$, we may assume that the order of $M$ is a prime $p$.

Let $u$ be a generator of $M$, $\zeta$ a primitive $p$-th root of unity in $\kbar$,
and write $\xi(u) = (\zeta^{a_i})_i$ for nonnegative integers $a_i$.
Rearranging the indices if necessary, we may assume that $a_1$ is relatively prime to $p$, and replacing $\zeta$ by $\zeta^{a_1}$ if necessary, we may assume that $a_1 = 1$.
In this case the map $\pi(x_i) = (x_1^p, x_1^{a_2} x_2^{-1}, \dots, x_1^{a_d}x_d^{-1})$ is a desired one.
\end{proof}

From this theorem we can show similar statements about the rational points of Shimura varieties with some levels.
For example we show a statement about $M^B_0(p)$ for a quadratic number field $F$.

\begin{corollary} \label{thm:field_of_moduli:Gamma_0}
Assume that $F$ is a quadratic number field.
Let $\mathfrak{n}$ be a prime-to-$\Delta$ integral ideal of $F$ and $k$ a field of characteristic zero containing $F$.
Let $x \in M^B_0(\mathfrak{n})(k)$ and $X \in \mathscr{M}^B_0(\mathfrak{n})(\kbar)$ be a model of $x$.
Assume the following three conditions:
\begin{enumerate}
\item The automorphism group $\Aut_\kbar X$ is not isomorphic to $\{ \pm 1 \}$.
\label{thm:field_of_moduli:Gamma_0:con1}

\item The algebra $B \otimes_\Q k$ is isomorphic to $M_2(F \otimes_\Q k)$.
\label{thm:field_of_moduli:Gamma_0:con2}

\item Either that $F(\sqrt{-1})$ does not split $B$, or that $F$ is not equal to $\Q(\sqrt{2})$ nor $\Q(\sqrt{3})$.
\label{thm:field_of_moduli:Gamma_0:con3}
\end{enumerate}
Then $X$ descends to $k$.
\end{corollary}

\begin{proof}
Let $X'$ be the image of $X$ under the canonical map $\mathscr{M}^B_0(\mathfrak{n}) \to \mathscr{M}^B$.
Then by the condition \eqref{thm:field_of_moduli:Gamma_0:con2} and by \cref{thm:field_of_moduli},
the model $X'$ descends to $k$.
We show that a model $X$ of $x$ over $\kbar$ also descends to $k$.
The canonical injection $\Aut_\kbar X \to \Aut_\kbar X'$ induces the map of the cohomology groups
$H^2(G_k, \Aut_\kbar X) \to H^2(G_k, \Aut_\kbar X')$.
The image of the obstruction $\operatorname{ob}(X, \kbar/k)$ to descent of $X$ from $\kbar$ to $k$ as in \cref{descent} under this map is $1$ since $X'$ descends to $k$.
It suffices to show that $\Aut_\kbar X$ is a direct summand of $\Aut_\kbar X'$ under the injection above:
For, in this case, the map of the cohomology groups is injective, and hence $X$ descends to $k$.
By \cref{Aut:nonCM,Aut:CM}, there exists a quadratic extension $F'$ of $F$ such that
the automorphism group $\Aut_\kbar X'$ is isomorphic to a subgroup of $\mu(F')$, the group of the roots of unity in $F'$,
hence is isomorphic to $\Z/m$, where $m$ is an even positive integer dividing $n_{\lcm}$.
Since $\Aut X'$ is larger than $\{ \pm  1\}$ by the condition \eqref{thm:field_of_moduli:Gamma_0:con1} and by the fact that the characteristic of $k$ is zero, this field $F'$ splits $B$ by \cref{Aut:nonCM,End:ch0:CM}.
Since now we are assuming that $F$ is quadratic, we have that $2 \le m \le 12$.
Thus if $\Aut_\kbar X$ is not a direct summand, then by the condition \eqref{thm:field_of_moduli:Gamma_0:con1},
we have either that $\Aut_\kbar X \isom \Z/4$ and $\Aut_\kbar X' \isom \Z/8$,
or that $\Aut_\kbar X \isom \Z/6$ and $\Aut_\kbar X' \isom \Z/12$.
In particular $F'$ contains $\zeta_8$ or $\zeta_{12}$, where $\zeta_i$ is a primitive $i$-th root of unit.
Since $F$ is real, we have that $F' = F(\zeta_8)$ or $F' = F(\zeta_{12})$.
Again by the hypothesis that $F$ is quadratic, we have that $F = \Q(\sqrt{2})$ or $\Q(\sqrt{3})$,
and hence $F' = F(\sqrt{-1})$.
This contradicts to the condition \eqref{thm:field_of_moduli:Gamma_0:con3}, hence $X$ also descends to $k$.
\end{proof}

\section{Characters of QM abelian varieties with level-$\Gamma_0$ structures over local fields} \label{section:local}

From now we define and study the Galois characters induced by level-$\Gamma_0$ structures of QM-abelian varieties, which are the main objects in this paper.
Before defining them, we translate our level structures into more convenient forms, as in the case of modular curves.
Recall the notation of \cref{section:quaternionic_Shimura_varieties}, in particular $\Lambda$ is a free left $O_B$-module of rank one, and $\psi$ is a nondegenerate alternating $\Z$-bilinear pairing $\Lambda \times \Lambda \to \Z$ satisfying that $\psi(bx,y) = \psi(x,b^*y)$ for all $x,y \in \Lambda$ and for all $b \in O_B$.

\begin{definition} \label{def:subgroup_of_type_Gamma_0(p)}
Let $\mathfrak{n}$ be a prime-to-$\Delta$ integral ideal of $F$, $k$ a field on which $N_{F/\Q}(\mathfrak{n})$ is invertible, and $A$ a QM-abelian variety over $k$.
We define a subgroup $k$-scheme of $A$ of type $\Gamma_0(\mathfrak{n})$ as a subgroup $k$-scheme $C$ of $A$ which is stable under the action of $O_B$,
and such that the group $C(\kbar)$ is isomorphic to $(O_F/\mathfrak{n})^2$ as $O_F$-modules.
\end{definition}

\begin{lemma} \label{A[n]/O_B/n}
Let $\mathfrak{n}$ be an integral ideal of $F$ whose prime factors are all unramified, $k$ a field on which $N_{F/\Q}(\mathfrak{n})$ is invertible, and $A$ an abelian variety of dimension divided by $2d$ over $k$ on which $O_B$ acts.
Let $A[\mathfrak{n}] := \cap_{a \in \mathfrak{n}} \ker(a \text{ on } A)$.
Then $A[\mathfrak{n}](\kbar)$ is a free $O_B/\mathfrak{n}$-module.
\end{lemma}

\begin{proof}
By the Chinese remainder theorem, without loss of generality we can assume that $\mathfrak{n} = \mathfrak{p}^m$ for a positive integer $m$ and for a unramified prime $\mathfrak{p}$ of $F$ above $p \ne \ch k$.
In this case by \cref{T_l/O_B_l}, the $O_{B,p}$-module $T_p A$ is free.
Thus dividing them by $\mathfrak{p}^m$, since $\mathfrak{p}$ is unramified, we obtain the freeness of the $O_B/\mathfrak{p}^m$-module $(T_p A / p^m) / \mathfrak{p}^m = A[p^m](\kbar)/\mathfrak{p}^m$, which is equal to $A[\mathfrak{p}^m]$ again by the Chinese remainder theorem.
Thus the result.
\end{proof}

\begin{lemma} \label{direct_summand}
Keep the notations as in \cref{def:subgroup_of_type_Gamma_0(p)}.
Let $C$ be a subgroup $k$-scheme of $A$ of type $\Gamma_0(\mathfrak{n})$.
Then $C(\kbar)$ is a direct summand of $A[\mathfrak{n}](\kbar)$ as an $O_B$-module.
\end{lemma}

\begin{proof}
As in the same manner as \cref{A[n]/O_B/n}, we may assume that $\mathfrak{n} = \mathfrak{p}^m$.
Let $M = A[\mathfrak{p}^m](\kbar)$ and $N = C(\kbar)$.
Then since $O_B/\mathfrak{p}$ is a semisimple algebra, there exists a projective $O_B/\mathfrak{p}^m$-module $W$ and a map $N \oplus W \to M$ inducing $(N \oplus W)/\mathfrak{p} \isom M/\mathfrak{p}$.
Since the target $M$ is free over $O_B/\mathfrak{p}^m$ by \cref{A[n]/O_B/n}, in particular is free over $O_F/\mathfrak{p}^m$, by Nakayama's lemma we obtain that the map $N \oplus W \to M$ is an isomorphism.
\end{proof}

Note that in \cref{def:subgroup_of_type_Gamma_0(p),A[n]/O_B/n,direct_summand} we only need $O_B$, but do not need involutions $*$ of $B$, pairings $\psi$, nor lattices $\Lambda$.

\begin{proposition} \label{level_Gamma_0}
Let $\mathfrak{n}$ be a prime-to-$[\Lambda^\vee : \Lambda]\Delta$ integral ideal of $F$, $k$ a field of characteristic zero, and $(A, i, \lambda)$ a $*$-polarized QM-abelian variety over $k$.
Suppose that there exists a symplectic isomorphism $\Lambda \otimes \hat{\Z} \to \hat{T} A_\kbar$.
Then there exists a canonical bijection between the set of level-$\Gamma_0(\mathfrak{n})$ structures on $A$
and the set of subgroup $k$-schemes of $A$ of type $\Gamma_0(\mathfrak{n})$.
\end{proposition}

\begin{proof}
Take a prime-to-$[\Lambda^\vee : \Lambda]\Delta$ rational integer $n$ divided by $\mathfrak{n}$.
By \cref{level-K_structures}, the set of level-$\Gamma_0(\mathfrak{n})$ structures on $A$ is equal to
the set of the $k$-rational points of the $k$-scheme $\Isom_{O_B}( (O_B/n)_k, A[n] ) / \Gamma'$, and hence is equal to the one of $\Isom_{O_B}( (O_B/\mathfrak{n})_k, A[\mathfrak{n}] ) / \Gamma$,
where $\Gamma' \le ((O_B/n)^*)^\opp$ is the preimage of
\begin{equation*}
\Gamma = \left\{ \left( \begin{matrix} a & b \\ 0 & d \end{matrix} \right) \bigg | a,d \in (O_F/\mathfrak{n})^*, b \in O_F/\mathfrak{n} \right\} \le \GL_2(O_F/\mathfrak{n})^\opp
\end{equation*}
under the canonical map $O_B/n \to O_B/\mathfrak{n} \isom M_2(O_F/\mathfrak{n})$.
We identify them.

Let $\alpha$ be a level-$\Gamma_0(\mathfrak{n})$ structure and $\beta \colon O_B/\mathfrak{n} \to A[\mathfrak{n}]_\kbar$ an isomorphism satisfying that
$\alpha_\kbar = \beta \Gamma$.
Let $C$ be the image of the subgroup $\left( \begin{matrix} 0 & * \\ 0 & * \end{matrix} \right)$ under $\beta$.
Then this is independent of the choice of $\beta$,
is defined over $k$, and has the desired properties.

Conversely let $C$ be a subgroup $k$-scheme of $A$ of type $\Gamma_0(\mathfrak{n})$.
Then $C(\kbar)$ is isomorphic to $(O_F/\mathfrak{n})^2$ as $O_B/\mathfrak{n}$-modules, where on the latter group $O_B/\mathfrak{n}$ acts as matrices.
Choose an isomorphism between them.
By \cref{direct_summand} the group $C(\kbar)$ is a direct summand of $A[\mathfrak{n}](\kbar)$, which is free of rank one over $O_B/\mathfrak{n}$ by \cref{A[n]/O_B/n}.
Therefore choosing another direct summand $C'$ of $A[\mathfrak{n}](\kbar)$ so that $A[\mathfrak{n}](\kbar) = C \oplus C'$
and choosing an isomorphism between $C'$ and $(O_F/\mathfrak{n})^2$,
we obtain an isomorphism $\beta$ between $O_B/\mathfrak{n}$ and $A[\mathfrak{n}](\kbar)$.
We can show that after taking modulo $\Gamma$, it is independent of the choices above
and is defined over $k$.
\end{proof}

\begin{remark} \label{rmk:level_Gamma_0}
\begin{enumerate}
\item[]
\item \label{rmk:level_Gamma_0:Gamma_1}
In the situation of \cref{level_Gamma_0}, if a level-$\Gamma_0(\mathfrak{n})$ structure on $A$ comes from a level-$\Gamma_1(\mathfrak{n})$ structure on $A$, then the corresponding subgroup $k$-scheme of $A$ of type $\Gamma_0(\mathfrak{n})$ is constant.
\item \label{rmk:level_Gamma_0:weak}
\cref{level_Gamma_0} fails for $\mathfrak{n}$ which is relatively prime to $\Delta$ but not so to $[\Lambda^\vee : \Lambda]$.
For that $\mathfrak{n}$, the set of level-$\Gamma_0(\mathfrak{n})$ structures corresponds to a subset of the set of subgroup $k$-schemes of $A$ of type $\Gamma_0(\mathfrak{n})$.
The point is that, the set of level-$\Gamma_0(\mathfrak{n})$ structures is, by \cref{level-n_structures}, a subset of the set of the $k$-rational points of the $k$-scheme $\Isom_{O_B}( (O_B/\mathfrak{n})_k, A[\mathfrak{n}] ) / \Gamma$, which corresponds to the set of subgroup $k$-schemes of $A$ of type $\Gamma_0(\mathfrak{n})$ by completely the same arguments as in the proof of \cref{level_Gamma_0}, where $\Gamma$ is as in the proof of \cref{level_Gamma_0}.
\end{enumerate}
\end{remark}

Let $\mathfrak{n}$ be a prime-to-$\Delta$ integral ideal of $F$, $k$ a field on which $N_{F/\Q}(\mathfrak{n})$ is invertible,
and $A$ be a QM-abelian variety over $k$.
Then the modulo $\mathfrak{n}$ representation gives $G_k \to \Aut_{O_B} A[\mathfrak{n}](\kbar)$.
By \cref{A[n]/O_B/n} the latter group is isomorphic to $((O_B / \mathfrak{n})^*)^\opp$,
which is isomorphic to $\GL_2 (O_F / \mathfrak{n})^\opp$ since now we are assuming that $\mathfrak{n}$ and $\Delta$ are relatively prime.
Now let $C$ be a subgroup $k$-scheme of $A$ of type $\Gamma_0(\mathfrak{n})$ as in \cref{def:subgroup_of_type_Gamma_0(p)}.
Then $C(\kbar)$ is a direct summand of $A[\mathfrak{n}](\kbar)$ as an $O_B$-module by \cref{direct_summand}.
Since $C$ is a $k$-group scheme, the Galois group $G_k$ acts on $C(\kbar)$.
Hence we obtain $G_k \to \Aut_{O_B} C(\kbar)$, whose target is canonically isomorphic to $(O_F/\mathfrak{n})^*$.
We denote the obtained character by $\nu \colon G_k \to (O_F/\mathfrak{n})^*$, and call it the $\Gamma_0$-character (induced by $A$ and $C$).
Hence choosing a suitable isomorphism $O_B/\mathfrak{n} \isom M_2(O_F/\mathfrak{n})$ and identifying them,
the mod-$\mathfrak{n}$ representation $G_k \to \GL_2(O_F/\mathfrak{n})$ factors through the subgroup
\begin{equation*}
\left\{ \left( \begin{matrix} a & b \\ 0 & d \end{matrix} \right) \right\} \subseteq \GL_2(O_F/\mathfrak{n}),
\end{equation*}
where the upper left entry is the character $\nu$.
This $C$ is a constant group if and only if $\nu$ is the trivial character.
In the rest of this paper we study this character, especially for a QM-abelian variety over a number field.

First in order to study its local behavior, we show statements about the reductions of QM abelian varieties over local fields.
Let $k$ be a local field, $v$ the normalized valuation on $k$, $O_k$ the ring of integers,
and $\F_v$ be its residue field of characteristic $\ell$.
Fix a separable closure $\ksep$ of $k$ and an extension of $v$ on it.
Denote $I_k$ the inertia group of the absolute Galois group $G_k$.
Let $A$ be a QM-abelian variety over $k$.

For a rational prime $\ell' \ne \ell$, let
$R_{\ell'} \colon G_k \to \Aut_{O_B} T_{\ell'} A$ be the $\ell'$-adic representation.
Since $A$ has potentially good reduction \cite[III]{BreenLabesse}, we have that $R_{\ell'} (I_k)$ is a finite group
\cite{SerreTate}, and its kernel on $I_k$ is independent of the choice of $\ell'$.
We denote it by $N$ and define $\Phi(A/k) = I_k/N$.

Using the theory of Serre--Tate \cite{SerreTate}, we can show the following proposition.
The proof is completely the same as in \cite[Proposition 3.2]{Jordan}.
Recall that $n_{\lcm} = \lcm \{ m : [F(\zeta_m) : F] \le 2 \}$, where $\zeta_m$ is a primitive $m$-th roof of $1$, and it is divisible by $12$.

\begin{proposition} \label{good_reduction}
There exists a totally ramified extension $k'/k$ of degree $\# \Phi(A/k)$ over which
$A \times_k k'$ has good reduction.
Moreover $k' k^\nr / k^\nr$ is Galois with the Galois group isomorphic to $\Phi(A/k)$, where $k^\nr$ is the maximal unramified extension of $k$.
\end{proposition}

Take $k'$ as in the proposition.
Then the residue field of $k'$ is $\F_v$, the same as the one of $k$.
Let $A'$ be the model of $A \times_k k'$ over the ring of integers $O_{k'}$,
which is a QM-abelian scheme.

As in \cite{Jordan}, we can show the following:

\begin{proposition} \label{Phi_injects_Aut}
The group $\Phi(A/k)$ is isomorphic to a subgroup of $\Aut_B (A'/\F_v)$.
\end{proposition}

With this proposition and the classification of the endomorphism rings of QM-abelian varieties over finite fields,
we classify the group $\Phi(A/k)$.

\begin{proposition} \label{Aut:local}
The order of the group $\Phi(A/k)$ divides $20n_{\lcm}$ and every element $f \in \Phi(A/k)$ has order dividing $n_{\lcm}$.
Further if $\ell \nmid 5 n_{\lcm}$ then $\Phi(A/k)$ is cyclic.
\end{proposition}

\begin{proof}
The first statement follows from \cref{Aut:nonCM,Aut:CM,Phi_injects_Aut} and from the fact that $\Phi(A/k)$ is finite.
For the second statement, assume $\ell \nmid 5 n_{\lcm}$.
If $d = 1$ then the proposition is due to \cite[Proposition 3.4]{Jordan}.
Hence assume $d \ge 2$.
Let $I_\ell$ be the pro-$\ell$-part of the inertia group of $I_k$,
and $I_t = I_k / I_\ell$ the tame ramification group of $k$.
Now by the assumption our $\ell$ does not divide $20n_{\lcm}$.
Therefore the canonical surjection $I_k \to \Phi(A/k)$ induces a surjection $I_t \to \Phi(A/k)$.
Since every finite quotient of $I_t$ is cyclic, we have the statement.
\end{proof}

From these propositions it is easy to study the local behavior of the $\Gamma_0$-characters.
Let $p_F$ be a prime of $F$ not dividing $\Delta$ and $p$ the rational prime below $p_F$.
Assume that $p$ is invertible in $k$.
Let $C$ be a subgroup $k$-scheme of $A$ of type $\Gamma_0(p_F)$ as in \cref{def:subgroup_of_type_Gamma_0(p)} and
$\nu \colon G_k \to \F_{p_F}^*$ be the $\Gamma_0$-character obtained from $A$ and $C$.
First we study the ramification of $\nu$ in the case of $p \ne \ell$.

\begin{lemma} \label{nu_nr}
If $p \ne \ell$ then the character $\nu^{n_{\lcm}}$ is unramified.
\end{lemma}

\begin{proof}
If $A/k$ has good reduction, then since the inertia group $I_k$ acts on $A[p](\kbar)$ trivially, the proposition follows.
Assume that $A/k$ has bad reduction.
By \cref{good_reduction,Aut:local}, there exists a tamely ramified Galois extension $M$ of $k^\nr$ whose exponent divides $n_{\lcm}$ and such that
$A_M$ has good reduction over $M$.
By the argument above we have that $\nu |_{G_M} = 1$,
and hence $\nu|_{I_k}^{n_{\lcm}} = 1$.
\end{proof}

Next we study the image of the Frobenius automorphism of $k$ under $\nu^{n_{\lcm}}$ in the case that $\ell \ne p$.

We prepare some symbols.
For a prime $q$ and a positive integer $f$, let
\begin{equation*}
\begin{aligned}
FR(q^f) = \left\{ \beta \in \Qbar \bigg | 
\begin{array}{c}
\beta^2 + b \beta + q^f = 0 \text{ for some } b \in O_F \text{ with } \\
| b |_v \le 2 \sqrt{q}^f \text{ for every infinite place } v \text{ of } F
\end{array}
\right\}.
\end{aligned}
\end{equation*}
Note that for every $\beta \in FR(q^f)$ and for every infinite place $v$ of $F(\beta)$,
the absolute value $| \beta |_v$ is equal to $\sqrt{q}^f$.
Since $O_F \subseteq F \otimes \R$ is discrete, the set $FR(q^f)$ is finite.

\begin{lemma} \label{F(beta)_is_totally_imaginary}
If $q$ is unramified in $F$ and if $f$ is odd, then for any $\beta \in FR(q^f)$, the fields $F(\beta)$ are
totally imaginary quadratic extensions of $F$.
\end{lemma}

\begin{proof}
Assume that $\beta^2 + b \beta + q^f = 0$ for $b \in O_F$ as in the definition of $FR(q^f)$.
Then $b$ satisfies that $| b |_v < 2 \sqrt{q}^f$ for every infinite place $v$,
for, if $b$ satisfied $ | b |_v = 2 \sqrt{q}^f$ for some $v$, then since $f$ is odd $F$ would contain $\sqrt{q}$,
which is a contradiction.
Hence $b^2 - 4 q^f$ is a totally negative element of $F$, and so $F(\beta)$ is totally imaginary.
\end{proof}

For a local field or a finite field $L$ we denote a (or the) Frobenius element of $G_L$ by $\Frob_L$.

First we consider the case of finite fields.

\begin{lemma} \label{nu_q_equals_to_beta:fin}
Let $p_F$ be a prime of $F$ not dividing $\Delta$, $p$ the rational prime below $p_F$, $\ell$ a rational prime different from $p$ and $f$ be a positive integer.
Let $A$ a QM-abelian variety over $\F_{\ell^f}$, $C$ a subgroup $\F_{\ell^f}$-scheme of $A$ of type $\Gamma_0(p_F)$ as in \cref{def:subgroup_of_type_Gamma_0(p)},
and $\nu \colon G_{\F_{\ell^f}} \to \F_{p_F}^*$ be the $\Gamma_0$-character induced by $A$ and $C$.
Then there exists $\beta \in FR(\ell^f)$ and a prime $\mathfrak{r}$ of $F(\beta)$ above $p_F$ such that $\nu(\Frob_{\F_{\ell^f}}) \equiv \beta \mod \mathfrak{r}$.
\end{lemma}

\begin{proof}
Let $\pi$ be the relative Frobenius endomorphism of $A/\F_{\ell^f}$.
The canonical map $O_F[\pi] \to \End_{O_B} C(\overline{\F_{\ell^f}}) \isom \F_{p_F}$, where the first map associates to $\pi$ the endomorphism $\pi|_C$, induces $O_F[\pi]/p_F \to \F_{p_F}$.
The image of $\pi$ under it is $\nu(\Frob_{\F_{\ell^f}})$.
Thus for a prime $\mathfrak{r}$ of the algebraic number field $F(\pi)$ dividing $p_F$
we obtain $\nu(\Frob_{\F_{\ell^f}}) = \pi \mod \mathfrak{r}$.
Therefore by \cref{Weil_number} we get the claim.
\end{proof}

Return to the case of local fields.

\begin{lemma} \label{nu_q_equals_to_beta:local}
Let $p_F$ be a prime of $F$ not dividing $\Delta$ and $p$ the rational prime below $p_F$.
Let $k$ be a local field with the normalized valuation $v$ whose residue characteristic is $\ell$, $A$ a QM-abelian variety over $k$, $C$ a subgroup $k$-scheme of $A$ of type $\Gamma_0(p_F)$, and $\nu \colon G_k \to \F_{p_F}^*$ be the $\Gamma_0$-character induced by $A$ and $C$.
Assume that $p \ne \ell$.
Let $f$ be the absolute inertia degree of $k$, i.e., the degree of the residue field of $k$ over $\F_\ell$.
Then there exists $\beta \in FR(\ell^f)$ and a prime $\mathfrak{r}$ of $F(\beta)$ above $p_F$ such that $\nu(\Frob_k) \equiv \beta \mod \mathfrak{r}$.
\end{lemma}

\begin{proof}
By \cref{good_reduction} there exists a totally ramified extension $M$ of $k$ such that the abelian variety $A_M$ has good reduction.
Since $M$ is totally ramified over $k$, the residue field of $M$ is equal to $\F_v$, the one of $k$.
We denote the reduction of $A_M$ to the residue field of $M$, which is $\F_v$, by $A_{\F_v}$.
Now again since the inertia degree of the extension $M/k$ is $1$, Frobenius elements $\Frob_M$ of $G_M$ are again Frobenius elements as elements of $G_k$.
Hence we have the equality $\nu(\Frob_k) = \nu |_{G_M}(\Frob_M)$.
Since $A_M$ has good reduction, the character $\nu|_{G_M}$ is unramified,
i.e., it induces
$G_{\F_v} \to \F_{p_F}^*$, which again we denote by the same symbol $\nu|_{G_M}$.
Hence $\nu(\Frob_k) = \nu|_{G_M}(\Frob_{\F_v})$.
Therefore by \cref{nu_q_equals_to_beta:fin} we obtain the result.
\end{proof}

Before studying the character obtained as above over a number field in the next section, we deduce theorems, which are ones of the main theorems of this paper, about the non-existence of the rational points of our Shimura varieties.

Let $\ell$ be a rational prime and $f$ be a positive integer.
We define two sets as follows:
\begin{equation*}
\begin{aligned}
W(\ell^f) & := \left\{ p_F \bigg | 
\text{ a prime of $F$ dividing } 
\ell \mathfrak{D}_{F/\Q} \prod_{\beta \in FR(\ell^f)} N_{F(\beta)/F}(1 - \beta) \right\}, \\
V(\ell^f) & := W(\ell^f) \cup \left\{ p_F \bigg | 
\text{ a prime of $F$ such that } 
N(p_F) < 4^d \right\}.
\end{aligned}
\end{equation*}
These sets only depend on $F, \ell$, and on $f$, and not on $B$.
Also since $FR(\ell^f)$ is finite, these are finite.

\begin{theorem} \label{no_lv_gamma_1_structures}
Let $k$ be a local field whose residue characteristic is $\ell$, $f$ be the absolute inertia degree of $k$, and let $A$ be a QM-abelian variety over $k$.
Then for every $p_F \not\in W(\ell^f) \cup \operatorname{Ram}(B/F)$, there are no constant subgroup $k$-schemes of $A$ of type $\Gamma_0(p_F)$.
\end{theorem}

Before proving the theorem, note that for a prime $p_F \not\in W(\ell^f) \cup \operatorname{Ram}(B/F)$, we have that $p_F$ is relatively prime to $\Delta$.
Moreover the rational prime $p$ below $p_F$ is not equal to $\ell$, and hence the characteristic of $k$, which is zero or $\ell$, is relatively prime to $p$.
Thus in this situation the assumptions of \cref{def:subgroup_of_type_Gamma_0(p)} are satisfied, and hence we can discuss subgroup $k$-schemes of $A$ of type $\Gamma_0(p_F)$.

\begin{proof}
Let $p_F \not\in W(\ell^f) \cup \operatorname{Ram}(B/F)$ be a prime of $F$ and $p$ be the rational prime below $p_F$, and suppose that there exists a constant subgroup $k$-scheme $C$ of $A$ of type $\Gamma_0(p_F)$.
Let $\nu \colon G_k \to \F_{p_F}^*$ be the $\Gamma_0$-character induced by $A$ and $C$.
Since $p_F \not\in W(\ell^f)$, the rational prime $p$ is not equal to $\ell$.
Thus by \cref{nu_q_equals_to_beta:local}, there exists $\beta \in FR(\ell^f)$ and a prime $\mathfrak{r}$ of $F(\beta)$ above $p_F$ such that $\nu(\Frob_k) \equiv \beta \mod \mathfrak{r}$.
Now since the character $\nu$ is trivial, it follows that $1 \equiv \beta \mod \mathfrak{r}$.
In particular, the prime $p_F$ divides $N_{F(\beta)/F} (1 - \beta)$.
Since now we are assuming that $p_F \not\in W(\ell^f)$, this leads the contradiction, thus the result.
\end{proof}

Note that the exceptional set $W(\ell^f) \cup \operatorname{Ram}(B/F)$ in \cref{no_lv_gamma_1_structures} does not depend on $k$, but only on the absolute inertia degree of $k$ (and on the residue characteristic of $k$).

Using the fact that $M_1^B(p_F)$ is the fine moduli for sufficiently large prime $p_F$ (\cref{M_1^B(n)_is_a_fine_moduli}), we can show the following obvious corollary about the non-existence of rational points of $M_1^B(p_F)$.

\begin{corollary} \label{M_1^B(p)(k)_is_empty}
Let $\ell$ be a rational prime and $k$ be a finite extension of $\Q_\ell$ whose inertia degree is $f$.
Then for every $p_F \not\in V(\ell^f) \cup \operatorname{Ram}(B/F)$ we have $M_1^B(p_F)(k) = \varnothing$.
\end{corollary}

\begin{proof}
Let $p_F \not\in V(\ell^f) \cup \operatorname{Ram}(B/F)$ a prime of $F$.
Since $N(p_F) \ge 4^d$, the Shimura variety $M_1^B(p_F)$ is the fine moduli by \cref{M_1^B(n)_is_a_fine_moduli}.
Thus if $M_1^B(p_F)$ had a rational point over $k$, then there would exist a QM-abelian variety $A$ and its level-$\Gamma_1(p_F)$ structure over $k$, and hence a QM-abelian variety $A$ would have a constant subgroup $k$-scheme of $A$ of type $\Gamma_0(p_F)$ by \cref{level_Gamma_0}
(see also \cref{rmk:level_Gamma_0} \eqref{rmk:level_Gamma_0:weak})
and by \cref{rmk:level_Gamma_0} \eqref{rmk:level_Gamma_0:Gamma_1}.
Therefore the existence of $k$-rational points of $M_1^B(p_F)$ contradicts \cref{no_lv_gamma_1_structures}.
\end{proof}

Since for prime-to-$\Delta$ integral ideals $\mathfrak{n} \mid \mathfrak{m}$ of $F$ there exists the canonical map $\mathscr{M}_1^B(\mathfrak{m}) \to \mathscr{M}_1^B(\mathfrak{n})$,
from this theorem obviously it follows, for every rational prime $p$ below a prime $p_F$ of $F$ not in $W(\ell^f) \cup \operatorname{Ram}(B/F)$, that $M_1^B(p)(k) = \varnothing$ under the same notations as in the theorem.
The proof of the theorem shows more:
For every rational prime $p$ not dividing $6 \ell \Delta$ nor $N_{F(\beta)/\Q}(1-\beta)$ for any $\beta \in FR(\ell^f)$, the Shimura variety $M_1^B(p)$ has no $k$-rational points.
The only point is that $M_1^B(p)$ is the fine moduli since $p \ge 4$.

\begin{corollary} \label{uniform_torsion_points_theorem}
Let $\ell$ be a rational prime and $f$ a positive integer.
Then there exists an integer $N$ depending only on $\ell$, $f$, and on $F$, but not on $B$, such that for every QM-abelian variety $A$ over a local field $k$ whose residue characteristic is $\ell$ and whose absolute inertia degree is $f$, if $A$ has a $k$-rational point of prime order $p$, then $p$ divides $N \Delta$.
\end{corollary}

\begin{proof}
Let $N$ be the product of $\ell$ and $N_{F(\beta)/\Q}(1-\beta)$ for all $\beta \in FR(\ell^f)$.
We show that this $N$ is a desired integer.
Let $k$ be a local field whose residue characteristic is $\ell$ and whose absolute inertia degree is $f$ and $A$ a QM-abelian variety over $k$.
Let $p$ be a rational prime not dividing $\ell\Delta$ and $P$ a $k$-rational point of $A$ of order $p$.
Then there exists a prime $p_F$ of $F$ dividing $p$ such that the image $C$ of the subgroup $O_B P$ of $A(\kbar)$ under the canonical projection $A[p] \to A[p_F]$ is constant subgroup of type $\Gamma_0(p_F)$.
Therefore by \cref{no_lv_gamma_1_structures} our prime $p_F$ belong to $W(\ell^f)$, in particular $p$ divides $N$.
\end{proof}

This corollary is analogous to Merel's theorem on the torsion points of elliptic curves \cite[Th\'eor\`eme]{Merel}.
While Merel's theorem requires a very deep theory of modular curves, our corollary does not require geometry nor arithmetic of Shimura varieties itself, but is derived from standard arguments on QM-abelian varieties.
This is due to the fact that our Shimura varieties have no cusps, i.e., QM-abelian varieties always have potentially good reduction, whereas we need to deal with the cusps in the case of modular curves.

Next we study the local behavior of $\nu$ in the case of $p = \ell$.

Before stating the lemma we define some symbols.
For a field $k$ and for a rational prime $p$ which is invertible in $k$, let $\chi \colon G_k \to \Aut \mu_p(\kbar) \isom \F_p^*$ be the mod $p$ cyclotomic character of $k$.
For a local field $L$ whose residue characteristic is $p$, let $I_{p, L}$ be the pro-$p$-part of $I_L$,
$I_{t,L} = I_L/I_{p,L}$ the tame ramification group of $L$,
and $\theta_{m, L} \colon I_{t,L} \to \mu_m(k^\nr)$ be the surjection defined in \cite[Section 1.3]{SerreGal} for a prime-to-$p$ positive integer $m$.
For a group $H$ of prime-to-$p$ order and for a map $\iota \colon I_L \to H$, it induces the character $I_{t,L} \to H$, which again we denote by the same symbol $\iota$.
Note that by \cite[Proposition 8]{SerreGal}, if $L$ is a finite extension of $\Q_p$ of ramification index $e$, then $\theta_{p-1, L}^e = \chi$.

Let $p_F$ be a prime of $F$ not dividing $\Delta$ and $p$ the rational prime below $p_F$.

\begin{lemma} \label{nu_above_p_local}
Let $k$ be a finite extension of $\Q_p$, $A$ a QM-abelian variety over $k$, $C$ a subgroup $k$-scheme of $A$ of type $\Gamma_0(p_F)$, and $\nu \colon G_k \to \F_{p_F}^*$ be the $\Gamma_0$-character induced by $A$ and $C$.
Write $\# \F_{p_F} = p^r$.
Assume that $n_{\lcm} < p$.
Fix an inclusion $\F_{p_F} \to \overline{\F_k}$ and identify $\mu_{p^r - 1}(k^\nr)$ with $\F_{p_F}^*$.
Then there exists an integer $b$ such that $\nu^{n_{\lcm}} |_{I_k} = \theta_{p^r-1, k}^b$ and that $b = \sum_{i=0}^{r-1} b_i p^i$, where $0 \le b_i \le n_{\lcm}$ for each $i$.
\end{lemma}

Note that $ p - 1 < (p^r+1)(p-1)/(p^r-1) \le p + 1$.
Since $n_{\lcm}$, which is divided by $12$, is not prime, we have that $n_{\lcm} < (p^r+1)(p-1)/(p^r-1)$ if and only if $n_{\lcm} < p$.
With this in mind, we show the statement.

\begin{proof}
First choose a field whose cardinality is $p^{2r}$ contained in $O_B/p_F = M_2(\F_{p_F})$, which we denote simply by $\F_{p^{2r}}$.
Then trivially $\F_{p^{2r}}$ acts on $C$ as endomorphisms, and we have $\End_{\F_{p^{2r}}} C(\kbar) \isom \F_{p^{2r}}$ canonically.
Also the character $G_k \to \Aut_{\F_{p^{2r}}} C(\kbar) = \F_{p^{2r}}^*$ factors through our character $\nu$.

Using \cref{good_reduction,Aut:local}, take a totally ramified extension $M$ of $k$ of degree $n_{\lcm}$ such that $A_M$ has good reduction over $M$.
Then the group scheme $A[p_F]_M$ is extended to a finite flat group scheme over $O_M$.
Thus the subgroup scheme $C_M$, which is an $\F_{p^{2r}}$-vector scheme in the terminology of \cite{Raynaud} inducing $G_k \to \F_{p^{2r}}^*$, also has a prolongation over $O_M$.
Therefore by \cite[Th\'eor\`eme 3.4.3]{Raynaud}, fixing an inclusion $\F_{p^{2r}} \to \overline{\F_k}$ over $\F_{p_F}$ and identifying $\mu_{p^{2r} - 1}(k^\nr)$ with $\F_{p^{2r}}^*$ there exists an integer $b'$ such that $\nu |_{I_M} = \theta_{p^{2r}-1,M}^{b'}$ and that $b' = \sum_{i=0}^{2r-1} b_i p^i$, where $0 \le b_i \le n_{\lcm}$ for each $i$.
Since its image lies in $\F_{p_F}^*$, it follows by the surjectivity of $\theta_{p^{2r}-1,M}$ that $1+p^r$ divides $b'$, and so $\sum_{i=0}^{r-1} b_i p^i - \sum_{i=0}^{r-1} b_{r+i} p^i$.
Hence by the assumption that $n_{\lcm} < (p^r+1)(p-1)/(p^r-1)$ we obtain that $\sum_{i=0}^{r-1} b_i p^i = \sum_{i=0}^{r-1} b_{r+i} p^i$.
In particular we obtain that $\nu |_{I_M} = \theta_{p^r-1,M}^b$ for $b = \sum_{i=0}^{r-1} b_i p^i$ since $\theta_{p^{2r}-1,k}^{1+p^r} = \theta_{p^r-1,k}$ by definition of these characters.

On the other hand, by \cite[Proposition 5]{SerreGal}
there exists an integer $c$ such that $\nu |_{I_k} = \theta_{p^r-1, k}^c$, and hence $\nu |_{I_M} = \theta_{p^r-1, k}^c |_{I_M}$, which equals to $\theta_{p^r-1,M}^{cn_{\lcm}}$ by \cite[Section 1.4]{SerreGal}.
Therefore $\theta_{p^r-1,M}^{b - cn_{\lcm}} = 1$, which implies that $p^r-1$ divides $b - cn_{\lcm}$.
Consequently $\nu^{n_{\lcm}} |_{I_k} = \theta_{p^r-1,k}^{cn_{\lcm}} = \theta_{p^r-1,k}^b$.
\end{proof}

For later use we need to know more about the integer $b$ as in \cref{nu_above_p_local}.

\begin{lemma} \label{nu_above_p_more_about_b}
Keep the notations as in \cref{nu_above_p_local}.
If $\nu^{n_{\lcm}} |_{I_k} = \theta_{p^r-1,k}^{n_{\lcm}/2}$, then $v_2(p^r-1) + 1 \le v_2(n_{\lcm})$.
\end{lemma}

\begin{proof}
Keep the notations as in the proof of \cref{nu_above_p_local}, and assume that $\nu^{n_{\lcm}} |_{I_k} = \theta_{p^r-1,k}^{n_{\lcm}/2}$.
Then $\theta_{p^r-1,k}^{cn_{\lcm} - (n_{\lcm}/2)} |_{I_k} = 1$.
It follows, from the surjectivity of $\theta_{p^r-1,k}$, that $p^r-1$ divides $cn_{\lcm} - (n_{\lcm}/2)$.
Since $v_2(n_{\lcm}) = v_2(cn_{\lcm} - (n_{\lcm}/2)) + 1$ we have the result.
\end{proof}

\section{Characters of QM abelian varieties with level-$\Gamma_0$ structures over global fields} \label{section:global}

In this section we study the rational points of the Shimura varieties $M_0^B(p_F)$.
At first we define a character obtained from a rational point of $M_0^B(p_F)$.
If a rational point $x$ of $M_0^B(p_F)$ over a field $k$ of characteristic zero has a model $X$ over $k$, then $X$ induces the $\Gamma_0$-character $\nu$ over $k$ as in the previous section.
Even if $x$ has no models over $k$, we can canonically define a character over $k$ which is essentially equal to ``$\nu^{n_{\lcm}}$" in a certain sense.
(Recall that $n_{\lcm} = \lcm \{ m : [F(\zeta_m) : F] \le 2 \}$, where $\zeta_m$ is a primitive $m$-th roof of $1$, and it is divisible by $12$.)

Let $k$ be a field of characteristic zero, $p_F$ a prime of $F$ not dividing $\Delta$, and $p$ be the rational prime below $p_F$.
Let $x \in M^B_0(p_F)(k)$ be a rational point, $X \in \mathscr{M}^B_0(p_F)(\kbar)$ a model of $x$ over $\kbar$, $(A, i)$ the underlying QM-abelian variety, and $C$ the underlying subgroup $\kbar$-scheme of $A$ of type $\Gamma_0(p_F)$.
For each $\sigma \in G_k$, choose an isomorphism $\lambda_\sigma \colon X^\sigma \to X$ over $\kbar$,
and define $\iota(\sigma)$ to be an automorphism of $C(\kbar)$ which sends $P$ to $\lambda_\sigma(P^\sigma)$.
Actually $\iota(\sigma)$ belongs to $\Aut_{O_B} C(\kbar)$.
For, since $\lambda_\sigma$ is an isomorphism of objects of $\mathscr{M}^B_0(p_F)(\kbar)$, for each $b \in O_B$, as endomorphisms of $A$, we have that $\lambda_\sigma i(b)^\sigma = i(b) \lambda_\sigma$.
Hence $\iota(\sigma) (i(b)P) = \lambda_\sigma i(b)^\sigma (P^\sigma) = i(b) \iota(\sigma) (P)$, which means that $\iota(\sigma) \in \Aut_{O_B} C(\kbar)$.
Thus associating to each $\sigma \in G_k$ the automorphism $\iota(\sigma)$, we obtain a morphism $G_k \to \Aut_{O_B} C(\kbar)$, and dividing its target by $\Aut (X/\kbar)$, we obtain a morphism $G_k \to \Aut_{O_B} C(\kbar) / \Aut (X/\kbar)$, which is independent of the choice of isomorphisms $\lambda_\sigma$.
Now $\Aut_{O_B} C(\kbar)$ is canonically isomorphic to $\F_{p_F}^*$ by $i$, and the exponent of $\Aut (X/\kbar)$ divides $n_{\lcm}$ by \cref{Aut:nonCM,Aut:CM}.
Therefore by taking the $n_{\lcm}$-th power, we obtain a morphism $\varphi_X \colon G_k \to (\F_{p_F}^*)^{n_{\lcm}}$.
We claim that $\varphi_X$ is actually independent of the choice of a model $X/\kbar$.
Let $X'/\kbar$ be another model of $x$, $(A', i')$ the underlying QM-abelian variety, $C'$ the underlying subgroup $\kbar$-scheme of $A'$ of type $\Gamma_0(p_F)$, and $f \colon X \to X'$ an isomorphism over $\kbar$.
For each $\sigma \in G_k$, choose an isomorphism $\lambda_\sigma \colon X^\sigma \to X$ and let $\lambda'_\sigma := f \lambda_\sigma (f^\sigma)^{-1} \colon X'^\sigma \to X'$.
Let $\iota(\sigma)$ be an isomorphism belonging to $\Aut_{O_B} C(\kbar)$ which sends $P$ to $\lambda_\sigma(P^\sigma)$, and similarly for $\iota'(\sigma) \in \Aut_{O_B} C'(\kbar)$.
Then $\iota'(\sigma) = f \iota(\sigma) f^{-1}$.
Since $f$ is an isomorphism of objects of $\mathscr{M}^B_0(p_F)(\kbar)$, we have following commutative diagram:
\begin{equation*}
\begin{aligned}
\xymatrix{
 & O_F \ar[ld]_i \ar[rd]^{i'} &  \\
\End_{O_B} C(\kbar) \ar[rr]^{f \circ (-) \circ f^{-1}} & & \End_{O_B} C'(\kbar).
}
\end{aligned}
\end{equation*}
Therefore $\varphi_X = \varphi_{X'}$.
We denote the obtained character $G_k \to (\F_{p_F}^*)^{n_{\lcm}}$, which depends only on $x$, by $\varphi_x$, and call it the $M_0$-character obtained from $x \in M^B_0(p_F)(k)$.

\begin{remark}
As in the case of $M^B$, which is described in \cite{Skorobogatov}, for sufficiently large $p_F$ there exists an intermediate covering $Y^B(p_F)$ of the canonical map $M^B_1(p_F) \to M^B_0(p_F)$ such that $Y^B(p_F) \to M^B_0(p_F)$ is a torsor under the constant group $(\F_{p_F}^*)^{n_{\lcm}}$.
Then we can show that for a field $k$ of characteristic zero and for a rational point $x \in M^B_0(p_F)(k)$, the pullback of $Y^B(p_F) \to M^B_0(p_F)$ along with $x \colon \Spec k \to M^B_0(p_F)$ corresponds, under the correspondence between torsors and first cohomology classes, to the $M_0$-character $\varphi_x$.
Thus by studying $M_0$-characters over a number field $k$, we can study not only the $k$-rational points of $M^B_0(p_F)$ but also the Brauer--Manin set $M^B_0(p_F)(\A_k)^{\operatorname{Br}}$.
We discuss in more detail in a future work.
\end{remark}

As we have mentioned, the $M_0$-characters are analogous to the $n_{\lcm}$-th powers of the $\Gamma_0$-characters:

\begin{lemma} \label{varphi=nu^n}
Let $k$ be a field of characteristic zero and $x \in M^B_0(p_F)(k)$.
If $x$ has a model $X$ over $k$, then $\varphi_x = \nu^{n_{\lcm}}$, where $\nu$ is the $\Gamma_0$-character obtained from $X$.
\end{lemma}

\begin{proof}
In this case for each $\sigma \in G_k$ take the natural isomorphism $X^\sigma \to X$.
Then $\iota(\sigma)(P) = P^\sigma$ for each $P \in C(\kbar)$, where $C$ is the underlying subgroup $k$-scheme of type $\Gamma_0(p_F)$.
Thus we obtain the result.
\end{proof}

With this lemma, in order to study the behavior of the $M_0$-character of a rational point, we need to find its models over a nice extension of the base field, and in order to find them, we need to study the Galois action on the automorphism group.

\begin{lemma} \label{AutX_injects_into_AutC}
Let $k$ be a field of characteristic zero, $X \in \mathscr{M}^B_0(k)$, $A$ the underlying QM-abelian variety, and $C$ be the underlying subgroup $k$-scheme of $A$ of type $\Gamma_0(p_F)$.
Assume that $N(p_F)$, the absolute norm of $p_F$, is greater than $4^d$.
Then the canonical map $\Aut X/k \to \Aut_{O_B} C(\kbar)$ is injective.
\end{lemma}

\begin{proof}
Let $f \in \Aut X/k$ be a nontrivial automorphism and assume that $f = 1$ as an element of $\Aut_{O_B} C(\kbar)$.
Then by \cref{isogeny,B-simple}, an endomorphism $f - 1 \in \End_{O_B} A/k$ is an isogeny.
Since $C \subseteq \ker f - 1$, the degree of $f-1$ is divided by $\# C(\kbar) = N(p_F)^2$.
Moreover, if we let $c := (f-1)^t \circ (f-1) \in O_F$ (see \cref{dual_isogeny}), then since $N_{F/\Q}(c)^2 = \deg (f-1)$ by \cref{dual_isogeny}, we have that $N(p_F)$ divides $N_{F/\Q}(c)$, in particular $N(p_F) \le N_{F/\Q}(c)$.

Now since $f$ preserves the polarization, we have that $f^t = f^{-1}$, and hence for $a := 2 - c$, our automorphism $f$ satisfies the equation $T^2 - aT + 1 = 0$ in $\End_B^0 A/k$.
Since $f$ has finite order, we obtain that $a^2 - 4$ is zero or totally negative in $F$, i.e., for every infinite place $v$ of $F$ we have that $|a|_v \le 2$.
Thus $|c|_v \le 4$.
Therefore we get an inequality $N(p_F) \le N_{F/\Q}(c) \le 4^d$, a contradiction.
\end{proof}

\begin{lemma} \label{G_k_acts_trivially_on_AutX}
Let $k$ be a field of characteristic zero, $x \in M^B_0(p_F)(k)$ a rational point, $X$ a model of $x$ over $\kbar$, and $C$ be the underlying subgroup $\kbar$-scheme of type $\Gamma_0(p_F)$.
Assume that $N(p_F)$ is greater than $4^d$.
Define an action of $G_k$ on $\End_{O_B} C(\kbar)$ from right as the one on $\Aut X/\kbar$, see \cref{descent}.
Then the actions of $G_k$ on $\Aut X/\kbar$ and $\Aut_{O_B} C(\kbar)$ are trivial.
\end{lemma}

\begin{proof}
By \cref{AutX_injects_into_AutC}, it suffices to show the triviality of the Galois module $\Aut_{O_B} C(\kbar)$.
Let $(A,i)$ be the underlying QM-abelian variety over $\kbar$ of $X$.
Then for every $\sigma \in G_k$, for every isomorphism $\lambda_\sigma \colon X^\sigma \to X$ over $\kbar$, and for every $b \in O_B$, the following diagram is commutative:
\begin{equation*}
\begin{aligned}
\xymatrix{
A^\sigma \ar[r]^{\lambda_\sigma} \ar[d]_{i(b)^\sigma} & A \ar[d]^{i(b)} \\
A^\sigma \ar[r]^{\lambda_\sigma} & A.
}
\end{aligned}
\end{equation*}
Therefore, since $i \colon O_F \to \End_{O_B} C(\kbar)$ is surjective, $G_k$ acts on $\Aut_{O_B} C(\kbar)$ trivially.
\end{proof}

Any rational point over a local field has models over the maximal unramified extension.

\begin{lemma} \label{model_over_knr}
Let $\ell$ be a rational prime, $k$ a finite extension of $\Q_\ell$, and $x \in M_0^B(p_F)(k)$.
Then $x$ has a model over $k^\nr$.
\end{lemma}

\begin{proof}
Let $X$ be a model of $x$ over $\kbar$.
As we have noted in \cref{descent}, the obstruction to descent of $X$ from $\kbar$ to $k^\nr$ belongs to $H^2(k^\nr, \Aut X/\kbar)$, which does, since $k^\nr$ has cohomological dimension one (e.g., see \cite[Theorem 4.5.9]{Fu}), vanish.
\end{proof}

For a field $k$ of characteristic zero, for $x \in M^B_0(p_F)(k)$, and for a model $X$ of $x$ over $\kbar$, we define $\Aut x/\kbar := \Aut X/\kbar$, which is independent of the choice of $X$ up to isomorphism.

\begin{lemma} \label{model_over_totally_ramified_extension}
Let $\ell$ be a rational prime, $k$ a finite extension of $\Q_\ell$, $x$ a rational point of $M_0^B(p_F)$ over $k$, $m := \# \Aut x/\kbar$, and $L$ be a totally ramified extension of $k$ of degree $m$.
Assume that $\ell$ does not divide $m$.
Then $x$ has a model over $L$.
\end{lemma}

\begin{proof}
By \cref{Aut:nonCM,Aut:CM,G_k_acts_trivially_on_AutX}, the Galois module $\Aut x/\kbar$ is isomorphic to $\Z/m$.
The groups $G_{\F_k}$ and $I_k$ have cohomological dimension one.
(E.g., see \cite[Example 6.1.11]{GS} and \cite[Theorem 4.5.9]{Fu}.)
Thus by Hochschild--Serre spectral sequence, we obtain a canonical isomorphism of cohomology groups $H^2(k, \Z/m) \isom H^1(\F_k, \Hom(I_k, \Z/m))$.
Now since we are assuming that $\ell$ does not divide $m$, the canonical map $\Hom(I_k, \Z/m) \to \Hom(I_L, \Z/m)$ is zero.
Therefore $H^2(k, \Z/m) \to H^2(L, \Z/m)$ is also zero, and hence the obstruction to descent of models of $x$ from $\kbar$ to $L$ as in \cref{descent} vanishes, i.e., $x$ has a model over $L$.
\end{proof}

By these lemmata, $M_0$-characters over a local field of characteristic zero have the same ramification data as the $n_{\lcm}$-th powers of $\Gamma_0$-characters.

\begin{lemma} \label{varphi_nr}
Let $\ell \ne p$ be a rational prime, $k$ a finite extension of $\Q_\ell$, and $x \in M^B_0(p_F)(k)$.
Then $\varphi_x$ is unramified.
\end{lemma}

\begin{proof}
By \cref{nu_nr,varphi=nu^n,model_over_knr}.
\end{proof}

\begin{lemma} \label{varphi_q_equals_to_beta}
Let $\ell \ne p$ be a rational prime not dividing $n_{\lcm}$, $k$ a finite extension of $\Q_\ell$ of inertia degree $f$, and $x \in M^B_0(p_F)(k)$.
Assume that $N(p_F)$ is greater than $4^d$.
Then there exists $\beta \in FR(\ell^{f})$ and a prime $\mathfrak{r}$ of $F(\beta)$ above $p_F$ such that $\varphi_x(\Frob_k) \equiv \beta^{n_{\lcm}} \mod \mathfrak{r}$.
\end{lemma}

\begin{proof}
Let $m := \# \Aut x/\kbar$ and $L$ be a totally ramified extension of $k$ of degree $m$.
Then since we are assuming that $\ell$ does not divide $n_{\lcm}$, in particular does not divide $m$, by \cref{model_over_totally_ramified_extension} the rational point $x$ has a model $X$ over $L$, which we denote by $X$.
By the construction, the inertia degree of the extension $L/k$ is equal to $1$.
Thus $\varphi_x(\Frob_k) = \varphi_x(\Frob_L)$, which equals to $\nu_{X/L}(\Frob_L)^{n_{\lcm}}$, where $\nu_{X/L}$ is the $\Gamma_0$-character induced by $X/L$, by \cref{varphi=nu^n}.
Consequently, since the absolute inertia degree of $L$ is $f$, which is the one of $k$, the result follows from \cref{nu_q_equals_to_beta:local}.
\end{proof}

\begin{lemma} \label{varphi_above_p_local}
Let $k$ be a finite extension of $\Q_p$ and $x \in M^B_0(p_F)(k)$.
Write $\# \F_{p_F} = p^r$.
Assume that $n_{\lcm} < p$.
Fix an inclusion $\F_{p_F} \to \overline{\F_k}$ and identify $\mu_{p^r - 1}(k^\nr)$ with $\F_{p_F}^*$.
Then there exists an integer $b$ such that $\varphi_x |_{I_k} = \theta_{p^r-1, k}^b$ and that $b = \sum_{i=0}^{r-1} b_i p^i$, where $0 \le b_i \le n_{\lcm}$ for each $i$.
Moreover if $\varphi_x |_{I_k} = \theta_{p^r-1,k}^{n_{\lcm}/2}$, then $v_2(p^r-1) + 1 \le v_2(n_{\lcm})$.
\end{lemma}

\begin{proof}
By \cref{nu_above_p_local,nu_above_p_more_about_b,varphi=nu^n,model_over_knr}.
\end{proof}

In the rest of this section let $k$ be a number field, $p_F$ a prime of $F$ not dividing $\Delta$, $p$ the rational prime below $p_F$, and $x \in M^B_0(p_F)(k)$.
Assume that $N(p_F) > 4^d$ and that $p > n_{\lcm}$.
Note that by \cref{no_real_points} the base field $k$ must be totally imaginary.
Let $\varphi \colon G_k \to \F_{p_F}^*$ be the $M_0$-character induced by $x$.
We study $\varphi$ following the methods of \cites{MazurX0,Momose,AraiI}.

For simplicity, for a number field $L$ and for a prime ideal $\mathfrak{p}$ of $L$,
we denote the inertia group $I_{L_\mathfrak{p}}$ by $I_\mathfrak{p}$.
For an ideal $\mathfrak{a}$ of $L$,
let $J_L^\mathfrak{a}$ denote the group of prime-to-$\mathfrak{a}$ fractional ideals of $L$.
For an abelian group $H$ and for a group homomorphism $\iota \colon G_L \to H$,
if $\iota$ is unramified outside $p$,
then it induces $J_L^{p O_L} \to H$, which we denote again by the same symbol $\iota$.
In particular by the same symbol $\varphi$ we denote the character $J_k^{p O_k} \to \F_{p_F}^*$ induced by $\varphi$ (see \cref{varphi_nr}).

In order to piece together the local information and obtain the global information of $\varphi$, we need a formal lemma.

\begin{lemma} \label{character_varepsilon}
Let $L$ be a totally imaginary Galois number field, $\mathfrak{p}$ a prime of $L$ above $p$,
and $\mu \colon G_L \to \F_p^*$ be a character which is unramified outside $p$.
Assume that for every prime $\mathfrak{r}$ of $L$ above $p$, there exists a nonnegative integer $c_\mathfrak{r}$ such that
$\mu|_{I_\mathfrak{r}} = \chi^{c_\mathfrak{r}}$.
Define $\varepsilon = \sum_{\sigma \in \Gal(L/\Q)} c_{(\mathfrak{p}^\sigma)} \sigma^{-1} \in \Z[\Gal(L/\Q)]$.
Then for every prime-to-$pO_L$ (as ideals) nonzero algebraic integer $\alpha \in O_L$,
the element $\alpha^\varepsilon \mod \mathfrak{p}$, which is a priori an element of $\F_\mathfrak{p}$,
lies in $\F_p$, and we have that $\mu(\alpha O_L) = \alpha^\varepsilon \mod \mathfrak{p}$.
\end{lemma}

\begin{proof}
Let $\alpha \in O_L$ be a nonzero element which is relatively prime to $p$ as ideals.
Since $L$ has no real infinite places, in the ray class group modulo $pO_L$, the idele $(\alpha^{-1}, \dots, \alpha^{-1}, 1, \dots, 1)$, whose components above $p$ are $\alpha^{-1}$ and others are $1$, corresponds to the ideal $\alpha O_L$.
Thus
\begin{equation*}
\begin{aligned}
\mu (\alpha O_L) & = \mu (\alpha^{-1}, \dots, \alpha^{-1}, 1, \dots, 1) \\
& = \prod_{\mathfrak{p'} \mid p} \mu |_{I_\mathfrak{p'}} (\alpha^{-1}) \\
& = \prod_{\mathfrak{p'} \mid p} \chi^{c_\mathfrak{p'}}(\alpha^{-1}) \\
& = \prod_{\mathfrak{p'} \mid p} N_{L_\mathfrak{p'}/\Q_p}(\alpha)^{c_\mathfrak{p'}} \mod p.
\end{aligned}
\end{equation*}
Since $L/\Q$ is Galois, if we denote the decomposition group of $\mathfrak{p}$ by $D_\mathfrak{p}$, then the right hand side equals to
\begin{equation*}
\prod_{D_\mathfrak{p}\sigma \in D_\mathfrak{p} \backslash \Gal(L/\Q)}
N_{L_{\mathfrak{p}^\sigma}/\Q_p}(\alpha)^{c_{(\mathfrak{p}^\sigma)}} \mod p.
\end{equation*}
Now for each $\sigma \in \Gal(L/\Q)$, we have
\begin{equation*}
N_{L_{\mathfrak{p}^\sigma}/\Q_p}(\alpha)^{c_{(\mathfrak{p}^\sigma)}} \mod p
= \prod_{\tau \in \Gal(L_{\mathfrak{p}^\sigma}/\Q_p)} \alpha^{c_{(\mathfrak{p}^\sigma)} \tau} \mod \mathfrak{p}^\sigma,
\end{equation*}
which is, since $\Gal(L_{\mathfrak{p}^\sigma}/\Q_p) = D_{\mathfrak{p}^\sigma} = \sigma^{-1} D_\mathfrak{p} \sigma$, equal to
\begin{equation*}
(\prod_{\tau \in D_\mathfrak{p}} \alpha^{c_{(\mathfrak{p}^\sigma)} \sigma^{-1} \tau} \mod \mathfrak{p} )^\sigma.
\end{equation*}
Since $\prod_{\tau \in D_\mathfrak{p}} \alpha^{c_{(\mathfrak{p}^\sigma)} \sigma^{-1} \tau}
= N_{L_\mathfrak{p}/\Q_p} (\alpha^{c_{(\mathfrak{p}^\sigma)} \sigma^{-1}})$,
which lies in the decomposition field of $\mathfrak{p}$ above $p$, we have
\begin{equation*}
N_{L_{\mathfrak{p}^\sigma}/\Q_p}(\alpha)^{c_{(\mathfrak{p}^\sigma)}} \mod p
= \prod_{\tau \in D_\mathfrak{p}} \alpha^{c_{(\mathfrak{p}^\sigma)} \sigma^{-1} \tau} \mod \mathfrak{p}.
\end{equation*}
Combining them together one gets
\begin{equation*}
\mu (\alpha O_L)
= \prod_{\sigma \in \Gal(L/\Q)} \alpha^{c_{(\mathfrak{p}^\sigma)} \sigma^{-1}} \mod \mathfrak{p}.
\end{equation*}
This shows the statement.
\end{proof}

\begin{corollary} \label{nu_above_p}
Assume that $k$ is Galois over $\Q$ and that $p_F$ splits totally over $\Q$.
Let $\mathfrak{p}$ be an unramified prime of $k$ above $p$.
Then there exists $\varepsilon = \sum_\sigma b_\sigma \sigma \in \Z[\Gal(k/\Q)]$ such that
$0 \le b_\sigma \le n_{\lcm}$ for every $\sigma \in \Gal(k/\Q)$, and satisfying the following three conditions:
\begin{enumerate}
\item For every prime-to-$p$ (as ideals) nonzero algebraic integer $\alpha \in O_k$,
the element $\alpha^\varepsilon \mod \mathfrak{p}$, which is a priori an element of $\F_\mathfrak{p}$, lies in $\F_p$,
and we have that $\varphi(\alpha O_k) = \alpha^\varepsilon \mod \mathfrak{p}$.

\item For every $\sigma \in \Gal(k/\Q)$, we have $\varphi |_{I_{(\mathfrak{p}^{\sigma^{-1}})}} = \chi^{b_\sigma}$.

\item If there exists $\sigma \in \Gal(k/\Q)$ such that $b_\sigma = n_{\lcm} / 2$,
then $v_2(p-1) + 1 \le v_2(n_{\lcm})$.
\end{enumerate}
\end{corollary}

\begin{proof}
By \cref{varphi_nr,varphi_above_p_local,character_varepsilon}.
\end{proof}

By \cref{nu_above_p}, we can say that the essential data of $\varphi$ is $\varepsilon$ as in \cref{nu_above_p}.
From now we study it.

For a number field $L$ denote the class number of $L$ by $h_L$.

Let $M(k)$ be the set of rational primes $\ell$ which split totally in $k$, which do not divide $n_{\lcm} h_k$, and which are unramified in $F$.
Let $N(k)$ be the set of prime ideals $\mathfrak{q}$ of $k$ which divide some $\ell \in M(k)$.

\begin{lemma} \label{finite_subset_of_N(k)}
Let $k$ be a Galois number field.
Then there exists a finite subset $S$ of $N(k)$ generating the class group $Cl_k$ of $k$.
\end{lemma}

\begin{proof}
Let $H$ be the Hilbert class field of $k$.
Then by global class field theory we have the isomorphism $Cl_k \isom \Gal(H/k)$.
We identify them by this isomorphism.
Let $\mathfrak{a} \in Cl_k$.
Consider the set of rational primes $\ell$ which are unramified in $H$
and $\ell$-th Frobenius automorphisms $\varphi_\ell$ of $\Gal(H/\Q)$ correspond to $\mathfrak{a}$.
By Chebotarev's density theorem, this set has positive density.
Let $\ell$ be an element of this set which is unramified in $F$ and which does not dividing $n_{\lcm}h_k$.
Then there exists a prime $\mathfrak{q}$ of $k$ dividing $\ell$ such that $\mathfrak{q} = \mathfrak{a}$ in $Cl_k$,
and since $\varphi_\ell |_k = 1$, the inertia degree of $\ell$ in $k$ is $1$.
This gives the statement.
\end{proof}

Henceforth in this section, we assume that $F$ and $k$ are Galois over $\Q$ and that $k$ contains $F$ unless otherwise stated.

Let $S$ be a set as in \cref{finite_subset_of_N(k)}.
Enlarging $S$ if necessary, we may assume that $S$ is nonempty
and that for every $\mathfrak{q} \in S$ and for every $\rho \in \Gal(k/\Q)$ the prime ideal $\mathfrak{q}^\rho$ lies in $S$.
For each $\mathfrak{q} \in S$, fix an element $\alpha_\mathfrak{q} \in O_k$ satisfying $\mathfrak{q}^{h_k} = \alpha_\mathfrak{q} O_k$.
Note that for every $\mathfrak{q} \in S$ and for every $\rho \in \Gal(k/\Q)$, we have $(\alpha_\mathfrak{q})^\rho O_k = \alpha_{(\mathfrak{q}^\rho)} O_k$.
We define five sets following \cite{AraiI}:
\begin{equation*}
\begin{aligned}
M_1(k) & = \left\{ (\mathfrak{q}, \varepsilon, \beta_\mathfrak{q}) \bigg | 
\begin{array}{c}
\mathfrak{q} \in S, \varepsilon = \sum_{\sigma \in \Gal(k/\Q)} a_\sigma \sigma \in \Z[\Gal(k/\Q)], \\
0 \le a_\sigma \le n_{\lcm} \text{ for every } \sigma \in \Gal(k/\Q),
\beta_\mathfrak{q} \in FR(N(\mathfrak{q}))
\end{array}
\right\}, \\
M_2(k) & = \left\{ N_{k(\beta_\mathfrak{q}) / F} (\alpha_\mathfrak{q}^\varepsilon - \beta_\mathfrak{q}^{n_{\lcm} h_k}) \bigg | 
\begin{array}{c}
(\mathfrak{q}, \varepsilon, \beta_\mathfrak{q}) \in M_1(k)
\end{array}
\right\} - \{ 0 \}, \\
N_0(k) & = \{ \ell_F | \text{ primes of } F \text{ dividing some } m \in M_2(k) \}, \\
T(k) & = \left\{ \ell_F \bigg |
\begin{array}{c}
\text{ primes of } F \text{ divided by some } \mathfrak{q} \in S, N(\ell_F) < 4^d, \\
\text{ or dividing a rational prime less than } n_{\lcm}
\end{array}
\right\}, \\
N_1(k) & = N_0(k) \cup T(k) \cup \{ \ell_F | \text{ ramified in } k \text{ or over } \Q \}.
\end{aligned}
\end{equation*}
Note that by the finiteness of $S$, all of these five sets are finite.

\begin{lemma} \label{type_of_varepsilon}
Let $\ell \ne p$ be a rational prime which is totally split in $k$, and let $\mathfrak{q}$ be a prime of $k$ above $\ell$.
Let $m$ be a positive integer and
let $\varepsilon = \sum_{\sigma \in \Gal(k/\Q)} a_\sigma \sigma$ be an element of $\Z[\Gal(k/\Q)]$ whose coefficients $a_\sigma$ are nonnegative.
Write $\mathfrak{q}^{h_k} = \alpha O_k$ for $\alpha \in O_k$
and assume that $\alpha^\varepsilon = \beta^{2 m h_k}$ for an element $\beta \in FR(\ell)$.
Then $L := F(\beta)$ is a totally imaginary quadratic extension of $F$ and one of the following holds:
\begin{enumerate}[Type 1.]
\setcounter{enumi}{1}
\item The element $\varepsilon$ is equal to $m \sum_{\sigma \in \Gal(k/\Q)} \sigma$.

\item The field $L$ is contained in $k$, and
for every system of representatives $\{ \tau_i \}_{i = 1, \dots, 2d}$ of the cosets of $\Gal(k/\Q) / \Gal(k/L)$,
there exists a system of representatives $I$ of the cosets of $\Gal(k/\Q) / \Gal(k/F)$ in $\{ \tau_i \}$ such that $\varepsilon = 2m \sum_{\tau \in I} \sum_{\sigma \in \Gal(k/L)} \tau \sigma $.
\end{enumerate}
\end{lemma}

\begin{proof}
We have a tower of fields $F(\beta) \supseteq F(\alpha^\varepsilon) \supseteq F$.
As we have noted, since now $\beta \in FR(\ell)$, the field $F(\beta)$ is a totally imaginary quadratic extension of $F$.
It follows that we have two possibilities:
$F(\alpha^\varepsilon) = F(\beta)$ or $\alpha^\varepsilon \in F$.
Also as we have noted for every infinite place $v$ of $F(\beta)$ that the absolute value $| \beta |_v$ is equal to $\sqrt{\ell}$.

First assume that $\alpha^\varepsilon \in F$.
Then for every infinite place $v$ of $F(\beta)$, we have $| \alpha^\varepsilon |_v = | \beta |_v^{2 m h_k} = \ell^{m h_k}$.
Therefore since $\alpha^\varepsilon$ lies in $F$, which is totally real, we have that $\alpha^\varepsilon = \pm \ell^{m h_k}$.
Thus $\mathfrak{q}^{h_k \varepsilon} = \alpha^\varepsilon O_k = \ell^{m h_k} O_k$,
and hence $\mathfrak{q}^\varepsilon = \ell^m O_k$.
Since we are assuming that $\ell$ is totally split in $k$, the ideal $\ell O_k$ equals to $\prod_{\sigma \in \Gal(k/\Q)} \mathfrak{q}^\sigma$.
It follows that $\varepsilon = m \sum_{\sigma \in \Gal(k/\Q)} \sigma$.

Next assume that $F(\alpha^\varepsilon) = F(\beta)$.
In particular the field $L = F(\beta)$ is contained in $k$.
Choose a system of representatives $\{ \tau_i \}_{i = 1, \dots, 2d}$ of the cosets of $\Gal(k/\Q) / \Gal(k/L)$.
Let $\mathfrak{r}$ be the prime of $L$ below $\mathfrak{q}$.
Since $N_{L/F}(\beta) = \ell$ and hence $N_{L/\Q}(\beta) = \ell^d$, the ideal $\beta O_L$ is of the form of $\prod_i \mathfrak{r}^{c_i \tau_i}$ for some $0 \le c_i$ satisfying $\sum_i c_i = d$, where we write $\mathfrak{r}^\sigma = \mathfrak{q}^\sigma \cap O_L$.
Note that since $\ell$ splits totally in $k$, for $i \ne j$, the primes $\mathfrak{r}^{\tau_i}$ and $\mathfrak{r}^{\tau_j}$ are distinct.
Thus we have the equation of ideals $\beta O_k = \prod_i \prod_{\rho \in \Gal(k/L)} \mathfrak{q}^{c_i \tau_i \rho} $.
On the other hand, since $\beta^{2 m h_k} O_k = \alpha^\varepsilon O_k = \mathfrak{q}^{\varepsilon h_k}$,
the ideal $\beta^{2 m} O_k$ is equal to $\mathfrak{q}^\varepsilon$.
It follows that $\varepsilon = 2m \sum_i \sum_{\rho \in \Gal(k/L)} c_i \tau_i \rho$.

Finally it only remains to prove that the set $I := \{ \tau_i | c_i \ne 0 \}$ is a system of representatives of the cosets of $\Gal(k/\Q) / \Gal(k/F)$,
for, if so, then by the nonnegativity of $c_i$'s and by the equation $\sum_i c_i = d$, we have that $c_i$ are $0$ or $1$.
Let $\ell_F$ be the prime of $F$ below $\mathfrak{r}$.
Applying the norm $N_{L/F}$ to the equation $\beta O_L = \prod_{i} \mathfrak{r}^{c_i \tau_i}$,
we get $\ell O_F = \prod_{i} \ell_F^{c_i \tau_i}$, which yields the result since $\ell$ splits totally in $F$.
\end{proof}

\begin{lemma} \label{I_is_a_group}
Let $\Gamma$ be a finite group, $H \subseteq N \subseteq \Gamma$ subgroups, and let $I$ be a system of representatives of the cosets of $\Gamma / N$.
If either $[\Gamma : N] = 1$, or $[\Gamma : N] = 2$ and the exponent of $\Gamma$ is $2$,
then there exists an element $\sigma \in \Gamma$ such that the set $\sigma I H$ is a subgroup of $\Gamma$.
\end{lemma}

Later we apply this lemma for $\Gamma = \Gal(k/\Q), N = \Gal(k/F),$ and for $H = \Gal(k/L)$ in the case that $\varepsilon$ is of type $3$ as in \cref{type_of_varepsilon}.

\begin{proof}
If $[\Gamma : N] = 1$ then this is trivial.
Assume that $[\Gamma : N] = 2$ and the exponent of $\Gamma$ is $2$.
In this case $\Gamma$ is abelian, in particular $H$ is normal in $\Gamma$, and hence we may assume that $H$ is trivial.
Let $\sigma \in I$.
Then the set $\sigma^{-1} I$,
whose cardinality is $[\Gamma : N] = 2$, contains the identity.
Since the exponent is $2$, the set $\sigma^{-1} I$ is a subgroup of $\Gamma$.
\end{proof}

This is a crucial lemma in this paper.
If we can show it in the case that $\varepsilon$ is of type $3$ as in \cref{type_of_varepsilon} for general $d$ and for general $k$,
it seems that the main theorems of this paper also hold in general case.

\begin{lemma} \label{L_is_independent}
Keep the notations as in \cref{type_of_varepsilon}, and assume that $\varepsilon$ is of type 3.
Assume that $I \Gal(k/L) = \rho \Gal(k/M)$ for $\rho \in \Gal(k/\Q)$ and for a subfield $M$ of $k$.
Then $M$ is an imaginary quadratic field, is independent of the choice of a system of representatives of the cosets of $\Gal(k/\Q) / \Gal(k/L)$, and $L = MF$.
\end{lemma}

\begin{proof}
Taking the restriction of the equation in the statement to $F$, we have that $I |_F = \rho|_F \Gal(F / F \cap M)$, and hence the intersection $F \cap M$ is equal to $\Q$.
On the other hand since $ \rho \in I \Gal(k/L)$, there exists $\tau \in I$ such that $\rho^{-1} \tau \in \Gal(k/L)$, and hence $\Gal(k/L)$ is contained in $\Gal(k/M)$, i.e., the field $M$ is contained in $L$.
Thus $[MF : \Q]$ divides $[L : \Q] = 2d$, is equal to $[MF : M] [M : \Q] = [F : \Q] [M : \Q]$, and $[M : \Q] \ne 1$.
Consequently $MF = L$ and $[M : \Q] = 2$, and since $F$ is totally real while $L$ is totally imaginary, $M$ must be imaginary.
Next take another system of representatives of the cosets of $\Gal(k/\Q) / \Gal(k/L)$, and let $I'$ be the subset of it as in \cref{type_of_varepsilon}.
Since $n_{\lcm} \sum_{\sigma \in I \Gal(k/L)} \sigma = \varepsilon = n_{\lcm} \sum_{\sigma \in I' \Gal(k/L) } \sigma$, we have that $I \Gal(k/L) = I' \Gal(k/L)$.
Therefore $I' \Gal(k/L) = \rho \Gal(k/M)$.
\end{proof}

\begin{proposition} \label{varphi_classification}
Assume that $p_F \not\in N_1(k)$, that $p$ splits totally in $F$, and that the class number $h_k$ is relatively prime to $d$.
Assume furthermore either that $d = 1$, or that $d = 2$ and the exponent of $\Gal(k/\Q)$ is $2$.
Then one of the following holds:
\begin{enumerate}[Type 1.]
\setcounter{enumi}{1}
\item We have that $\varphi = \chi^{n_{\lcm} / 2}$,
and $v_2(p-1) + 1 \le v_2(n_{\lcm})$.

\item There exists an imaginary quadratic number field whose Hilbert class field is contained in $k$.
\end{enumerate}
\end{proposition}

\begin{proof}
Take a prime $\mathfrak{p}$ of $k$ above $p_F$.
Since $p_F$ does not lie in $N_1(k)$, the prime $p$ is unramified in $k$.
Applying \cref{nu_above_p} to $\mathfrak{p}$, we obtain an element $\varepsilon$ as in \cref{nu_above_p}.
On the other hand by \cref{varphi_q_equals_to_beta}, for every $\mathfrak{q} \in S$ there exists $\beta_\mathfrak{q} \in FR(N(\mathfrak{q}))$ such that
$\varphi(\mathfrak{q}) = \beta_\mathfrak{q}^{n_{\lcm}} \mod \mathfrak{r}$ for a prime $\mathfrak{r}$ of $F(\beta_\mathfrak{q})$ above $p_F$, and hence $\alpha_\mathfrak{q}^\varepsilon \mod \mathfrak{p} = \beta_\mathfrak{q}^{n_{\lcm} h_k} \mod \mathfrak{r}$.
Thus replacing $\beta_\mathfrak{q}$ by its conjugate over $F$ if necessary, we obtain that $p_F$ divides $N_{k(\beta_\mathfrak{q}) / F} (\alpha_\mathfrak{q}^\varepsilon - \beta_\mathfrak{q}^{n_{\lcm} h_k})$.
Now since we are assuming that $p_F \not\in N_0(k)$, it implies that $\alpha_\mathfrak{q}^\varepsilon = \beta_\mathfrak{q}^{n_{\lcm} h_k}$ for every $\mathfrak{q} \in S$.
It follows that, fixing one $\mathfrak{q}_0 \in S$, by \cref{type_of_varepsilon} the element $\varepsilon$
(which is independent of the choice of $\mathfrak{q}_0$ of course)
is of type 2 or type 3 for $m = n_{\lcm}/2$ with respect to $\mathfrak{q}_0$.

First assume that $\varepsilon$ is of type 2,
i.e., assume that $\varepsilon = (n_{\lcm}/2) \sum_{\sigma \in \Gal(k/\Q)} \sigma$.
Take $\mathfrak{q} \in S$ and for simplicity we write $\beta = \beta_\mathfrak{q}$ and $\alpha = \alpha_\mathfrak{q}$,
and let $\ell$ be the rational prime below $\mathfrak{q}$.
In this case $\beta^{n_{\lcm} h_k} = \alpha^\varepsilon = N_{k/\Q}(\alpha)^{n_{\lcm}/2}$.
The latter element is of course a rational number, and since $n_{\lcm}$ is divisible by $4$, it is positive.
Now $| \beta |_v = \sqrt{\ell}$ for every infinite place $v$.
Combining them together we obtain that the element $\beta^{n_{\lcm} h_k}$ is equal to $\ell^{n_{\lcm} h_k /2}$.
Write $\beta = \zeta \sqrt{-\ell}$ for an $n_{\lcm} h_k$-th root $\zeta$ of $1$.
Now since $\ell$ is unramified in $F$ and since $\ell \nmid n_{\lcm} h_k$, the field $F(\sqrt{-\ell})$ is not contained in $F(\zeta)$,
i.e., $F(\zeta) \subsetneq F(\zeta, \sqrt{-\ell}) = F(\beta, \sqrt{-\ell})$, whose degree over $F$ divides $4$.
Thus $[F(\zeta) : F] \mid 2$, hence $\zeta^{n_{\lcm}} = 1$.
Therefore we have that $\varphi(\mathfrak{q}) = \ell^{n_{\lcm} / 2} \mod p$.
Since $\ell$ splits totally in $k$ we have $N_{k/\Q}(\mathfrak{q}) = \ell$,
and hence $\varphi(\mathfrak{q}) = \chi(\mathfrak{q})^{n_{\lcm}/2}$.
This holds for any $\mathfrak{q} \in S$.
Therefore the fact that $S$ generates the class group $Cl_k$ forces that $\varphi = \chi^{n_{\lcm}/2}$, which is the desired equation.
The latter statement in this case is a consequence of \cref{varphi_above_p_local}.
Note that in this case, i.e., in the case that $\varepsilon$ is of type 2, we do not use any assumptions
about $h_k$, $d$, nor about the Galois group $\Gal(k/\Q)$.

Next assume that $\varepsilon$ is of type 3 (with respect to our fixed $\mathfrak{q}_0 \in S$).
Let $L = F(\beta_{\mathfrak{q}_0})$ and fix a system of representatives of the cosets of $\Gal(k/\Q)/\Gal(k/L)$.
In this case we obtain the subset $I$ as in \cref{type_of_varepsilon}.
Assume that there exists $\rho \in \Gal(k/\Q)$ such that $\rho^{-1} I \Gal(k/L)$ is a subgroup of $\Gal(k/\Q)$.
Note that under our assumptions this does hold by \cref{I_is_a_group}.
This is the only point at which we use the assumptions about $d$ and $\Gal(k/\Q)$.
Write $I \Gal(k/L) = \rho \Gal(k/M)$.
We claim, for every $\mathfrak{q} \in S$, that $L = F(\beta_\mathfrak{q})$.
Let $\mathfrak{q} \in S$ and $L' := F(\beta_\mathfrak{q})$.
Then of course $\varepsilon$ is also of type $3$ with respect to $\mathfrak{q}$.
Taking a system of representatives of the cosets of $\Gal(k/\Q) / \Gal(k/L')$,
we obtain a subset $I'$ of it as in \cref{type_of_varepsilon}.
Since $n_{\lcm} \sum_{\sigma \in I \Gal(k/L)} \sigma = \varepsilon = n_{\lcm} \sum_{\sigma \in I' \Gal(k/L') } \sigma$, we have that $I \Gal(k/L) = I' \Gal(k/L')$, and in particular $I' \Gal(k/L') = \rho \Gal(k/M)$.
Therefore by \cref{L_is_independent} we have that $L' = MF = L$, which shows our claim.
Since we have chosen $\alpha_\mathfrak{q}$ so that $\mathfrak{q}^{h_k} = \alpha_\mathfrak{q} O_k$,
we get the equation of ideals $\alpha_\mathfrak{q}^{\rho^{-1}\varepsilon} O_k = N_{k/M}(\alpha_\mathfrak{q})^{n_{\lcm}} O_k = N_{k/M}(\mathfrak{q})^{n_{\lcm} h_k} O_k$.
Now we have seen that $\alpha_\mathfrak{q}^\varepsilon = \beta_\mathfrak{q}^{n_{\lcm} h_k}$ and that $F(\beta_\mathfrak{q}) = L$ for every $\mathfrak{q} \in S$.
Therefore we obtain the equation of ideals $N_{k/M}(\mathfrak{q}) O_k = \beta_{\mathfrak{q}^{\rho^{-1}}} O_k$,
and hence the equation $N_{k/M}(\mathfrak{q}) O_L = \beta_{\mathfrak{q}^{\rho^{-1}}} O_L$
for each $\mathfrak{q} \in S$.
Since $N_{k/M}(\mathfrak{q})$ is an ideal of $M$ and since $[L : M] = d$,
we have $N_{k/M}(\mathfrak{q})^d = N_{L/M}(N_{k/M}(\mathfrak{q}) O_L)$,
which is equal to $N_{L/M}(\beta_{\mathfrak{q}^{\rho^{-1}}}) O_M$, a principal ideal of $M$.
Therefore, since $S$ generates the class group $Cl_k$,
we have that $N_{k/M}(J_k^d)$ is contained in the group of principal ideals of $M$.
It follows, since we are now assuming that $d$ and $h_k$ are relatively prime to each other,
that $N_{k/M}(J_k)$ is contained in the group of principal ideals of $k$.
Finally since the norm map $Cl_k \to Cl_M$ is trivial, considering the global residue map, we obtain that $k$ contains the Hilbert class field $H_M$ of $M$.
This completes the proof.
\end{proof}

We use the same notations of types in \cref{type_of_varepsilon,varphi_classification} following \cite[Lemma 2, Theorem 1]{Momose}.
Since every QM-abelian variety has potentially good reduction, we do not need to treat "type 1", as noted in \cite[Remark 5.7]{AraiI}.

In the proof of this proposition, we have shown following variant:

\begin{proposition} \label{varphi_classification:variant}
Assume that $p_F \not\in N_1(k)$, that $p$ splits totally in $F$, and that there exists $\mathfrak{q} \in S$ such that $k$ does not contain any elements of $FR(N(\mathfrak{q}))$.
Then we have that $\varphi = \chi^{n_{\lcm} / 2}$,
and $v_2(p-1) + 1 \le v_2(n_{\lcm})$.
\end{proposition}

\begin{proof}
Keep the notations as in the proof of \cref{varphi_classification}.
For $\mathfrak{q} \in S$, if $\varepsilon$ is of type 3 as in \cref{type_of_varepsilon} with respect to $\mathfrak{q}$, then $\beta_\mathfrak{q}$, which is an element of $FR(N(\mathfrak{q}))$, is contained in $k$ by \cref{type_of_varepsilon}.
Therefore taking $\mathfrak{q} \in S$ such that $k$ does not contain any elements of $FR(N(\mathfrak{q}))$, we have that $\varepsilon$ is of type 2 as in \cref{type_of_varepsilon} with respect to $\mathfrak{q}$.
The rest of the proof is done in the one of \cref{varphi_classification}.
\end{proof}

From now we study the case of "type 2".

\begin{lemma} \label{F(beta)=F(sqet{-l})}
Let $n$ be a positive integer divided by $2^{v_2(p-1) - 1}$ and $M$ a positive integer divided by $12$.
Let $M' := M/2^{v_2(M)} 3$, $\ell$ a rational prime not dividing $Mn \Delta$, $f$ an odd positive integer, and $\beta \in FR(\ell^f)$.
Assume that $\ell^{dfM'n} < p/4^d$, and that there exists a prime $\mathfrak{p}_0$ of $F(\beta)$ such that $\beta \mod \mathfrak{p}_0 \in \F_p^*$ and that $\beta^{Mn} \mod \mathfrak{p}_0 = \ell^{ fMn/2} \mod p$.
Then $F(\sqrt{-\ell}) = F(\beta)$.
\end{lemma}

Note that since now $f$ is odd and since $\ell$ is unramified in $F$, by \cref{F(beta)_is_totally_imaginary} the field $F(\beta)$ is a totally imaginary quadratic extension of $F$.
Later we apply this lemma for $n = n_{\lcm}/12$ and for $M = 12$ in the case that $v_2(p-1) + 1 \le v_2(n_{\lcm})$.

\begin{proof}
Let $\mathfrak{p}_0$ be a prime of $F(\beta)$ as in the statement.
For $t := v_2(p-1) + 1$, define
\begin{equation*}
\gamma := \beta \ell^{-f \frac{p + 2^{t-1} - 1}{2^t} } \mod \mathfrak{p}_0 \in \F_p^*.
\end{equation*}
Then by the assumptions we have that $\gamma^{Mn} = 1$.
Since $2^{t-1}$ divides $6M'n$ by the assumption, it follows moreover that $\gamma^{6M'n} = 1$.
On the other hand, as
\begin{equation*}
2n ( \frac{p + 2^{t-1} - 1}{2^t} - (1 - \frac{p + 2^{t-1} - 1}{2^t}) ) = \frac{n}{2^{t-2}} (p-1),
\end{equation*}
and as $2^{t-2}$ divides $n$ by the assumption, it follows that
\begin{equation*}
\ell^{ 2n \frac{p + 2^{t-1} - 1}{2^t} } \equiv \ell^{ 2n (1 - \frac{p + 2^{t-1} - 1}{2^t}) } \mod p.
\end{equation*}
Combining them together, if we let $\overline{\beta}$ be the conjugate of $\beta$ over $F$, we obtain that
\begin{equation*}
\begin{aligned}
\beta^{ 2M'n } + \overline{\beta}^{ 2M'n } \mod \mathfrak{p}_0 & =
\ell^{ 2fM'n \frac{p + 2^{t-1} - 1}{2^t} } ( \gamma^{ 2M'n } + \gamma^{ -2M'n } ) \\
& = \ell^{ f M'\frac{n}{2^{t-2}} ( \frac{p-1}{2} + 2^{t-2} ) } ( \gamma^{ 2M'n } + \gamma^{ -2M'n } ),
\end{aligned}
\end{equation*}
which is, since $\ell^{(p-1)/2} \equiv \pm 1 \mod p$ and sine $\gamma^{ 2M'n }$ is a third root of $1$, equal to $\pm \ell^{fM'n} \text{ or } \pm 2 \ell^{fM'n} \mod p$, and hence
$(\beta^{ M'n } + \overline{\beta}^{ M'n })^2 \mod \mathfrak{p}_0
= c \ell^{fM'n} \mod p$ for $c = 0,1,3,$ or $4$.

Since $\beta^{ M'n } + \overline{\beta}^{ M'n } \in F$, the number $(\beta^{ M'n } + \overline{\beta}^{ M'n })^2$ is totally nonnegative.
Moreover, since for every infinite place $v$ of $\Q(\beta)$ the absolute value $| \beta |_v$ is equal to $\sqrt{\ell}^f$,
we have that $ | ( \beta^{ M'n } + \overline{\beta}^{ M'n } )^2 |_v \le 4 \ell^{f M'n} $
for every infinite place $v$ of $F$.
It follows, for every infinite place $v$ of $F$, that
\begin{equation*}
| (\beta^{ M'n } + \overline{\beta}^{ M'n })^2 - c \ell^{f M'n} |_v  \le 4 \ell^{f M'n}.
\end{equation*}
Now we are assuming that $\ell^{df M'n} < p/4^d$.
Hence this absolute value is strictly less than $\sqrt[d]{p}$.
Consequently by the 
product formula we obtain that
\begin{equation*}
(\beta^{ M'n } + \overline{\beta}^{ M'n })^2 = c \ell^{f M'n},
\end{equation*}
and hence
$\beta = \sqrt{\ell}^f \zeta$ for a $12M'n$-th root $\zeta$ of $1$.

Finally we show that $\zeta + \zeta^{-1} = 0$, i.e., $\zeta = \pm \sqrt{-1}$.
This yields, since $f$ is odd by the assumption, that $\beta = \pm \sqrt{-\ell}^f$, in particular $F(\beta) = F(\sqrt{-\ell})$, which was what we wanted to show.
The field $F(\zeta)$ is contained in $F(\zeta, \sqrt{\ell}) = F(\beta, \sqrt{\ell})$, whose degree over $F$ divides $4$ since $\beta$ and $\sqrt{\ell}$ are quadratic over $F$.
Since $\ell$ does not divide $Mn\Delta$, it is unramified in $F(\zeta)$.
Thus $F(\zeta)$ must not be equal to $F(\zeta, \sqrt{\ell})$, for, otherwise $F(\zeta)$ contains $\sqrt{\ell}$.
In particular $[F(\zeta) : F]$ divides $2$, and hence $\zeta + \zeta^{-1} \in F$.
Now $\beta + \overline{\beta}$, which lies in $F$, is equal to $\sqrt{\ell}^f (\zeta + \zeta^{-1})$.
Therefore if $\zeta + \zeta^{-1}$ were not zero, then $\sqrt{\ell}$ would lie in $F$, which is a contradiction.
It finishes the proof.
\end{proof}

\begin{lemma} \label{F(sqrt{-l})_splits_B}
Assume that $p_F \not\in N_1(k)$, that $p$ splits totally in $F$, and that $\varphi = \chi^{n_{\lcm}/2}$.
Let $\ell$ be a rational prime not dividing $n_{\lcm} \Delta$ and $\mathfrak{q}$ be a prime of $k$ above $\ell$ whose inertia degree is odd and satisfying
that $N_{k/\Q}(\mathfrak{q})^{dn_{\lcm}/12} < p/4^d$.
Then the field $F(\sqrt{-\ell})$ splits $B/F$.
\end{lemma}

\begin{proof}
Let $f$ be the inertia degree of $\mathfrak{q}$ over $\ell$.
By \cref{varphi_q_equals_to_beta} there exists $\beta \in FR(\ell^f)$ (which is a Weil number of a QM-abelian variety
over $\F_{\ell^f}$ by the proof of \cref{nu_q_equals_to_beta:fin}) and a prime $\mathfrak{p}_0$ of $F(\beta)$ above $p_F$ such that
$\varphi(\Frob_\mathfrak{q}) = \beta^{n_{\lcm}} \mod \mathfrak{p}_0$.
Thus since $\chi(\Frob_\mathfrak{q}) = \ell^f \mod p$, by the assumption that $\varphi = \chi^{n_{\lcm}/2}$, we obtain that $\beta^{n_{\lcm}} \mod \mathfrak{p}_0 = \ell^{fn_{\lcm}/2} \mod p.$
Again by this assumption and by \cref{varphi_above_p_local}, since $p$ is unramified, we have that $v_2(p-1) + 1 \le v_2(n_{\lcm})$.
Therefore by \cref{F(beta)=F(sqet{-l})}, the field $F(\beta)$, which splits $B$ by \cref{End:ch0:CM}, is equal to $F(\sqrt{-\ell})$.
\end{proof}

Now we are ready to prove the theorem about the non-existence of the rational points of $M_0^B(p_F)$.
The result for $d = 1$ is due to \cite{AraiI,AraiII,AraiIII}.

\begin{theorem} \label{M_0^B(p)(k)_is_empty}
Assume the following four conditions:
\begin{enumerate}
\item The class number $h_k$ of $k$ is relatively prime to $d$.
\label{M_0^B(p)(k)_is_empty:con1}

\item The field $k$ does not contain the Hilbert class fields of any imaginary quadratic number fields.
\label{M_0^B(p)(k)_is_empty:con2}

\item There exists a rational prime $\ell$ not dividing $n_{\lcm} \Delta$ whose inertia degree in $k$ is odd such that
the field $F(\sqrt{-\ell})$ does not split $B$.
\label{M_0^B(p)(k)_is_empty:con3}

\item Either $d=1$, or $d = 2$ and the exponent of $\Gal(k/\Q)$ is $2$.
\end{enumerate}
Then there exists a finite set $N_3(k)$ of primes of $F$, which also depends on B, such that
for every $p_F \not\in N_3(k)$ which splits totally over $\Q$, we have that $M^B_0(p_F)(k) = \varnothing$.
\end{theorem}

\begin{proof}
Let $\ell_0$ be the smallest rational prime not dividing $n_{\lcm} \Delta$ whose inertia degree $f_0$ in $k$ is odd satisfying that the field $F(\sqrt{-\ell_0})$ does not split $B$.
Such a prime does exist by the condition \ref{M_0^B(p)(k)_is_empty:con3}.
Let $N_2(k)$ be the set of primes of $F$ dividing rational primes less than $4 \ell_0^{df_0 n_{\lcm}/12}$,
and let $N_3(k)$ be the union of the sets $N_1(k)$, $N_2(k)$, and the set of primes of $F$ dividing $\Delta$.
Note that $N_3(k)$ is a finite set.
We show that this $N_3(k)$ is the desired set.

Let $p_F \not\in N_3(k)$ be a prime which splits totally over $\Q$, $p$ the rational prime below $p_F$,
and let $x \in M_0^B(p_F)(k)$ be a rational point.
Let $\varphi$ be the $M_0$-character defined by $x$.
By \cref{varphi_classification}, by the condition \ref{M_0^B(p)(k)_is_empty:con1},
and by the condition \ref{M_0^B(p)(k)_is_empty:con2},
we obtain that $\varphi$ is of type 2 as in \cref{varphi_classification}.
Therefore by \cref{F(sqrt{-l})_splits_B},
for every rational prime $\ell$ not dividing $n_{\lcm} \Delta$ satisfying that the inertia degree $f$ of $\ell$ in $k$ is odd and
that $4\ell^{d f n_{\lcm}/12} < p$, the field $F(\sqrt{-\ell})$ splits $B$.
However, since $p_F \not\in N_2(k)$, there exists a rational prime $\ell$ which satisfies these conditions but $F(\sqrt{-\ell})$ does not split $B$.
(For example $\ell_0$ is one of such a prime.)
This is a contradiction.
As a result there are no $k$-rational points of $M_0^B(p_F)$.
\end{proof}

Using \cref{varphi_classification:variant} instead of \cref{varphi_classification}, we show the following variant:

\begin{theorem} \label{M_0^B(p)(k)_is_empty:variant}
Assume the following two conditions:
\begin{enumerate}
\item There exists a rational prime $\ell_1$ not dividing $n_{\lcm} h_k$ which splits in $k$ totally such that $k$ does not contain any elements of the finite set $FR(\ell_1)$.\label{M_0^B(p)(k)_is_empty:variant:con1}

\item There exists a rational prime $\ell_2$ not dividing $n_{\lcm} \Delta$ whose inertia degree in $k$ is odd such that the field $F(\sqrt{-\ell_2})$ does not split $B$. 
\label{M_0^B(p)(k)_is_empty:variant:con2}
\end{enumerate}
Then there exists a finite set $N_4(k)$ of primes of $F$, which also depends on B, such that
for every $p_F \not\in N_4(k)$ which splits totally over $\Q$, we have that $M^B_0(p_F)(k) = \varnothing$.
\end{theorem}

\begin{proof}
Only what we need to do is, for a prime $\ell$ satisfying the condition \ref{M_0^B(p)(k)_is_empty:variant:con1}, to enlarge our fixed finite set $S$ so that every prime $\mathfrak{q}$ of $k$ over $\ell$ is contained in it.
The rest of the proof goes through as in the proof of \cref{M_0^B(p)(k)_is_empty}, except that we use \cref{varphi_classification:variant} instead of \cref{varphi_classification}.
\end{proof}

Finally in this paper we give examples satisfying the condition \ref{M_0^B(p)(k)_is_empty:variant:con1} of \cref{M_0^B(p)(k)_is_empty:variant}, and give an easy-to-understand sufficient condition for the condition \ref{M_0^B(p)(k)_is_empty:variant:con2} of \cref{M_0^B(p)(k)_is_empty:variant}.

\begin{example}
Let $F = \Q(\sqrt{2})$, $k = \Q(\sqrt{2}, \sqrt{-17})$, and $\ell = 7$.
Then $n_{\lcm} = 24, h_k = 8$, $\ell$ splits totally in $k$, and we can show that $k$ does not contain any elements of $FR(\ell)$.

Secondly, let $F = \Q(\zeta_7)^+$, the maximal totally real subfield of the $7$-th cyclotomic number field, $k = F(\sqrt{-17})$, and $\ell = 13$.
Then $n_{\lcm} = 84, h_k = 36$, $\ell$ splits totally in $k$, and we can show that $k$ does not contain any elements of $FR(\ell)$.

In order to verify them, we use the computer algebra programs Magma \cite{Magma}.
The codes verifying them are available at the author's Github repository
\begin{center}
\url{https://github.com/kmatsuda111/QMAV_gamma0}.
\end{center}
\end{example}

\begin{lemma} \label{a_sufficient_condition}
Let $F$ be a totally real number field, $k$ a Galois number field containing $F$, and $B/F$ be a division quaternion algebra.
Let $\mathfrak{d}_{B/F}$ be the discriminant of $B/F$, i.e., the product of all finite primes of $\operatorname{Ram}(B/F)$.
If $\mathfrak{d}_{B/F} \nmid \mathfrak{d}_{k/F} \mathfrak{D}_{F/\Q}$, then there exist infinitely many rational primes $\ell$ which split totally in $k$ such that $F(\sqrt{-\ell})$ do not split $B$.
\end{lemma}

\begin{proof}
Let $v$ be a finite place of $F$ dividing $\mathfrak{d}_{B/F}$ but not dividing $\mathfrak{d}_{k/F}\mathfrak{D}_{F/\Q}$ and $p$ be the rational prime below $v$.
Let $p' = p$ if $p$ is odd and $p' = 8$ if $p=2$.
Then since $p$ is unramified in $k$ and is totally ramified in $\Q(\zeta_{p'})$, we have that canonically $\Gal(k(\zeta_{p'}) / \Q) \isom \Gal(k/\Q) \times (\Z/p')^*$, where $\zeta_{p'}$ is a primitive ${p'}$-th root of unity.
By Chebotarev's density theorem, considering the element of $\Gal(k(\zeta_{p'}) / \Q)$ corresponding to $(1,-1)$, there exist infinitely many rational primes $\ell$ which are unramified in $k(\zeta_{p'})$, whose inertia degree in $k$ are $1$, and satisfying that $\ell \equiv -1 \mod {p'}$.
For such a prime $\ell$, the prime $p$ splits in $\Q(\sqrt{-\ell})$, in particular so does $v$ in $F(\sqrt{-\ell})$ since $\sqrt{-\ell} \not\in F$ by the assumption that $F$ is totally real.
Now since we are assuming that $v$ divides $\mathfrak{d}_{B/F}$, if $v$ splits in $F(\sqrt{-\ell})$, then $F(\sqrt{-\ell})$ does not split $B$.
Thus there exist infinitely many primes satisfying the desired conditions.
\end{proof}

Note that as in \cite{AraiRT} we can apply these theorems to a finiteness conjecture on abelian varieties, called the Rasmussen--Tamagawa conjecture \cite[Conjecture 1]{RT}.
Also note that, in \cite{AraiIII}, the author eliminates the possibility of CM points of $M^B_0(p_F)$ by another method in the case of $F = \Q$.

\begin{acknowledgements}
The author would like to thank my supervisor Takeshi Saito for his kind and valuable suggestions.
This work was supported by JSPS KAKENHI Grant Number JP23KJ0568.
\end{acknowledgements}

\end{document}